\documentclass[12pt,a4paper]{article}

\usepackage{graphicx}
\usepackage{subfig}
\usepackage{amsmath,amsthm,amssymb}
\usepackage{esint}
\usepackage{mathrsfs}
\usepackage{appendix}

\usepackage{hyperref}

\usepackage[left=2.2cm,right=2.2cm,top=3cm,bottom=3.5cm]{geometry}

\usepackage{color}
\usepackage{array}
\usepackage{url}

\usepackage{float}

\newtheorem{theorem}{Theorem}[section]
\newtheorem*{theorem*}{Theorem}
\newtheorem{lemma}[theorem]{Lemma}
\newtheorem{corollary}[theorem]{Corollary}
\newtheorem{proposition}[theorem]{Proposition} 

\theoremstyle{definition}
\newtheorem{definition}[theorem]{Definition}

\DeclareMathOperator{\dist}{dist}

\theoremstyle{remark}
\newtheorem{remark}[theorem]{Remark}
\usepackage{authblk}

\usepackage{mathtools}
\usepackage{amssymb}
\usepackage{appendix}

\numberwithin{equation}{section}

\usepackage[numbers,sort&compress]{natbib}

\newcommand{\vc}[1]{{\bf #1}}
\newcommand{\vu}{\vc{u}}

\newcommand{\dx}{\,{\rm d} x}
\newcommand{\dy}{\,{\rm d} y}
\newcommand{\dz}{\,{\rm d} z}
\newcommand{\dt}{\,{\rm d} t}
\newcommand{\dtau}{\,{\rm d} \tau}

\newcommand{\R}{\mathbb{R}} 
\newcommand{\N}{\mathbb{N}}
 
\newcommand{\T}{\mathbb{T}} 
\newcommand{\far}{{\rm far}}
\newcommand{\osc}{{\rm osc}}
\newcommand{\lin}{{\rm lin}}

\newcommand{\cor}{{\rm cor}}
\newcommand{\supp}{{\rm supp}\,}

\usepackage{scalerel}
\usepackage{stackengine,wasysym}

\newcommand\reallywidetilde[1]{\ThisStyle{%
  \setbox0=\hbox{$\SavedStyle#1$}%
  \stackengine{-.1\LMpt}{$\SavedStyle#1$}{%
    \stretchto{\scaleto{\SavedStyle\mkern.2mu\AC}{.5150\wd0}}{.6\ht0}%
  }{O}{c}{F}{T}{S}%
}}

\begin{document}
	
\title{Nonuniqueness of generalised weak solutions to the primitive and Prandtl equations}

\author[]{Daniel W. Boutros\footnote{Department of Applied Mathematics and Theoretical Physics, University of Cambridge, Cambridge CB3 0WA, UK. Email: \textsf{dwb42@cam.ac.uk}}\space,  Simon Markfelder\footnote{Institute of Mathematics, University of W{\"u}rzburg, Emil-Fischer-Str. 40, 97074 W{\"u}rzburg, Germany. Email: \textsf{simon.markfelder@gmail.com}} \space and Edriss S. Titi\footnote{Department of Mathematics, Texas A\&M University, College Station, TX 77843-3368, USA; Department of Applied Mathematics and Theoretical Physics, University of Cambridge, Cambridge CB3 0WA, UK; also Department of Computer Science and Applied Mathematics, Weizmann Institute of Science, Rehovot 76100, Israel. Emails: \textsf{titi@math.tamu.edu} \; \textsf{Edriss.Titi@maths.cam.ac.uk} }} 
	
\date{March 21, 2024} 
	
\maketitle
	
\begin{abstract} 
We develop a convex integration scheme for constructing nonunique weak solutions to the hydrostatic Euler equations (also known as the inviscid primitive equations of oceanic and atmospheric dynamics) in both two and three dimensions. We also develop such a scheme for the construction of nonunique weak solutions to the three-dimensional viscous primitive equations, as well as the two-dimensional Prandtl equations. 

While in [D.W. Boutros, S. Markfelder and E.S. Titi, Calc. Var. Partial Differential Equations, 62 (2023), 219] the classical notion of weak solution to the hydrostatic Euler equations was generalised, we introduce here a further generalisation. For such generalised weak solutions we show the existence and nonuniqueness for a large class of initial data. Moreover, we construct infinitely many examples of generalised weak solutions which do not conserve energy. The barotropic and baroclinic modes of solutions to the hydrostatic Euler equations (which are the average and the fluctuation of the horizontal velocity in the $z$-coordinate, respectively) that are constructed have different regularities. 
\end{abstract}


\vspace{0.1cm} \noindent \textbf{Mathematics Subject Classification:} 76B03 (primary), 35Q35, 35D30, 35A01, 35A02, 76D03, 42B35, 42B37  (secondary) 

\newpage

\tableofcontents

\section{Introduction} \label{sec:intro}

\subsection{Problems considered in this paper and context} 

In this work, we consider the following general equation (with $(d+1)$-dimensional spatial domain for $d=1,2$)
\begingroup
\allowdisplaybreaks
\begin{align} 
    \partial_t u - \nu_h^* \Delta_h u - \nu_v^* \partial_{zz} u + u\cdot \nabla_h u + w \partial_z u + \nabla_h p = 0, \label{generalequation-u} \\
    \partial_z p = 0, \label{generalequation-p} \\
    \nabla_h \cdot u + \partial_z w = 0, \label{generalequation-div} 
\end{align}
\endgroup
where the horizontal velocity field $u : \mathbb{T}^{d+1} \times (0,T) \rightarrow \mathbb{R}^d$, the vertical velocity field $w : \mathbb{T}^{d+1} \times (0,T) \rightarrow \mathbb{R}$, and the pressure $p : \mathbb{T}^{d+1} \rightarrow \mathbb{R}$ are unknown, and the horizontal and vertical viscosity parameters $\nu_h^*, \nu_v^* \geq 0$ are given constants. The $d$-dimensional horizontal gradient is denoted by $\nabla_h$ and the $d$-dimensional horizontal Laplacian by $\Delta_h$. Since the pressure $p$ is only determined up to an additive constant, we may require that $p$ is mean-free. 

In this paper we are interested in the following cases:
\begin{itemize} 
    \item Taking $d = 2$ and $\nu_h^* = \nu_v^* = 0$ gives the three-dimensional hydrostatic Euler equations of an incompressible fluid (also known as the inviscid primitive equations of oceanic and atmospheric dynamics). In this paper, the terms inviscid primitive equations and hydrostatic Euler equations will be used interchangeably. 
    \item Taking $d = 2$ and $\nu_h^* , \nu_v^* > 0$ leads to the three-dimensional viscous primitive equations. We remark that the cases with anistropic viscosities ($\nu_h^* > 0$ and $\nu_v^* = 0$, or $\nu_h^* = 0$ and $\nu_v^* > 0$) have also been studied.
    \item Taking $d = 1$ and $\nu_h^* = \nu_v^* = 0$ yields the two-dimensional inviscid primitive equations (or hydrostatic Euler equations).
    \item Taking $d = 1$, $\nu_h^* = 0$ and $\nu_v^* > 0$ yields the two-dimensional Prandtl equations.
\end{itemize}

In this paper, we will develop a convex integration scheme for system \eqref{generalequation-u}-\eqref{generalequation-div} for the cases mentioned above. In particular, we will work with a generalised notion of weak solution. While classical weak solutions have sufficient Lebesgue integrability for the nonlinearity to make sense as an $L^1 (\mathbb{T}^{d+1} \times (0,T))$ function, another notion of weak solution was introduced in \cite{boutroshydrostatic} where the nonlinearity is interpreted as a paraproduct. The generalised weak solutions introduced in this paper treat the nonlinearity in an even more general way, see section~\ref{subsubsec:general-weak} below.

In all the cases of \eqref{generalequation-u}-\eqref{generalequation-div} that we are interested in, we will show the existence of such generalised weak solutions (for a dense set of initial data in the relevant spaces). In addition, we will show that such weak solutions are nonunique. 

If $\nu_h^* = \nu_v^* = 0$, we recall that classical spatially analytic solutions of \eqref{generalequation-u}-\eqref{generalequation-div} (see \cite{ghoul,kukavicatemamvicolziane,kukavicatemamvicolziane2}) conserve the energy, i.e. the spatial $L^2 (\mathbb{T}^{d+1})$ norm of $u$. In \cite{boutroshydrostatic} an analogue of Onsager's conjecture was studied for the three-dimensional hydrostatic Euler equations and it was found that there exist several sufficient regularity criteria for weak solutions which guarantee the conservation of energy. In particular, there exist several notions of weak solutions for these equations, each of which have their own version of the analogue of the Onsager conjecture. 

In this work, we will construct generalised weak solutions to these equations, which do not conserve energy and do not satisfy the regularity criteria mentioned above. In other words, in this paper we prove a first result towards the aim of resolving the dissipation part of the analogue of the Onsager conjecture for the inviscid primitive equations (hydrostatic Euler), while the conservation part of the analogue of the Onsager conjecture has been studied in \cite{boutroshydrostatic}, as was mentioned before.

\subsection{Literature overview} 
In this section we will provide an overview of some of the literature that is related to this work. As both the primitive and Prandtl equations as well as the Onsager conjecture have been the subject matter of many works in recent years, this overview is by no means comprehensive and is by necessity incomplete in reviewing all the relevant work.

Onsager's conjecture was originally posed in \cite{onsager} for the incompressible Euler equations. The conjecture states that if a weak solution lies in $L^3 ((0,T); C^{0,\alpha} (\mathbb{T}^3))$ for $\alpha > \frac{1}{3}$ it must conserve energy. If $\alpha < \frac{1}{3}$ energy might not be conserved. 

In \cite{eyink} a proof of a slightly weaker result than the first half of the conjecture was given. A full proof of the first half was then given in \cite{constantin}. In \cite{duchon} a different proof was presented, which relied on an equation of local energy balance and a defect measure. In \cite{titi2018,titi2019} (see also \cite{bardosholder}) the problem was considered in the presence of physical boundaries and the first half of the conjecture was proved in this case. 

The existence of non-energy conserving solutions of the Euler equations of an incompressible fluid was first shown in \cite{scheffer,shnirelman}. To prove the existence of dissipative weak solutions of the Euler equations (and to prove the second half of Onsager's conjecture), techniques from convex integration were used. They were introduced for the first time in the context of incompressible fluid mechanics in \cite{lellisinclusion,lellisadmissibility}. 

The second half of the conjecture was then proven in \cite{isettproof}, after gradual success in the papers \cite{lelliscontinuous,buckmaster} (and see references therein). The proof in \cite{isettproof} relied on the Mikado flows that were developed in \cite{daneri}. In the work \cite{buckmasteradmissible} dissipative H\"older continuous solutions of the Euler equations up to $\frac{1}{3}$ were constructed. 

Subsequently, an intermittent version of convex integration was developed. This was first used in \cite{buckmasternonuniqueness} to prove the nonuniqueness of very weak (not Leray-Hopf) solutions to the Navier-Stokes equations. In \cite{colombo} this result was extended to show the existence of nonunique weak solutions with a bound on the singular set. In \cite{buckmasternonconservative,novack} an intermittent scheme was constructed to prove the existence of non-energy conserving weak solutions of the Euler equations with Sobolev regularity. In \cite{luononuniqueness} the method of \cite{buckmasternonuniqueness} was generalised to the hyperviscous Navier-Stokes equations to show the sharpness of the Lions exponent.

After the works \cite{modena,modena2019,modena2020} where a spatially intermittent convex integration scheme was developed for the transport equation, temporal intermittency was introduced to the scheme in \cite{cheskidovluo1,cheskidovluo2} to prove the nonuniqueness of weak solutions to the transport equation. This scheme was then adapted to the Navier-Stokes equations in \cite{cheskidovluo3} to prove the sharpness of the Prodi-Serrin criteria, and in \cite{cheskidovluo4} to show that $L^2$ is the critical space for nonuniqueness for the 2D Navier-Stokes equations. 

The primitive equations of oceanic and atmospheric dynamics were introduced in \cite{richardson}. They were studied mathematically for the first time in \cite{lionsjl1,lionsjl2,lionsjl3}, in which the global existence of weak solutions was proved. The short time existence of strong solutions was then obtained in \cite{guillen}. The global well-posedness of the viscous primitive equations was proven in \cite{cao}, see also \cite{kobelkov}. In \cite{kukavica1,kukavica2} different boundary conditions were considered, and in \cite{hieber} global well-posedness was established using a semigroup method. 

Subsequently, the cases with only horizontal viscosity (as well as only horizontal diffusivity) were studied in \cite{cao1,cao2,CaoLiTiti}. The case with only vertical diffusivity and full viscosity was looked at in \cite{caolocal,caoglobal}. The case with only horizontal diffusivity and full viscosity was investigated in \cite{caohorizontal}. The small aspect ratio was rigorously justified in a weak sense in \cite{azerad} (see also \cite{bresch}). It was subsequently proven in a strong sense with full viscosity in \cite{LiTiti} and with only horizontal viscosity in \cite{liprimitive} with error estimates in terms of the small aspect ratio . 

The case with only vertical viscosity was studied in \cite{renardy}, in which linear ill-posedness was proven. The ill-posedness can be counteracted by adding a linear damping term, see \cite{cao3} for more details. By considering the case of initial data with Gevrey regularity with certain convexity conditions, in \cite{gerardvaret} local well-posedness was established. The local well-posedness for analytic data was proven in \cite{kukavicatemamvicolziane,kukavicatemamvicolziane2} (without rotation) and \cite{ghoul} (with rotation). By considering small data which are analytic in the horizontal variables, the paper \cite{paicu} established global well-posedness for the case without rotation and Dirichlet boundary conditions. Finally, \cite{lineffect} considered the case of impermeable and stress-free boundary conditions.

The linear and nonlinear ill-posedness of the inviscid primitive equations in all Sobolev spaces was proven in \cite{renardy,han}. The ill-posedness results in Sobolev spaces suggest that the natural space for showing local well-posedness of the inviscid primitive equations is the space of analytic functions, which was proved in \cite{ghoul,kukavicatemamvicolziane,kukavicatemamvicolziane2}. In \cite{ghoul} the role of fast rotation in prolonging the life-span of solutions was investigated. 

In \cite{cao4} it was shown that smooth solutions of the inviscid primitive equations can form a singularity in finite time, see also \cite{wong}. In \cite{chiodaroli} the existence and nonuniqueness of weak solutions with $L^\infty$ data was proven. In \cite{boutroshydrostatic} several sufficient criteria for energy conservation were proven. In the inviscid setting there have also been works studying the case of initial data with a monotonicity assumption, see \cite{brenierhomogeneous,kukavicalocal,masmoudi}.

In addition to \cite{chiodaroli}, there have been several papers in which convex integration schemes are developed for geophysical models. In particular, there has been a sequence of works \cite{tao,tao2017,tao2018,luoboussinesq} in which nonunique weak solutions for the Boussinesq equations were constructed in a variety of settings. To be precise, in \cite{tao2018} the effects of vertical viscosity were included, while in \cite{luoboussinesq} the Boussinesq system with full diffusion for the temperature was studied.  Moreover, in \cite{novackquasigeostrophic} the nonuniqueness of weak solutions was established for the quasi-geostrophic equations.

The Prandtl equations for the boundary layer were derived by Prandtl in \cite{prandtl}. In \cite{oleinik1,oleinik2} the local well-posedness of the equations was shown under a monotonicity assumption. In \cite{sammartino} the local well-posedness for analytic data was proven, while in \cite{eprandtl} the blow-up of solutions for certain classes of $C^\infty$ data was proven. Further local well-posedness results were proved in \cite{kukavicaanalytic,kukavicalocal,paicu,ignatova,paicuglobal}. In \cite{xin,xin2022,liu2016} the global existence of weak solutions was established under the assumption that the pressure is favourable. In \cite{grenierprandtl} it was shown that the equations are nonlinearly unstable.

The linear ill-posedness of the Prandtl equations in all Sobolev spaces was shown in \cite{gerardprandtl} (for further work see \cite{liuprandtl} and references therein).
 In the three-dimensional case a convex integration scheme was developed in \cite{luo}. The analytic local well-posedness has been improved to Gevrey function spaces, see \cite{ligevrey} and references therein.

\subsection{Definitions and main results}

\subsubsection{Baroclinic and barotropic modes} \label{barotropicsection}
Now we introduce the notion of barotropic and baroclinic modes, which is an important decomposition of the solutions which has been explored extensively in the investigation of the primitive equations. In the construction of the convex integration scheme for the primitive equations we will not use this decomposition explicitly. However, it is an important idea underlying the scheme.

We will illustrate this concept for the equations in the inviscid case, the viscous case is similar and can be found in \cite{cao}. The 3D inviscid primitive equations are given by
\begingroup
\allowdisplaybreaks
\begin{align} 
    \partial_t u + u\cdot \nabla_h u + w \partial_z u + \nabla_h p = 0, \label{eq:hyd-euler-u} \\
    \partial_z p = 0, \label{eq:hyd-euler-p} \\
    \nabla_h \cdot u + \partial_z w = 0, \label{eq:hyd-euler-div} 
\end{align}
\endgroup
where $u : \mathbb{T}^3 \times (0,T) \rightarrow \mathbb{R}^2$ is the horizontal velocity field, $w : \mathbb{T}^3 \times (0,T) \rightarrow \mathbb{R}$ the vertical velocity field and $p : \mathbb{T}^3 \times (0,T) \rightarrow \mathbb{R}$ the pressure.

The barotropic mode $\overline{u}$ of a velocity field $u$ is defined as follows
\begin{equation}
\overline{u} (x_1,x_2,t) \coloneqq \int_{\mathbb{T}} u(x_1,x_2,z,t) \dz.
\end{equation}
The baroclinic mode $\widetilde{u}$ is defined as the fluctuation
\begin{equation}
\widetilde{u} \coloneqq u - \overline{u}.
\end{equation}
The primitive equations \eqref{eq:hyd-euler-u}-\eqref{eq:hyd-euler-div} can then be written formally as a coupled system of evolution equations for the barotropic and baroclinic modes $\overline{u}$ and $\widetilde{u}$, which are
\begin{align} 
&\partial_t \overline{u} + (\overline{u} \cdot \nabla_h) \overline{u} + \overline{\big[ (\widetilde{u} \cdot \nabla_h) \widetilde{u} + (\nabla_h \cdot \widetilde{u}) \widetilde{u} \big]} + \nabla_h p = 0, \\
&\partial_t \widetilde{u} + (\widetilde{u} \cdot \nabla_h ) \widetilde{u} + w \partial_z \widetilde{u} + (\widetilde{u} \cdot \nabla_h) \overline{u} + (\overline{u} \cdot \nabla_h) \widetilde{u} - \overline{\big[ (\widetilde{u} \cdot \nabla_h) \widetilde{u} + (\nabla_h \cdot \widetilde{u}) \widetilde{u} \big]} = 0.
\end{align}
Moreover, we have the following incompressibility conditions
\begin{equation}
\nabla_h \cdot \overline{u} = \nabla_h \cdot \widetilde{u} + \partial_z w = 0,
\end{equation}
which formally follow from equation \eqref{eq:hyd-euler-div} and the periodicity of the functions.

In the convex integration scheme, we will add separate barotropic and baroclinic perturbations. This leads to different regularities of the barotropic and baroclinic modes of the solution and allows us to control different parts of the error. 

The following estimates on the baroclinic and barotropic modes are standard
\begin{align*}
    \lVert \overline{u} \rVert_{L^p} \lesssim \lVert u \rVert_{L^p}, \qquad
    \lVert \widetilde{u} \rVert_{L^p} \lesssim \lVert u \rVert_{L^p}.
\end{align*}

\subsubsection{Notation} \label{subsubsec:notation}

Throughout the paper we will use the following notation.
\begin{itemize}
    \item The components of the spatial variable are given by $x=(x_1,z)$ if $d=1$, and $x=(x_1,x_2,z)$ if $d=2$. For $d=1$, $x_1$ represents the horizontal direction, for $d=2$ the horizontal position is given by $(x_1,x_2)$. In both cases, $z$ is the vertical direction.
    
    \item The horizontal velocity field is called $u$, the vertical velocity is denoted by $w$ and the full velocity by $\mathbf{u}=(u,w)$. They are $d$-, $1$- and $(d+1)$-dimensional, respectively.
    
    \item We use the symbol $\nabla_h$ for the horizontal gradient (which equals $\partial_{x_1}$ if $d=1$), and $\nabla$ for the full ($(d+1)$-dimensional) gradient. 

    \item For an integrability parameter $1\leq p\leq \infty$, the H\"older conjugate is denoted by $p'$, i.e. $\frac{1}{p} + \frac{1}{p'} = 1$.

    \item Let $1< p\leq \infty$. In section~\ref{sec:intro}, $p-$ denotes any parameter $1\leq p- <p$. In the other sections we have to be a bit more precise. In particular there is a need to quantify the `$-$' in $p-$. More precisely there will be a $\delta>0$ and we set $p-\coloneqq \frac{1}{\frac{1}{p}+\delta}$. Here we tacitly assume that $\delta$ is sufficiently small, such that $p-\geq 1$.

    \item For $1\leq p,q\leq \infty$ and $s\in \mathbb{R}$ the Besov space $B_{p,q}^s(\mathbb{T}^3)$ is defined in appendix~\ref{subsec:ap-littlewood-besov}. Let us emphasise here that $B_{2,2}^s(\mathbb{T}^3)=H^s(\mathbb{T}^3)$, see Remark~\ref{rem:besov-sobolev}.

    \item Throughout this paper, we will omit the domain of a space-time norm if it is $\mathbb{T}^{d+1}\times [0,T]$, e.g. we write $\lVert \cdot \rVert_{L^p(H^s)} = \lVert \cdot \rVert_{L^p((0,T);H^s(\mathbb{T}^{d+1}))} $.

    \item In view of section~\ref{barotropicsection} we define the barotropic and baroclinic part of any quantity $a=a(x)$ by 
    \begin{equation*}
        \overline{a} = \int_{\mathbb{T}} a(x) \dz, \qquad \widetilde{a} = a - \overline{a}.
    \end{equation*}
\end{itemize}

\subsubsection{Generalised weak solutions} \label{subsubsec:general-weak}
In \cite{boutroshydrostatic} two new types of weak solutions to the hydrostatic Euler equations \eqref{eq:hyd-euler-u}-\eqref{eq:hyd-euler-div} were introduced. In the present paper we will consider a slightly different notion of weak solution, which we will refer to as a \textit{generalised weak solution}. This notion of solution is inspired by the notion of a type III weak solution, as introduced in \cite{boutroshydrostatic}.

Before we state the results for the different cases of the system \eqref{generalequation-u}-\eqref{generalequation-div}, we will be more specific regarding the notion of weak solution used in this paper. The weak solutions of \eqref{generalequation-u}-\eqref{generalequation-div} we consider are defined as follows: We assume that $u \in L^2 (\mathbb{T}^3 \times (0,T))$, $w \in \mathcal{D}' ( \mathbb{T}^3 \times (0,T) )$ and $u w \in L^1 ((0,T); B^{-s}_{1,\infty} (\mathbb{T}^3))$ for some suitably large $s \in \mathbb{R}$. System \eqref{generalequation-u}-\eqref{generalequation-div} must then be satisfied in the sense of distributions, where the vertical advection term
\begin{equation*}
\int_0^T \langle u w , \partial_z \phi \rangle_{B^{-s}_{1,\infty} \times B^s_{\infty,1}} \dt,
\end{equation*}
is interpreted as a duality bracket between the term $u w$ and the test function $\phi \in \mathcal{D} ( \mathbb{T}^3 \times (0,T))$. 

If $u$ and $w$ happen to have sufficient regularity, for example when $u \in L^2 ((0,T); H^{s + \delta} (\mathbb{T}^3))$ and $w \in L^2 ((0,T); H^{-s} (\mathbb{T}^3))$ (for some small $\delta > 0$), then by applying the paradifferential calculus (see appendix \ref{paradifferentialappendix}) we know that $u w \in L^1 ((0,T); B^{-s}_{1,\infty} (\mathbb{T}^3))$ . This is a stronger notion of solution compared to the notion of a generalised weak solution that we introduced above, as $u$ is required to have (positive) Sobolev regularity and $w$ has to possess some regularity (i.e., it is more than just a distribution). 

The reason we introduce these generalised weak solutions to the system \eqref{generalequation-u}-\eqref{generalequation-div} is that the velocity field does not have isotropic regularity. In particular, the vertical part of the advection term can be formally written as follows
\begin{equation} \label{verticaladvection}
\partial_z (w u) = - \partial_z \bigg[ \bigg( \int_0^z \nabla_h \cdot u dz' \bigg)  u \bigg].
\end{equation}
The only a priori regularity bound for $u$ in the inviscid case is that $u \in L^\infty ((0,T); L^2 (\mathbb{T}^3))$ (i.e. $u$ lies in the energy space). Therefore by relation \eqref{verticaladvection} it is natural to consider weak solutions for the system \eqref{generalequation-u}-\eqref{generalequation-div} such that $w$ lies in a negative Sobolev space with respect to the spatial variables. 

As was already mentioned, solutions of such form have been introduced in \cite{boutroshydrostatic} (by applying paradifferential calculus) but their existence was left open. One of the aims of the present work is to prove the existence of solutions of this type for the primitive and Prandtl equations. Generally speaking, this approach provides a general way of defining weak solutions for systems with a loss of derivative (for the system \eqref{generalequation-u}-\eqref{generalequation-div} the loss of derivative is in the horizontal directions). 

In the next few subsections, we will give precise definitions of the notion of weak solution we will use, and we will state the theorems we will prove for the different cases of the system \eqref{generalequation-u}-\eqref{generalequation-div}. But generally speaking, we will split the nonlinearity $u w$ into the barotropic-vertical and baroclinic-vertical interactions, i.e., the terms $\overline{u} w$ and $\widetilde{u} w$. 

The baroclinic mode $\overline{u}$ of the constructed solutions will have sufficient regularity such that $\overline{u} w$ can be interpreted as a paraproduct. The terms $\widetilde{u}$ and $w$ do not have sufficient regularity to apply the paradifferential calculus. However, as part of the convex integration scheme we will obtain separate estimates on $\widetilde{u} w$ in order to show that it lies in $ L^1 ((0,T); B^{-s}_{1,\infty} (\mathbb{T}^3))$ for some suitable $s$. Therefore the weak solutions we obtain are partly `generalised' (as for the baroclinic-vertical part of the nonlinearity) and partly `paradifferential' (for the barotropic-vertical part of the nonlinearity).

\subsubsection{Results for the 3D inviscid primitive equations}

We first introduce the notion of weak solution for the 3D inviscid primitive equations \eqref{eq:hyd-euler-u}-\eqref{eq:hyd-euler-div}. 

\begin{definition} \label{weaksolution3Dinviscid}
A triple $(u,w,p)$ is called a weak solution of the hydrostatic Euler equations \eqref{eq:hyd-euler-u}-\eqref{eq:hyd-euler-div} if $u \in L^2 (\mathbb{T}^3 \times (0,T))$, $w \in \mathcal{D}' (\mathbb{T}^3 \times (0,T) )$ and $p \in L^1 (\mathbb{T}^3 \times (0,T))$ such that $u w \in L^1 ((0,T); B^{-s}_{1,\infty} (\mathbb{T}^3))$ (where $s>0$ is referred to as the regularity parameter) and the equations are satisfied in the sense of distributions, i.e.
\begingroup
\allowdisplaybreaks
\begin{align}
    \int_0^T \int_{\mathbb{T}^3} u \cdot\partial_t \phi_1 \dx\dt + \int_0^T \int_{\mathbb{T}^3} u \otimes u : \nabla_h \phi_1 \dx \dt + \qquad\quad & \notag \\
    + \int_0^T \langle uw , \partial_z \phi_1 \rangle_{B^{-s}_{1,\infty} \times B^s_{\infty,1}} \dt + \int_0^T \int_{\mathbb{T}^3} p \nabla_h \cdot \phi_1 \dx\dt &= 0, \label{eq:3D-inv-weak-u}  \\ 
    \int_0^T \int_{\mathbb{T}^3} p \partial_z \phi_2 \dx\dt &= 0, \label{eq:3D-inv-weak-p}\\
    \int_0^T \langle \textbf{u} , \nabla \phi_3 \rangle \dt &= 0, \label{eq:3D-inv-weak-div}
\end{align}
\endgroup
for all test functions $\phi_1, \phi_2$ and $\phi_3$ in $\mathcal{D} (\mathbb{T}^3 \times (0,T))$.
\end{definition}

\begin{remark} 
We emphasise that this definition of weak solutions to \eqref{eq:hyd-euler-u}-\eqref{eq:hyd-euler-div} is more general than the notion of weak solution introduced in \cite{boutroshydrostatic}. While in \cite{boutroshydrostatic} the velocity field of a weak solution has sufficient regularity to automatically guarantee that $u w \in L^1 ((0,T); B^{-s}_{1,\infty} (\mathbb{T}^3))$ (by using the paradifferential calculus), in Definition \ref{weaksolution3Dinviscid} we do not have sufficient (separate) regularity requirements on $u$ and $w$ such that the product $u w$ is well-defined. Hence $u w \in L^1 ((0,T); B^{-s}_{1,\infty} (\mathbb{T}^3))$ is a separate independent requirement of Definition \ref{weaksolution3Dinviscid}. 
\end{remark}

\begin{remark}
It should be noted that for a general $u \in L^2 (\mathbb{T}^3 \times (0,T))$ and $w \in \mathcal{D}' (\mathbb{T}^3 \times (0,T) )$ the product $u w$ need not be well-defined. Hence the meaning of the requirement from Definition \ref{weaksolution3Dinviscid} that $u w \in L^1 ((0,T); B^{-s}_{1,\infty} (\mathbb{T}^3))$ could be unclear in such cases. Therefore we specify what we mean with $u w$ in this general setting. \\
Firstly, both $u$ and $w$ can be expressed as a Fourier series. Therefore one can formally define $u w$ as a Fourier series with coefficients which are the convolution of the Fourier coefficients of $u$ and $w$ (in general the convolutions need not to converge). Therefore the requirement $u w \in L^1 ((0,T); B^{-s}_{1,\infty} (\mathbb{T}^3))$ can be understood as the condition that the convolutions of the Fourier coefficients of $u $ and $w$ converge and also that subsequently the Fourier series is bounded in this Besov norm. 
\end{remark}

In this paper we will prove the following result.

\begin{theorem} \label{mainresult}
Let $T>0$ and suppose there exist smooth solutions of the hydrostatic Euler equations \eqref{eq:hyd-euler-u}-\eqref{eq:hyd-euler-div} $(u_1,w_1,p_1)$ on $[0,T/2]$ and $(u_2,w_2,p_2)$ on $[T/2,T]$. Moreover, let $1 \leq q_1, q_2, q_3 \leq \infty$ and\footnote{Note that $s_2$ does not appear in this paper.} $0 < s_1, s_3$ be parameters satisfying
\begin{align} \label{3dconstraints}
    q_2 &> 2, \quad q_3 \leq q_1, \quad s_1 > s_3, \quad \frac{2}{q_1} > s_1 + 1.
\end{align}
Then there exists a weak solution $(u,w,p)$ in the sense of Definition \ref{weaksolution3Dinviscid} with regularity parameter $s=1$ and with the following properties: 
\begin{enumerate}
    \item The solution satisfies that
    \begin{equation} \label{solutionagreement}
            (u,w,p)(\cdot,t) = \left\{ \begin{array}{ll}
            (u_1,w_1,p_1)(\cdot,t) & \text{ if } t\in[0,T/4), \\
            (u_2,w_2,p_2)(\cdot,t) & \text{ if } t\in(3T/4,T].
            \end{array} \right.
    \end{equation}
    \item We have that 
    \begin{align*}
        \overline{u} &\in L^2(\mathbb{T}^3\times (0,T)) \cap L^{q_1} ((0,T); H^{s_1} (\mathbb{T}^3)), \\
        \widetilde{u} &\in L^{q_2-} ((0,T); L^2 (\mathbb{T}^3)) \cap L^{q_3-} ((0,T); H^{s_3} (\mathbb{T}^3)) , \\
        w &\in L^{q_2'} ((0,T); L^2 (\mathbb{T}^3)) \cap L^{q_3'} ((0,T) ; H^{-s_3} (\mathbb{T}^3)),
    \end{align*} 
    where $\overline{u}$ and $\widetilde{u}$ denote the barotropic and baroclinic modes of $u$ respectively.
\end{enumerate}
\end{theorem}

\begin{remark} \label{rem:endpoint}
Alternatively one can construct a weak solution with the properties stated in Theorem~\ref{mainresult} where the only difference is that the endpoint time integrability is attained for $\widetilde{u}$ rather than $w$. In other words 
\begin{align*}
    \widetilde{u} &\in L^{q_2} ((0,T); L^2 (\mathbb{T}^3)) \cap L^{q_3} ((0,T); H^{s_3} (\mathbb{T}^3)) , \\
    w &\in L^{q_2'-} ((0,T); L^2 (\mathbb{T}^3)) \cap L^{q_3'-} ((0,T) ; H^{-s_3} (\mathbb{T}^3)),
\end{align*} 
see Remarks~\ref{rem:uptilde-endpoint} and \ref{rem:uptilde-endpoint-proof} below. To this end however, we have to require that $q_3< q_1$ (strictly) in \eqref{3dconstraints}.
\end{remark}

\begin{remark} \label{rem:u-sobolev}
By proceeding as in section~\ref{viscoussection} below, we can achieve in addition that $\overline{u},\widetilde{u}\in L^1((0,T);W^{1,1}(\mathbb{T}^3))$. To this end, however, we have to require the constraints \eqref{viscousconstraints} rather than \eqref{3dconstraints}, see also Theorem~\ref{viscousmainthm} and Remark~\ref{rem:viscosities}, below.
\end{remark}

\begin{remark} 
Again we would like to remark that the solutions constructed in Theorem \ref{mainresult} are partially `generalised' (see section~\ref{subsubsec:general-weak}) and partially `paradifferential' as in \cite{boutroshydrostatic}. In particular, they have been inspired by the type III weak solutions that were introduced in \cite{boutroshydrostatic}. \\
More precisely, from the regularities of $\overline{u}$ and $w$ stated in Theorem \ref{mainresult} it follows that $\overline{u} w \in L^1 ((0,T); B^{- 1}_{1,\infty} (\mathbb{T}^3))$ (see the proof of Theorem~\ref{mainresult} in sections \ref{perturbationsection}-\ref{reynoldsestimates} for details). The term $\widetilde{u} w$ is estimated directly in $L^1 ((0,T); B^{-1}_{1,\infty} (\mathbb{T}^3))$ as part of the convex integration scheme, as one cannot obtain the regularity of the product $\widetilde{u} w$ simply from the regularities of $\widetilde{u}$ and $w$ (as they are insufficient to apply the paradifferential calculus directly). \\
The specific form of the perturbations allows for a direct estimate, as was done for example in \cite{cheskidovluo2}. Therefore the interpretation of the term $\overline{u} w$ can be seen as `paradifferential', while the interpretation of the term $\widetilde{u} w$ is in the sense of a `generalised weak solution' (as in Definition \ref{weaksolution3Dinviscid}).
\end{remark}

\begin{remark} 
In addition, we would like to emphasise that in the presence of physical boundaries the primitive equations are often studied with no-normal flow boundary conditions on the top and bottom of the channel, i.e. $w \lvert_{z=0,1} = 0$. 
However, in the convex integration scheme developed in this paper we will work on the three-dimensional torus rather than the channel. Note that solutions in the torus can be understood as solutions in the channel with an in-flow out-flow boundary condition, i.e.
\begin{equation*}
w ( x_1, x_2, 0, t) = w (x_1,x_2, 1, t) = w_B (x_1,x_2,t),
\end{equation*}
for a flow $w_B$. In our case $w_B$ will be constructed as part of the convex integration scheme. In other words we will not solve the boundary value problem for given $w_B$ and in particular, not for the case of the impermeability boundary condition $w_B = 0$. \\
We also remark that the constructed flow $w_B$ belongs to the space $L^{q_2'} ((0,T); L^2 (\mathbb{T}^2)) \cap L^{q_3'} ((0,T) ; H^{-s_3} (\mathbb{T}^2))$, where the parameters $q_2', q_3'$ and $s_3$ are the same as in Theorem \ref{mainresult}. 
\end{remark} 

\begin{remark}
In this paper we will not consider the role of density variations in the Boussinesq approximation model of the primitive equations. If one takes density effects into account and applies the Boussinesq approximation, the full inviscid primitive equations are given by
\begingroup
\allowdisplaybreaks
\begin{align} 
    \partial_t u + u\cdot \nabla_h u + w \partial_z u + \nabla_h p = 0, \label{fullprimeq1} \\
    \nabla_h \cdot u + \partial_z w = 0, \\
    \partial_z p + g \rho = 0, \\
    \rho = \rho (T,S,p), \\
    \partial_t T + u \cdot \nabla_h T + w \partial_z T = 0, \\
    \partial_t S + u \cdot \nabla_h S + w \partial_z S = 0, \label{fullprimeq6}
\end{align}
\endgroup
where $\rho$ is the density, $g$ is the gravitational constant, $T$ is the temperature and $S$ is the salinity. Below we will ignore the salinity effects. \\
We note that for system \eqref{fullprimeq1}-\eqref{fullprimeq6} to be fully determined, an equation of state $\rho (T,p)$ needs to be provided. The well-posedness results for the viscous case in \cite{cao,kukavica2} consider the case of a linear equation of state, i.e. the assumption
\begin{equation*}
\rho = \rho_0 - \alpha (T - T_0),
\end{equation*}
where $\rho_0$ and $T_0$ are the mean density and temperature, respectively. The thermal expansion coefficient is denoted by $\alpha$. In \cite{korn} the global well-posedness of the viscous primitive equations with a nonlinear equation of state was established. \\
If one includes density variations as part of the equations, the nature of the pressure changes. Namely, the pressure can now be split into a surface pressure $p_{\rm s}$ (which is independent of $z$ and in our case can be taken to be equal to the pressure at $z = 0$) and the hydrostatic pressure $p_{\rm hyd}$ (which is the integral with respect to the vertical coordinate of the density)
\begin{equation*}
p (x_1,x_2,z,t) = p_{\rm s} (x_1,x_2,t) + p_{\rm hyd} (x_1,x_2,z,t), \quad p_{\rm hyd} (x_1,x_2,z,t) = \int_0^z \rho (x_1,x_2,z',t) \dz'.
\end{equation*}
The surface pressure solves an elliptic problem, and it has played an important role in the development of numerical schemes to solve the primitive equations, see for example \cite{samelson,pinardi,smith,dukowicz}. \\
In the convex integration scheme developed in this paper, introducing density variations as part of the dynamics leads to an additional linear error in the construction (which is coming from the hydrostatic pressure) which with the current approach can probably not be controlled. We leave the consideration of this very important additional effect to future work.
\end{remark}

Theorem~\ref{mainresult} allows to show the nonuniqueness and existence of solutions which do not conserve energy: 

\begin{corollary} \label{infweaksolcor}
    For any analytic initial data there exist infinitely many global-in-time weak solutions $(u,w,p)$ of the hydrostatic Euler equations \eqref{eq:hyd-euler-u}-\eqref{eq:hyd-euler-div} (in the sense of Definition \ref{weaksolution3Dinviscid} which satisfy the regularity properties of Theorem \ref{mainresult}) and they do not conserve energy. 
\end{corollary}

\begin{proof} 
We take the smooth local-in-time solution for the given choice of analytic data (whose existence can be proven using the methods from \cite{kukavicatemamvicolziane2,ghoul}) as the first solution $(u_1, w_1, p_1)$ on $[0,T/2]$, and the zero solution on $[T/2,T]$ as $(u_2, w_2, p_2)$. Then Theorem \ref{mainresult} yields a weak solution, which we may extend by zero for $t>T$. If the initial data are non-zero, we can conclude that the energy is not conserved as it is positive on $[0,T/4)$ and zero on $(3T/4,\infty)$. 

Another global-in-time weak solution can be constructed similarly with replacing $T$ by $T/4$. This solution has positive energy on $[0,T/16)$ while the energy is zero on $(3T/16,\infty)$. Consequently the two solutions cannot coincide. Repeating this argument leads to infinitely many global-in-time weak solutions with the same initial data, which are smooth and unique for a small initial interval of time, but which do not conserve energy. 

For zero initial data, we observe that Theorem~\ref{mainresult} allows one to `connect' any analytic initial data with any analytic data in finite time. Hence we may connect the zero initial data to arbitrary analytic data with positive energy at $t=\widetilde{T}$. On the time interval $[\widetilde{T},\infty)$ we then proceed as above.
\end{proof}

\begin{remark}
Finally, we should emphasise that we do not claim that the solutions constructed in Theorem \ref{mainresult} belong to $u \in L^\infty ((0,T); L^2 (\mathbb{T}^3))$. Therefore it is not possible to directly compare the results in this paper with the Onsager-type conjectures which were formulated in \cite{boutroshydrostatic}, and there is no admissibility criterion for the constructed weak solutions at this point. We leave the study of the sharpness of the conjectures from \cite{boutroshydrostatic} to future work.
\end{remark}

\subsubsection{Results for the 3D viscous primitive equations} 
We now consider the viscous primitive equations, which are given by
\begingroup
\allowdisplaybreaks
\begin{align} 
    \partial_t u - \nu_h^* \Delta_h u - \nu_v^* \partial_{zz} u + u\cdot \nabla_h u + w\partial_z u + \nabla_h p &= 0, \label{eq:visc-u} \\ 
    \partial_z p &= 0, \label{eq:visc-p} \\
    \nabla_h \cdot u + \partial_z w &= 0, \label{eq:visc-div}
\end{align}
\endgroup
where $\nu_h^*$ and $\nu_v^*$ are the horizontal and vertical viscosities. As before, $u : \mathbb{T}^3 \times (0,T) \rightarrow \mathbb{R}^2$ is the horizontal velocity field, $w : \mathbb{T}^3 \times (0,T) \rightarrow \mathbb{R}$ the vertical velocity and $p : \mathbb{T}^3 \times (0,T) \rightarrow \mathbb{R}$ the pressure. We have the following notion of weak solution for these equations.

\begin{definition} \label{viscousweaksoldef} 
A triple $(u,w,p)$ is called a weak solution of the viscous primitive equations \eqref{eq:visc-u}-\eqref{eq:visc-div} if $u \in L^2 (\mathbb{T}^3 \times (0,T)) \cap L^1 ((0,T); W^{1,1} (\mathbb{T}^3))$, $w \in \mathcal{D}' (\mathbb{T}^3 \times (0,T))$ and $p \in L^1 ( \mathbb{T}^3 \times (0,T))$ such that $u w \in L^1 ((0,T); B^{-s}_{1,\infty} (\mathbb{T}^3))$ (where $s > 0$ is referred to as the regularity parameter) and the equations are satisfied in the sense of distributions, i.e.
\begingroup
\allowdisplaybreaks
\begin{align*}
    \int_0^T \int_{\mathbb{T}^3} u \cdot \partial_t \phi_1 \dx\dt - \nu_h^* \int_0^T \int_{\mathbb{T}^3}\nabla_h u : \nabla_h \phi_1 \dx\dt - \nu_v^* \int_0^T \int_{\mathbb{T}^3} \partial_z u \cdot\partial_z \phi_1 \dx\dt + \qquad \quad & \\
    + \int_0^T \int_{\mathbb{T}^3} u\otimes u : \nabla_h \phi_1 \dx\dt + \int_0^T \langle uw, \partial_z \phi_1 \rangle_{B^{-s}_{1,\infty} \times B^s_{\infty,1}} \dt + \int_0^T \int_{\mathbb{T}^3} p \nabla_h \cdot \phi_1 \dx\dt &= 0, \\
    \int_0^T \int_{\mathbb{T}^3} p \partial_z \phi_2 \dx\dt &= 0, \\
    \int_0^T \langle \vu , \nabla \phi_3 \rangle \dt &= 0,
\end{align*}
\endgroup
for all test functions $\phi_1$, $\phi_2$ and $\phi_3$ in $\mathcal{D} (\mathbb{T}^3 \times (0,T))$. 
\end{definition}

In this paper we will prove the following result.

\begin{theorem} \label{viscousmainthm} 
Let $T>0$ and suppose there exist smooth solutions of the viscous primitive equations \eqref{eq:visc-u}-\eqref{eq:visc-div} $(u_1,w_1,p_1)$ on $[0,T/2]$ and $(u_2,w_2,p_2)$ on $[T/2,T]$. Moreover, let $1 \leq q_1, q_2, q_3 \leq \infty$ and $0 < s_1, s_3$ be parameters satisfying the following relations
\begin{align} \label{viscousconstraints} 
    q_2 &> 2, \quad q_3 < q_1 , \quad s_1 > s_3, \quad \frac{2}{q_1} > s_1 + 1, \quad s_3> \frac{1}{2\left(1-\frac{1}{q_2}\right)} \left(\frac{1}{q_3} - \frac{1}{q_2}\right).
\end{align}
Then there exists a weak solution $(u,w,p)$ in the sense of Definition~\ref{viscousweaksoldef} with regularity parameter $s=1$ and with the following properties:
\begin{enumerate}
    \item The solution satisfies that
    $$
        (u,w,p)(\cdot,t) = \left\{ \begin{array}{ll}
        (u_1,w_1,p_1)(\cdot,t) & \text{ if } t\in[0,T/4), \\
        (u_2,w_2,p_2)(\cdot,t) & \text{ if } t\in(3T/4,T].
        \end{array} \right.
    $$

    \item We have that 
    \begin{align*}
        \overline{u} &\in L^2 (\mathbb{T}^3 \times (0,T)) \cap L^{q_1} ((0,T); H^{s_1} (\mathbb{T}^3)) \cap L^1 ((0,T); W^{1,1} (\mathbb{T}^3)), \\ 
        \widetilde{u} &\in L^{q_2-} ((0,T); L^2 (\mathbb{T}^3)) \cap L^{q_3-} ((0,T); H^{s_3} (\mathbb{T}^3)) \cap L^1 ((0,T); W^{1,1} (\mathbb{T}^3)) ,\\
        w &\in L^{q_2'-} ((0,T); L^2 (\mathbb{T}^3)) \cap L^{q_3'-} ((0,T); H^{-s_3} (\mathbb{T}^3)).
    \end{align*}
\end{enumerate}
\end{theorem}

\begin{remark} \label{rem:endpoint-othercases}
Similar to Theorem~\ref{mainresult} one can even obtain endpoint time integrability for $w$. With the modification described in Remarks~\ref{rem:endpoint}, \ref{rem:uptilde-endpoint} and \ref{rem:uptilde-endpoint-proof} one can alternatively establish endpoint time integrability for $\widetilde{u}$.
\end{remark} 

\begin{remark}
The reader should notice that there exist parameters $1 \leq q_1, q_2, q_3 \leq \infty$ and $0 < s_1, s_3$ satisfying \eqref{viscousconstraints}. Indeed for every $q_3<3/2$, we have 
$$
    \frac{1}{q_3} > \frac{1}{4q_3} + \frac{1}{2}.
$$
Hence there exists $q_1$ with 
$$
    \frac{1}{q_3} > \frac{1}{q_1} > \frac{1}{4q_3} + \frac{1}{2}.
$$
Thus $q_3<q_1$ and for $q_2>2$ sufficiently large, the estimate
$$
    \frac{2}{q_1} -1 > \frac{1}{2\left(1-\frac{1}{q_2}\right)} \left(\frac{1}{q_3} - \frac{1}{q_2}\right)
$$
holds since the right-hand side converges to $\frac{1}{2q_3}$ for $q_2\to\infty$. This allows to choose $s_1$ and $s_3$ such that 
$$
    \frac{2}{q_1} -1 > s_1 > s_3 > \frac{1}{2\left(1-\frac{1}{q_2}\right)} \left(\frac{1}{q_3} - \frac{1}{q_2}\right),
$$
so all constraints in \eqref{viscousconstraints} are satisfied.
\end{remark}

\begin{remark} \label{rem:viscosities} 
We would like to emphasise that Theorem~\ref{viscousmainthm} holds for any choice of viscosities $\nu_h^*,\nu_v^*\in \mathbb{R}$, in particular even for the inviscid case $\nu_h^*=\nu_v^*=0$.
\end{remark} 

\begin{remark} \label{sharpviscousregularity}
In the case of the 2D and 3D viscous primitive equations, it has been shown in \cite{medjo,ju1,petcu} that the so-called $z$-weak solutions of the viscous primitive equations are unique. Such weak solutions possess the regularity
\begin{equation*}
u, \partial_z u \in L^\infty ((0,T); L^2 (\mathbb{T}^3)) \cap L^2 ( (0,T); H^1 (\mathbb{T}^3)).
\end{equation*}
The solutions constructed in Theorem \ref{viscousmainthm} possess the regularity $u, \partial_z u \in L^1 ((0,T); W^{1,1} (\mathbb{T}^3)) \cap L^1 ((0,T); H^{1-} (\mathbb{T}^3))$. However, the convex integration scheme used to prove Theorem \ref{viscousmainthm} is currently unable to obtain solutions such that $u, \partial_z u \in L^{2-} ((0,T); H^1 (\mathbb{T}^3))$. We leave the study of the sharpness of the $z$-weak solution regularity class with respect to the (non)uniqueness problem to future work.
\end{remark}

\begin{remark}
Finally, we emphasise that the solutions constructed in Theorem \ref{viscousmainthm} are not of Leray-Hopf type, as they do not have a finite rate of mean energy dissipation (i.e. the horizontal velocity field does not belong to the space $L^2 ((0,T);H^1(\mathbb{T}^3))$).
\end{remark}

We now obtain the global existence of weak solutions as a corollary.
\begin{corollary} 
    For $\nu_h^*,\nu_v^*>0$ and any initial data $u_0 \in H^{1} (\mathbb{T}^3)$ there exists infinitely many global-in-time weak solutions $(u,w,p)$ of the viscous primitive equations \eqref{eq:visc-u}-\eqref{eq:visc-div} (in the sense of Definition~\ref{viscousweaksoldef}) which satisfy the regularity properties of Theorem~\ref{viscousmainthm}. 
\end{corollary} 

\begin{proof}
The proof works exactly as the proof of Corollary~\ref{infweaksolcor} where the corresponding local (even global) well-posedness result can be achieved by using the methods from \cite{cao}.
\end{proof}

\begin{remark}
The proof of nonuniqueness of global weak solutions works equally well in the three cases of full, horizontal or vertical viscosity, which were studied in the works \cite{cao1,cao2,CaoLiTiti}. Moreover, in the case of full viscosity the result can also be adapted to classes of initial data belonging to different function spaces, by relying on the well-posedness results from \cite{giga}.
\end{remark} 

\subsubsection{Results for the 2D hydrostatic Euler equations}
It is also possible to develop a convex integration scheme for the two-dimensional hydrostatic Euler equations. They are given by
\begingroup
\allowdisplaybreaks
\begin{align}
    \partial_t u + u \partial_{x_1} u + w \partial_z u + \partial_{x_1} p &= 0, \label{2Dhydreuler1} \\
    \partial_z p &= 0, \label{2Dhydreuler2} \\
    \partial_{x_1} u + \partial_z w &=0, \label{2Dhydreuler3}
\end{align}
\endgroup
where $u : \mathbb{T}^2 \times (0,T) \rightarrow \mathbb{R}$ is the horizontal velocity, $w : \mathbb{T}^2 \times (0,T) \rightarrow \mathbb{R}$ is the vertical velocity and $p : \mathbb{T}^2 \times (0,T) \rightarrow \mathbb{R}$ is the pressure. We first state the definition of weak solution to these equations.

\begin{definition} \label{2Dweaksoldef} 
A triple $(u,w,p)$ is called a weak solution of the two-dimensional hydrostatic Euler equations \eqref{2Dhydreuler1}-\eqref{2Dhydreuler3} if $u \in L^2 (\mathbb{T}^2 \times (0,T))$, $w \in \mathcal{D}' (\mathbb{T}^2 \times (0,T))$ and $p \in L^1 (\mathbb{T}^2 \times (0,T))$ such that $u w \in L^1 ( (0,T); B^{-s}_{1,\infty} (\mathbb{T}^2))$ (where $s > 0$ is the regularity parameter) and the equations are satisfied in the sense of distributions, i.e.,
\begin{align*}
    \int_0^T \int_{\mathbb{T}^2} u \partial_t \phi_1 \dx\dt + \int_0^T \int_{\mathbb{T}^2} u^2 \partial_{x_1} \phi_1 \dx \dt + \qquad \quad &\\
    +\int_0^T \langle u w , \partial_z \phi_1 \rangle_{B^{-s}_{1,\infty} \times B^s_{\infty,1}} \dt + \int_0^T \int_{\mathbb{T}^2} p \partial_{x_1} \phi_1 \dx\dt &= 0, \\
    \int_0^T \int_{\mathbb{T}^2} p \partial_z \phi_2 \dx\dt &= 0, \\
    \int_0^T \langle \vu , \nabla \phi_3 \rangle \dt &= 0,
\end{align*}
for all test functions $\phi_1$, $\phi_2$ and $\phi_3$ in $\mathcal{D}(\mathbb{T}^2 \times (0,T))$.
\end{definition}

In particular, we will prove the following theorem.

\begin{theorem} \label{2Dtheorem} 
Let $T>0$ and suppose there exist smooth solutions of the two-dimensional hydrostatic Euler equations \eqref{2Dhydreuler1}-\eqref{2Dhydreuler3} $(u_1,w_1,p_1)$ on $[0,T/2]$ and $(u_2,w_2,p_2)$ on $[T/2,T]$. Moreover, let $1 \leq q_2, q_3 \leq \infty$ and $0 < s_3$ be parameters satisfying\footnote{Note that \eqref{2dconstraints} implies $q_3<q_2$ and $q_2>2$.} 
\begin{equation} \label{2dconstraints} 
    \frac{3}{2} q_2 \left(\frac{1}{q_3} - \frac{1}{q_2}\right)> s_3 > \frac{1}{1-\frac{2}{q_2}} \left(\frac{1}{q_3} - \frac{1}{q_2}\right)>0, \quad 1\geq s_3.
\end{equation} 
Then there exists a weak solution $(u,w,p)$ in the sense of Definition~\ref{2Dweaksoldef} with regularity parameter $s=1$ and with the following properties:
\begin{enumerate}
    \item The solution satisfies that
    $$
        (u,w,p)(\cdot,t) = \left\{ \begin{array}{ll}
        (u_1,w_1,p_1)(\cdot,t) & \text{ if } t\in[0,T/4), \\
        (u_2,w_2,p_2)(\cdot,t) & \text{ if } t\in(3T/4,T].
        \end{array} \right.
    $$

    \item We have that 
    \begin{align*}
        u &\in L^{q_2-} ((0,T); L^2 (\mathbb{T}^2)) \cap L^{q_3-} ((0,T); H^{s_3} (\mathbb{T}^2)), \\
        w &\in L^{q_2'-} ((0,T); L^2 (\mathbb{T}^2)) \cap L^{q_3'-} ((0,T); H^{-s_3} (\mathbb{T}^2)).
    \end{align*}
\end{enumerate}
\end{theorem}

\begin{remark}
It might seem slightly odd to label the parameters by $q_2$, $q_3$ and $s_3$ (rather than $q_1$ etc.). The reason we chose to do so is because it will allow for easy comparisons with the three-dimensional scheme from Theorem \ref{mainresult}. We emphasise that there are no equivalent parameters to $q_1$ and $s_1$ in the two-dimensional version of the scheme.
\end{remark}

\begin{remark} \label{rem:endpoint-2D} 
Remark~\ref{rem:endpoint-othercases} is also true in the context of the two-dimensional hydrostatic Euler equations \eqref{2Dhydreuler1}-\eqref{2Dhydreuler3}, see Remark~\ref{rem:endpoint-2D-proof} below.
\end{remark}

\begin{remark} 
By proceeding as in section~\ref{prandtlsection} we can achieve in addition that $\overline{u},\widetilde{u}\in L^1((0,T);W^{1,1}(\mathbb{T}^3))$. In contrast to the three-dimensional case (cf. Remark~\ref{rem:u-sobolev}) in two dimensions there is no need to require stronger constraints for the parameters, see also Theorem~\ref{prandtltheorem} and Remark~\ref{rem:viscosities-2D}.
\end{remark}

We observe that it is possible to establish a two-dimensional analogue of Corollary \ref{infweaksolcor} using the local well-posedness result from \cite{kukavicatemamvicolziane,kukavicatemamvicolziane2} for analytic data in the channel. This yields existence of infinitely many global weak solutions for suitable initial data.

\begin{remark}
We observe that a suitable modification (essentially by changing the scaling of the Mikado flows and densities in Proposition \ref{prop:mikado-v}) of the proof of Theorem \ref{2Dtheorem} yields the existence of nonunique weak solutions of the two-dimensional hydrostatic Euler equations in the energy space, i.e. solutions such that $u \in L^\infty ((0,T); L^2 (\mathbb{T}^2))$. To the knowledge of the authors, this is the first example of a temporally intermittent convex integration scheme that is able to construct weak solutions lying in the energy space (which in essence is due to the structure of the primitive equations). Moreover, we expect that by using some of the ideas in for example \cite{cheskidovluo4} one can achieve control over the energy profile (but this lies outside the scope of this work).
\end{remark}
\begin{remark}
By taking $q_3 = 3$, $s_3 > \frac{2}{3}$ and $q_2 > \frac{13}{3}$ in Theorem \ref{2Dtheorem} we observe that the solutions constructed have the regularity $u \in L^3 ((0,T); B^{1/3+}_{3,\infty} (\mathbb{T}^2))$. This regularity is sufficient (by an argument similar to \cite{boutroshydrostatic}) to conclude that the horizontal advection term is not responsible for any dissipation or increase of energy. In this setting, the mechanism resulting in non-conservation of energy is `purely hydrostatic', i.e. completely due to the irregularity of the vertical velocity $w$. \\
In fact, by taking $q_3 = 4$, $s_3 > \frac{3}{4}$ and $q_2 > 6$ and using a more refined argument one can show that $u \in L^4 ((0,T); B^{1/2+}_{4,\infty} (\mathbb{T}^2))$ and that $w \in L^{6/5-} ((0,T); L^2 (\mathbb{T}^2))$, by using the estimates from Proposition \ref{prop:mikado-v-2D} and the specific form of the perturbation. We recall from \cite{boutroshydrostatic} that if $u \in L^4 ((0,T); B^{1/2+}_{4,\infty} (\mathbb{T}^2))$ and $w \in L^2 (\mathbb{T}^2 \times (0,T))$ then the energy is conserved. Therefore the result from Theorem \ref{2Dtheorem} would be sharp if we could show that $w \in L^{2-} ((0,T); L^2 (\mathbb{T}^2))$.
\end{remark}

\subsubsection{Results for the two-dimensional Prandtl equations}
Now we turn to studying the two-dimensional Prandtl equations, which are given by
\begingroup
\allowdisplaybreaks
\begin{align}
    \partial_t u - \nu_v^* \partial_{zz} u + u \partial_{x_1} u + w \partial_z u + \partial_{x_1} p &= 0, \label{eq:prandtl-u} \\
    \partial_z p &= 0, \label{eq:prandtl-p} \\
    \partial_{x_1} u + \partial_z w &=0, \label{eq:prandtl-div}
\end{align}
\endgroup
where $u : \mathbb{T}^2 \times (0,T) \rightarrow \mathbb{R}$ is the horizontal velocity, $w : \mathbb{T}^2 \times (0,T) \rightarrow \mathbb{R}$ is the vertical velocity and $p : \mathbb{T}^2 \times (0,T) \rightarrow \mathbb{R}$ the pressure. 

We observe that these equations differ from the two-dimensional hydrostatic Euler equations \eqref{2Dhydreuler1}-\eqref{2Dhydreuler3} by the vertical viscosity term $\nu_v^* \partial_{zz} u$. We introduce the following notion of weak solution to the Prandtl equations \eqref{eq:prandtl-u}-\eqref{eq:prandtl-div}.

\begin{definition} \label{prandtlweaksolutiondef} 
A triple $(u,w,p)$ is called a weak solution of the two-dimensional Prandtl equations \eqref{eq:prandtl-u}-\eqref{eq:prandtl-div} if $u \in L^2 (\mathbb{T}^2 \times (0,T))\cap L^1((0,T);W^{1,1}(\mathbb{T}^2))$, $w \in \mathcal{D}' ( \mathbb{T}^2 \times (0,T))$ and $p \in L^1 (\mathbb{T}^2 \times (0,T))$ such that $u w \in L^1 ((0,T); B^{-s}_{1,\infty} (\mathbb{T}^2))$ (where $s > 0$ is the regularity parameter) and the equations are satisfied in the sense of distributions, i.e.,
\begin{align*}
    \int_0^T \int_{\mathbb{T}^2} u \partial_t \phi_1 \dx\dt - \nu_v^* \int_0^T \int_{\mathbb{T}^2} \partial_z u \partial_{z} \phi_1 \dx \dt + \int_0^T \int_{\mathbb{T}^2} u^2 \partial_{x_1} \phi_1 \dx \dt + \qquad \quad & \\
    + \int_0^T \langle u w , \partial_z \phi_1 \rangle_{B^{-s}_{1,\infty} \times B^s_{\infty,1}} \dt + \int_0^T \int_{\mathbb{T}^2} p \partial_{x_1} \phi_1 \dx\dt &= 0, \\ 
    \int_0^T \int_{\mathbb{T}^2} p \partial_z \phi_2 \dx\dt &= 0, \\
    \int_0^T \langle \textbf{u} , \nabla \phi_3 \rangle \dt &= 0,
\end{align*}
for all test functions $\phi_1$, $\phi_2$ and $\phi_3$ in $\mathcal{D} ( \mathbb{T}^2 \times (0,T))$.
\end{definition}

We will prove the following result. 

\begin{theorem} \label{prandtltheorem} 
Let $T>0$ and suppose there exist smooth solutions of the two-dimensional Prandtl equations \eqref{eq:prandtl-u}-\eqref{eq:prandtl-div} $(u_1,w_1,p_1)$ on $[0,T/2]$ and $(u_2,w_2,p_2)$ on $[T/2,T]$. Moreover, let $1 \leq q_2, q_3 \leq \infty$ and $0 < s_3$ be parameters satisfying 
\begin{equation} \label{prandtlconstraints} 
    \frac{3}{2} q_2 \left(\frac{1}{q_3} - \frac{1}{q_2}\right)> s_3 > \frac{1}{1-\frac{2}{q_2}} \left(\frac{1}{q_3} - \frac{1}{q_2}\right)>0, \quad 1\geq s_3.
\end{equation}
Then there exists a weak solution $(u,w,p)$ in the sense of Definition~\ref{prandtlweaksolutiondef} with regularity parameter $s=1$ and with the following properties:
\begin{enumerate}
    \item The solution satisfies that
    $$
        (u,w,p)(\cdot,t) = \left\{ \begin{array}{ll}
        (u_1,w_1,p_1)(\cdot,t) & \text{ if } t\in[0,T/4), \\
        (u_2,w_2,p_2)(\cdot,t) & \text{ if } t\in(3T/4,T].
        \end{array} \right.
    $$

    \item We have that 
    \begin{align*}
        u &\in L^{q_2-} ((0,T); L^2 (\mathbb{T}^2)) \cap L^{q_3-} ((0,T); H^{s_3} (\mathbb{T}^2)) \cap L^1((0,T);W^{1,1}(\mathbb{T}^2)), \\
        w &\in L^{q_2'-} ((0,T); L^2 (\mathbb{T}^2)) \cap L^{q_3'-} ((0,T); H^{-s_3} (\mathbb{T}^2)) . 
    \end{align*}
\end{enumerate}
\end{theorem}

\begin{remark} 
Remark~\ref{rem:endpoint-othercases} is also true in the context of the Prandtl equations \eqref{eq:prandtl-u}-\eqref{eq:prandtl-div}.
\end{remark}

\begin{remark} \label{rem:viscosities-2D} 
Similar to Remark~\ref{rem:viscosities}, Theorem~\ref{prandtltheorem} holds for any $\nu_v^* \in \mathbb{R}$, in particular even for the inviscid case $\nu_v^* = 0$.
\end{remark}

We note that it is possible to establish an analogue of Corollary \ref{infweaksolcor} where one has to use the local well-posedness result from \cite[p.~6]{paicu} (see also \cite[p.~7186]{wang}) for analytic data in the strip/channel. A straightforward adaption of the proof of Corollary \ref{infweaksolcor} yields the existence of infinitely many global weak solutions for suitable initial data. 

\begin{remark}
We note that Theorem \ref{prandtltheorem} also applies to the 2D viscous primitive equations, with the same constraints \eqref{prandtlconstraints}. This is because the $L^1 ((0,T); W^{1,1} (\mathbb{T}^2))$ regularity needed to control the horizontal and vertical viscosities, has already been obtained in the proof of Theorem \ref{prandtltheorem} (and Proposition \ref{prop:prandtl-perturbative}).
\end{remark}

\subsection{Further remarks and outline of the paper} 
Now that we have presented the results for the four cases of system \eqref{generalequation-u}-\eqref{generalequation-div} that we consider in this paper, we would like to make some further remarks on these results. Some conclusions that can be drawn are:
\begin{enumerate}
    \item There exist weak solutions of the inviscid primitive equations \eqref{eq:hyd-euler-u}-\eqref{eq:hyd-euler-div} that do not conserve energy. Compared to the solutions constructed in \cite{chiodaroli}, the solutions that we construct in this paper have Sobolev regularity. Moreover, they are related to the notion of type III weak solutions, as introduced in \cite{boutroshydrostatic}, as the barotropic-vertical part of the nonlinearity is interpreted as a paraproduct.
    \item In addition, the scheme is able to construct solutions where the baroclinic and ba\-ro\-tro\-pic modes have different regularities. This is expected, as the loss of derivative in the advective term only occurs in the baroclinic equation. The barotropic mode must have higher Sobolev regularity than the baroclinic mode in the scheme as otherwise the paraproduct between the vertical velocity and the barotropic mode will not make sense.
    \item As far as we can tell, this is the first proof of nonuniqueness of weak solutions for the viscous primitive equations. It shows that although the system is globally well-posed (as shown in \cite{cao}), at low regularity the system has nonunique weak solutions. This is true even if one has sufficiently regular Sobolev data for which global well-posedness holds in the class of strong solutions. 
    \item To the best of our knowledge, this is the first convex integration scheme for the two-dimensional Prandtl equations (in the three-dimensional case there is the work \cite{luo}), as well as the two-dimensional hydrostatic Euler equations. Indeed, the fact that in this paper we consider a new class of weak solutions to the Prandtl equations (as discussed in section \ref{subsubsec:general-weak}), makes it possible to address several of the issues raised in \cite{luo}. In particular, we are able to use Mikado flows (while in \cite{luo} linear plane waves are used for the perturbation) and (temporal and spatial) intermittency as part of the construction for the Prandtl equations. It was left open in \cite{luo} whether these tools can be applied to the Prandtl case.
\end{enumerate}
There are a few new features of the scheme that we wish to highlight:
\begin{itemize}
    \item We have introduced a splitting of the Reynolds stress tensor into a barotropic and baroclinic part. We add perturbations to separately deal with both these parts of the error. We then ensure that the interactions between the two perturbations are controlled. 
    \item The splitting of the Reynolds stress tensor requires us to construct and use horizontal and vertical inverse divergence operators, as the barotropic part depends only on the horizontal variables, while the baroclinic part is mean-free with respect to the $z$-variable.
    \item Having two parts of the perturbation allows us to use different scalings of the temporal intermittency functions for the barotropic and baroclinic parts of the perturbation. This makes it possible to ensure that the different perturbations have different regularities, such that the interactions between the different parts can be controlled. In particular, this is crucial to control the terms $\overline{u}_p \otimes \widetilde{u}_p$ and $w_p \overline{u}_p$ (the barotropic-baroclinic and vertical-baroclinic parts of the nonlinearity). 
\end{itemize} 

Now we present an outline of the paper. In sections \ref{inductivesection}-\ref{reynoldsestimates} we will develop the convex integration scheme for the 3D inviscid primitive equations, in order to prove Theorem \ref{mainresult}. In section \ref{inductivesection} we state the core inductive proposition of the convex integration scheme and prove Theorem \ref{mainresult} using this proposition. In section \ref{preliminaries} we discuss several preliminaries. In particular, we introduce the inverse divergence operators, the spatial building blocks for the convex integration, as well as the temporal intermittency functions. In addition, we will discuss the choice of the frequency parameters. 

In section \ref{perturbationsection} we introduce the perturbation that will be used in each iteration of the convex integration scheme, and compute the new Reynolds stress tensor after adding the perturbation. We will prove the estimates on the perturbation required for Proposition \ref{perturbativeproposition} in section \ref{perturbationestimates}. The estimates on the Reynolds stress tensor will be proven in section \ref{reynoldsestimates}. 

In sections \ref{viscoussection}-\ref{prandtlsection} we will develop convex integration schemes to study the other cases of equations \eqref{generalequation-u}-\eqref{generalequation-div} that we are interested in this paper. These schemes differ from the scheme presented in sections \ref{preliminaries}-\ref{reynoldsestimates} in some aspects, while other parts are similar. Therefore for the sake of brevity, in sections \ref{viscoussection}-\ref{prandtlsection} we will focus on the parts that differ from the convex integration scheme for the 3D inviscid primitive equations.

In section \ref{viscoussection} we provide an extension of the convex integration scheme to the viscous primitive equations with full viscosity and prove Theorem \ref{viscousmainthm}. The cases with anisotropic viscosities can be studied in a similar manner. In section \ref{2Dsection} we investigate the two-dimensional hydrostatic Euler equations and prove Theorem \ref{2Dtheorem}. Finally, in section \ref{prandtlsection} we consider the (two-dimensional) Prandtl equations and provide the proof for Theorem \ref{prandtltheorem}.

In appendix \ref{paradifferentialappendix} we give a short introduction to Littlewood-Paley theory, Besov spaces and paradifferential calculus, in order to make the paper self-contained. In appendix \ref{holdersection} we state the improved H\"older inequality, which was introduced in \cite[Lemma~2.1]{modena}, and we prove an oscillatory paraproduct estimate based on this inequality. Moreover, we provide another proof of Lemma \ref{lemma:prod} as an alternative to the proof given in section \ref{subsubsec:vpp}. This Lemma states a new inequality needed to control the interaction between the vertical velocity and the baroclinic mode, which turns out to be a critical part of the scheme.

\section{The inductive proposition} \label{inductivesection}

The following underdetermined system of equations is called the hydrostatic Euler-Reynolds system
\begingroup
\allowdisplaybreaks
\begin{align}
    \partial_t u + u \cdot \nabla_h u + w \partial_z u + \nabla_h p &= \nabla_h \cdot R_h + \partial_z R_v, \label{eq:hyd-er-u} \\
    \partial_z p &= 0, \label{eq:hyd-er-p} \\
    \nabla_h \cdot u + \partial_z w &=0 , \label{eq:hyd-er-div}
\end{align}
\endgroup
where $u, w, p, R_h$ and $R_v$ are the unknowns. 
Here the horizontal Reynolds stress tensor\footnote{We denote the set of all symmetric $2\times 2$ matrices by $\mathcal{S}^{2 \times 2}$.} $R_h : \mathbb{T}^2 \times [0,T] \rightarrow \mathcal{S}^{2 \times 2}$ is a function of $(x_1,x_2,t)$, while the vertical Reynolds stress tensor $R_v : \mathbb{T}^3 \times [0,T] \rightarrow \mathbb{R}^2$ is a function of $(x_1,x_2,z,t)$, and which is mean-free with respect to $z$, i.e. $\int_0^1 R_v \dz=0$. We will only work with smooth solutions to this system.

\begin{remark}
    Notice that $R_h$ is independent of $z$. Hence we have $\overline{R_h}=R_h$, see section~\ref{subsubsec:notation}. Moreover by definition $R_v$ is mean-free with respect to $z$ and thus $\widetilde{R_v}=R_v$.
\end{remark}

The following definition is inspired by \cite[Definition~2.1]{cheskidovluo3}.

\begin{definition} \label{wellprepareddef}
    We say that a smooth solution $(u,w,p,R_h,R_v)$ of the hydrostatic Euler-Reynolds system \eqref{eq:hyd-er-u}-\eqref{eq:hyd-er-div} is \emph{well-prepared} if there exists a time interval $I \subseteq [0,T]$ and parameter $\tau > 0$ such that $R_h (x,t) = 0$, $R_v (x,t) = 0$ whenever $\dist (t, I^c) \leq \tau$.
\end{definition}

\begin{remark}
In the definition of well-preparedness, the trivial case $I = [0,T]$ (i.e. without restrictions on the support of $R_h$ and $R_v$) has not been excluded. In this case, the perturbations considered in the inductive proposition will be supported on the whole time interval, but the estimates stated in Proposition \ref{perturbativeproposition} below also hold when $I = [0,T]$. Including the trivial case in Definition \ref{wellprepareddef} therefore allows us to phrase Proposition \ref{perturbativeproposition} in a more general way.
\end{remark}

The core of the proof of Theorem \ref{mainresult} will revolve around proving the following inductive proposition.

\begin{proposition} \label{perturbativeproposition}
    Suppose $(u, w, p, R_h,R_v)$ is a smooth solution of the hydrostatic Euler-Reynolds system \eqref{eq:hyd-er-u}-\eqref{eq:hyd-er-div} which is well-prepared with associated time interval $I$ and parameter $\tau>0$. Moreover consider parameters $1 \leq q_1, q_2, q_3 \leq \infty$ and $0<s_1, s_3 $ which satisfy the following constraints\footnote{Note that the constraints \eqref{constraints-prop} are weaker than \eqref{3dconstraints}. Indeed from \eqref{3dconstraints} we deduce
    \begin{equation*}
        s_3 + \frac{2}{q_2} < s_1 + 1 < \frac{2}{q_1} \leq \frac{2}{q_3}. 
    \end{equation*}}
    \begin{align} \label{constraints-prop}
        q_2 &> 2, \quad \frac{2}{q_1} > s_1 + 1, \quad \frac{2}{q_3} > s_3 + \frac{2}{q_2}.
    \end{align}
    Finally let $\delta, \epsilon > 0$ be arbitrary. Then there exists another smooth solution $(u + \overline{u}_p + \widetilde{u}_p, w + w_p, p + P, R_{h,1}, R_{v,1})$ of the hydrostatic Euler-Reynolds system \eqref{eq:hyd-er-u}-\eqref{eq:hyd-er-div} which is well-prepared with respect to the same time interval $I$ and parameter $\tau/2$, and has the following properties:
    \begin{enumerate}
        \item $(\overline{u}_p,\widetilde{u}_p,w_p)(x,t)=(0,0,0)$ whenever $\dist (t, I^c) \leq \tau/2$.
        
        \item The following estimates are satisfied\footnote{As usual we write $X\lesssim Y$ if $X\leq C Y$ with a constant $C$. The implicit constant $C$ in \eqref{eq:main-w1e}-\eqref{eq:main-product} does not depend on $u,w,p,R_h,R_v$ or $\epsilon$.}\footnote{As mentioned earlier, we quantify the `$-$' in $p-$ via $p-\coloneqq\frac{1}{\frac{1}{p} + \delta}$, where $\delta > 0$ was fixed in the statement of the proposition.} 
        \begingroup
        \allowdisplaybreaks
        \begin{align}
            \lVert R_{h,1} \rVert_{L^1 (L^1)} &\leq \epsilon, \label{eq:main-Rh} \\
            \lVert R_{v,1} \rVert_{L^1 (B^{-1}_{1,\infty})} &\leq \epsilon, \label{eq:main-Rv} \\
            \lVert \overline{u}_p \rVert_{L^{q_1} (H^{s_1})} &\leq \epsilon, \label{eq:main-ubar1} \\
            \lVert \widetilde{u}_p \rVert_{L^{q_2-} (L^2)} &\leq \epsilon, \label{eq:main-utilde1} \\
            \lVert \widetilde{u}_p \rVert_{L^{q_3-} (H^{s_3})} &\leq \epsilon, \label{eq:main-utilde2} \\
            \lVert w_p \rVert_{L^{q_2'-} (L^2)} &\leq \epsilon, \label{eq:main-w1} \\
            \lVert w_p \rVert_{L^{q_3'-} (H^{-s_3})} &\leq \epsilon. \label{eq:main-w2} \\
            \lVert w_p \rVert_{L^{q_2'} (L^2)} &\lesssim \lVert R_h \rVert_{L^1(L^1)}, \label{eq:main-w1e} \\ 
            \lVert w_p \rVert_{L^{q_3'} (H^{-s_3})} &\lesssim \lVert R_h \rVert_{L^1(L^1)}. \label{eq:main-w2e} 
        \end{align} 
        \endgroup
        \item Moreover, we have the following bounds
        \begin{align}
            \lVert \overline{u}_p \rVert_{L^{2} (L^2)} &\lesssim \lVert R_h \rVert_{L^1 (L^1)}^{1/2}, \label{eq:main-ubar2} \\
            \lVert w_p \widetilde{u}_p + w \widetilde{u}_p + w_p u \rVert_{L^1 (B^{-1}_{1,\infty})} &\lesssim \lVert R_v \rVert_{L^1 (B^{-1}_{1,\infty})}. \label{eq:main-product} 
        \end{align}
    \end{enumerate}
\end{proposition}

\begin{remark}
    Note that when writing $\overline{u}_p$, $\widetilde{u}_p$, we implicitly require $\overline{u}_p$ to be independent of $z$ and $\widetilde{u}_p$ mean-free with respect to $z$, see section~\ref{barotropicsection}.
\end{remark}

\begin{remark} \label{rem:uptilde-endpoint}
    Another version of Proposition~\ref{perturbativeproposition} where \eqref{eq:main-w1e} and \eqref{eq:main-w2e} are replaced by 
    \begin{align}
        \lVert \widetilde{u}_p \rVert_{L^{q_2} (L^2)} &\lesssim \lVert R_h \rVert_{L^1(L^1)}, \label{eq:main-utilde1e} \\
        \lVert \widetilde{u}_p \rVert_{L^{q_3} (H^{s_3})} &\lesssim \lVert R_h \rVert_{L^1(L^1)}, \label{eq:main-utilde2e} 
    \end{align}
    is true as well, see Remark~\ref{rem:uptilde-endpoint-proof}. This way we obtain the endpoint time integrability for $\widetilde{u}$ rather than $w$ in Theorem~\ref{mainresult}, see Remark~\ref{rem:endpoint}.
\end{remark}

Next we prove Theorem \ref{mainresult} using Proposition \ref{perturbativeproposition}.

\begin{proof}[Proof of Theorem \ref{mainresult}]
We first take $\chi_1$ and $\chi_2$ to be a $C^\infty$ partition of unity of $[0,T]$ such that $\chi_1 \equiv 1$ on $[0,3T/8]$ and $\chi_2 \equiv 1$ on $[5T/8, T]$.  
Then we define $(u_0, w_0, p_0)$ as follows
\begin{equation}
    (u_0, w_0, p_0) \coloneqq \chi_1 (u_1, w_1, p_1) + \chi_2 (u_2, w_2, p_2).
\end{equation}
For a suitable choice of $\chi_1$ and $\chi_2$, $(u_0, w_0, p_0)$ is no longer a solution of the hydrostatic Euler equations, but with a proper definition\footnote{Using the inverse divergence operators from section~\ref{inversedivergenceappendix}, a precise definition of $R_{h,0},R_{v,0}$ is straightforward.} of $R_{h,0},R_{v,0}$ it solves the hydrostatic Euler-Reynolds system \eqref{eq:hyd-er-u}-\eqref{eq:hyd-er-div}. Moreover $(u_0,w_0,p_0,R_{h,0},R_{v,0})$ is well-prepared for the time interval $I \coloneqq [T/4,3T/4] \subset [0,T]$ and parameter $\tau_0 \coloneqq T/16$. 

Taking the sequence $\epsilon_n = 2^{-n}$, $\delta_n=\delta$ with a suitable choice of $\delta>0$ (see below) and applying Proposition \ref{perturbativeproposition} inductively, we find a sequence of well-prepared solutions 
\begin{equation} \label{eq:sequence-subsol}
    \bigg(u_0 + \sum_{k=1}^n \big( \overline{u}_k + \widetilde{u}_k \big), w_0 + \sum_{k=1}^n w_{k} , p_0 + \sum_{k=1}^n P_{k} , R_{h,n}, R_{v,n}\bigg) ,
\end{equation}
of the hydrostatic Euler-Reynolds system with a sequence of well-preparedness parameters $\{ \tau_n \}$ (and the same time interval $I$). Note that $\tau_n \rightarrow 0$.

Estimates \eqref{eq:main-Rh}, \eqref{eq:main-ubar1}-\eqref{eq:main-ubar2} imply that the sequence $\Big\{ \overline{u}_0 + \sum_{k=1}^n \overline{u}_k \Big\}$ is a Cauchy sequence in the space $L^2 (\mathbb{T}^3 \times (0,T)) \cap L^{q_1} ((0,T); H^{s_1} (\mathbb{T}^3))$, the sequence $\Big\{ \widetilde{u}_0 + \sum_{k=1}^n \widetilde{u}_k \Big\}$ is Cauchy in $L^{q_2-} ((0,T); L^2 (\mathbb{T}^3)) \cap L^{q_3-} ((0,T); H^{s_3} (\mathbb{T}^3))$ and the sequence $\Big\{ w_0 + \sum_{k=1}^n w_{k} \Big\}$ is Cauchy in $L^{q_2'} ((0,T) ; L^2 (\mathbb{T}^3)) \cap L^{q_3'} ((0,T); H^{-s_3} (\mathbb{T}^3))$. In particular, by choosing $\delta$ appropriately we can identify the limits $\overline{u}, \widetilde{u}$ and $w$, as $n \rightarrow \infty$, lying in the spaces stated in Theorem~\ref{mainresult}.

Now we define the pressure $p$ by
\begin{equation} \label{pressuredefinition}
    p \coloneqq - (\Delta_h)^{-1} \big( \nabla_h \cdot ( \nabla_h \cdot (\overline{u\otimes u})) \big),
\end{equation}
where $u=\overline{u}+\widetilde{u}$.

Next, we check that the triple $(\overline{u} + \widetilde{u}, w,p)$ is a weak solution in the sense of Definition~\ref{weaksolution3Dinviscid}. We first show that $uw\in L^1((0,T);B_{1,\infty}^{-1}(\mathbb{T}^3))$. According to Lemma~\ref{lemma:paradiff-summary}, $\Big( \overline{u}_0 + \sum_{k=1}^n \overline{u}_k \Big) \Big( w_0 + \sum_{k=1}^n w_{k} \Big) \xrightarrow[]{n \rightarrow \infty} \overline{u} w$ in $L^1 ((0,T); B_{1,\infty}^{-1} (\mathbb{T}^3))$. Here we have also used that $1\geq s_1>s_3$ (which follows from \eqref{3dconstraints}) as well as Lemma~\ref{lemma:essential-besov}, and the fact that $\frac{1}{q_1} + \frac{1}{q_3'}\leq 1$ (which follows from $q_3\leq q_1$, see \eqref{3dconstraints}) in order to obtain $L^1$ integrability in time. 

Next, we need to check that $\Big( \widetilde{u}_0 + \sum_{k=1}^n \widetilde{u}_k \Big) \Big( w_0 + \sum_{k=1}^n w_{k} \Big) \xrightarrow[]{n \rightarrow \infty} \widetilde{u} w$ in $L^1 ((0,T); B^{-1}_{1,\infty} (\mathbb{T}^3))$. By estimates \eqref{eq:main-Rv} and \eqref{eq:main-product}, it follows that $\Big( \widetilde{u}_0 + \sum_{k=1}^n \widetilde{u}_k \Big) \Big( w_0 + \sum_{k=1}^n w_{k} \Big) \xrightarrow[]{n \rightarrow \infty} \widetilde{U} W$ in $L^1 ((0,T); B^{-1}_{1,\infty} (\mathbb{T}^3))$ for some element $\widetilde{U} W \in L^1 ((0,T); B^{-1}_{1,\infty} (\mathbb{T}^3))$. We need to show that $\widetilde{U} W = \widetilde{u} w$ (in a suitable Besov space).

By using estimates \eqref{eq:main-Rh}, \eqref{eq:main-utilde2} and \eqref{eq:main-w2e} we find that $\Big( \widetilde{u}_0 + \sum_{k=1}^n \widetilde{u}_k \Big) \Big( w_0 + \sum_{k=1}^n w_{k} \Big) \xrightarrow[]{n \rightarrow \infty} \widetilde{u} w$ in $L^{1-} ((0,T); B^{-s_3}_{1,\infty} (\mathbb{T}^3))$, where we have also applied Lemmas \ref{paraproduct} and \ref{resonance} (as well as the completeness of $L^p$ spaces for $0 < p < 1$, see for example \cite[~Proposition 10.5]{komornik}). Therefore because of the fact that $\Big( \widetilde{u}_0 + \sum_{k=1}^n \widetilde{u}_k \Big) \Big( w_0 + \sum_{k=1}^n w_{k} \Big)$ converges both to $\widetilde{U} W$ and $\widetilde{u} w$ in $L^{1-} ((0,T); B^{-1}_{1,\infty} (\mathbb{T}^3))$ (which is a quasi-normed space), it follows that $\widetilde{U} W = \widetilde{u} w$ in $L^{1-} ((0,T); B^{-1}_{1,\infty} (\mathbb{T}^3))$.

We note that the identification can also be made in a different way. The fact that $\Big\{ \overline{u}_0 + \sum_{k=1}^n \overline{u}_k \Big\}$ and $\Big\{ w_0 + \sum_{k=1}^n w_{k} \Big\}$ are Cauchy sequences in $L^1 (\mathbb{T}^3 \times (0,T))$ means that subsequences converge to $\tilde{u}$ respectively $w$ almost everywhere in $\mathbb{T}^3 \times (0,T)$. Hence $\Big( \widetilde{u}_0 + \sum_{k=1}^n \widetilde{u}_k \Big) \Big( w_0 + \sum_{k=1}^n w_{k} \Big)$ converges to $\tilde{u} w$ almost everywhere in $\mathbb{T}^3 \times (0,T)$. 

One can also show that $\Big( \widetilde{u}_0 + \sum_{k=1}^n \widetilde{u}_k \Big) \Big( w_0 + \sum_{k=1}^n w_{k} \Big)$ converges in $L^{1-} ((0,T); L^1(\mathbb{T}^3))$ to $\Tilde{U} W$ (which follows from slightly adapting the proof of Lemma \ref{verticalprinc} and related estimates). Therefore by Proposition 10.5 in \cite{komornik} there is another subsequence that converges to $\Tilde{U} W$ almost everywhere in $\mathbb{T}^3 \times (0,T)$. Therefore we are able to conclude that $\Tilde{U} W = \tilde{u} w$ almost everywhere in $\mathbb{T}^3 \times (0,T)$.

Furthermore, we observe that \eqref{eq:3D-inv-weak-p} immediately follows from the definition of $p$. Moreover since for any $n\in \mathbb{N}$ the quintuple \eqref{eq:sequence-subsol} satisfies \eqref{eq:hyd-er-div}, we find that $(u,w,p)$ complies with \eqref{eq:3D-inv-weak-div}. In order to show \eqref{eq:3D-inv-weak-u} we define the abbreviations
\begin{align*}
    u_n &\coloneqq u_0 + \sum_{k=1}^n \big( \overline{u}_k + \widetilde{u}_k \big), \\
    w_n &\coloneqq w_0 + \sum_{k=1}^n w_{k} ,\\
    p_n &\coloneqq p_0 + \sum_{k=1}^n P_{k} .
\end{align*}
Since \eqref{eq:sequence-subsol} satisfies \eqref{eq:hyd-er-u}, we observe that
\begingroup
\allowdisplaybreaks
\begin{align}
    &\int_0^T \int_{\mathbb{T}^3} u_n\cdot \partial_t\varphi \dx\dt + \int_0^T \int_{\mathbb{T}^3} \big( u_n\otimes u_n \big) : \nabla_h \varphi \dx \dt \notag\\
    &\quad+ \int_0^T \langle u_n w_n, \partial_z \varphi\rangle_{B^{-1}_{1,\infty} \times B^1_{\infty,1}} \dt + \int_0^T \int_{\mathbb{T}^3} p_n \nabla_h\cdot \varphi \dx\dt \notag \\
    &= \int_0^T \int_{\mathbb{T}^3} R_{h,n} : \nabla_h \varphi \dx\dt+ \int_0^T\langle R_{v,n} , \partial_z \varphi \rangle_{B^{-1}_{1,\infty} \times B^1_{\infty,1}} \dt, \label{eq:proof-2}
\end{align}
\endgroup
for any $n\in \mathbb{N}$ and any test function $\varphi\in \mathcal{D} (\mathbb{T}^3 \times (0,T))$. Note that \eqref{eq:main-Rh} and \eqref{eq:main-Rv} imply $R_{h,n} \xrightarrow[]{n \rightarrow \infty} 0$ in $L^1 ((0,T); L^1 (\mathbb{T}^3))$ and $R_{v,n} \xrightarrow[]{n \rightarrow \infty} 0$ in $L^1 ((0,T); B^{-1}_{1,\infty} (\mathbb{T}^3))$, respectively. Hence by taking the limit we deduce from \eqref{eq:proof-2} that 
\begin{equation} \label{eq:proof-3} 
    \int_0^T \int_{\mathbb{T}^3} u\cdot \partial_t\varphi \dx\dt + \int_0^T \int_{\mathbb{T}^3} \big( u\otimes u \big): \nabla_h \varphi\dx \dt + \int_0^T \langle uw, \partial_z \varphi\rangle_{B^{-1}_{1,\infty} \times B^1_{\infty,1}} \dt = 0, 
\end{equation} 
for any test function $\varphi\in \mathcal{D} (\mathbb{T}^3 \times (0,T))$ which is either mean-free with respect to $z$, or independent of $z$ with $\nabla_h\cdot \varphi=0$. Here we have used that for any $n\in \mathbb{N}$ and $\varphi$ mean-free with respect to $z$, 
$$
    \int_0^T \int_{\mathbb{T}^3} p_n \nabla_h \cdot \varphi \dx\dt = \int_0^T \int_{\mathbb{T}^2} \bigg[ p_n \nabla_h \cdot \bigg( \int_{\mathbb{T}} \varphi \dz \bigg) \bigg] \dx_1\dx_2\dt = 0, 
$$
according to \eqref{eq:hyd-er-p}, and, furthermore, that for any $n\in \mathbb{N}$ and $\varphi$ independent of $z$ with $\nabla_h\cdot \varphi=0$
$$
    \int_0^T \int_{\mathbb{T}^3} p_n \nabla_h \cdot \varphi \dx\dt = 0 .
$$
Now we are ready to prove \eqref{eq:3D-inv-weak-u}. We may split the test function $\phi_1 = \overline{\phi}_1 + \widetilde{\phi}_1$ into the barotropic and baroclinic parts, and use the Helmholtz decomposition to find test functions $\varphi,\psi$, which are independent of $z$, and such that $\overline{\phi}_1= \varphi + \nabla_h \psi$ and $\nabla_h\cdot \varphi = 0$. Then by \eqref{eq:proof-3} we have 
\begin{align*} 
    &\int_0^T \int_{\mathbb{T}^3} u \cdot\partial_t \phi_1 \dx\dt + \int_0^T \int_{\mathbb{T}^3} \big( u \otimes u \big) : \nabla_h \phi_1 \dx \dt \\
    &\quad + \int_0^T \langle uw , \partial_z \phi_1 \rangle_{B^{-s}_{1,\infty} \times B^s_{\infty,1}} \dt + \int_0^T \int_{\mathbb{T}^3} p \nabla_h \cdot \phi_1 \dx\dt \\
    &= \int_0^T \int_{\mathbb{T}^3} p \nabla_h\cdot \widetilde{\phi}_1 \dx\dt + \int_0^T \int_{\mathbb{T}^3} u \cdot \partial_t \nabla_h \psi \dx\dt \\
    &\quad + \int_0^T \int_{\mathbb{T}^3} \Big( (u\otimes u ):\nabla_h (\nabla_h \psi) + p \Delta_h\psi \Big) \dx\dt \\
    &= 0,
\end{align*}
where the latter equality follows from the fact that $p$ is independent of $z$, \eqref{eq:3D-inv-weak-div} and the definition of $p$ in equation \eqref{pressuredefinition}.

Finally, we observe that equation \eqref{solutionagreement} follows from Proposition \ref{perturbativeproposition} because the time interval $I$ of well-preparedness stays the same for the sequence \eqref{eq:sequence-subsol}. In particular, all the perturbations have support in the time interval $I$. Therefore since $(u_0, w_0, p_0)$ agrees with $(u_1, w_1, p_1)$ on $[0, T/4)$ and with $(u_2, w_2, p_2)$ on $(3T/4, T]$, the constructed solution $(u,w,p)$ will have the same properties, as no perturbations with support in $[0,T/4) \cup (3T/4, T]$ have been added. 
\end{proof}

After recalling some preliminaries in section~\ref{preliminaries}, we will prove Proposition~\ref{perturbativeproposition} in sections~\ref{perturbationsection}-\ref{reynoldsestimates}.

\section{Preliminaries} \label{preliminaries}

\subsection{Outline} \label{subsec:outline}
In this paper we are going to use Mikado flows as building blocks. These have been introduced in \cite{daneri} and are built upon a geometric lemma which goes back to \cite{Nash54}, see also \cite[Lemma~3.3]{szekelyhidi}. Later on \emph{concentrated} Mikado flows have been introduced in \cite{modena} in order to construct solutions with Sobolev regularity. In this paper the term \emph{Mikado flows} will always refer to such Mikado flows with concentration. 

In the proof of Proposition~\ref{perturbativeproposition} we will handle the two error terms $R_h$ and $R_v$ separately. To treat $R_v$ we use two-dimensional Mikado flows and Mikado densities in two directions, whereas for $R_h$ we use two-dimensional Mikado flows in several directions which are given by the above mentioned geometric lemma. We use the version of the Mikado flows which was introduced in \cite{cheskidovluo3}. We recall these flows in section~\ref{subsec:buildingblocks}. We call the perturbation which reduces the error $R_v$ \emph{vertical} and the one which reduces $R_h$ \emph{horizontal}.

The above mentioned concentration is represented by the spatial concentration parameters $\mu_h$ and $\mu_v$ which are used in the horizontal and vertical perturbation, respectively. Moreover, the perturbations will be highly oscillating flows, and this oscillation is represented by the spatial oscillation parameters $\sigma_h$ and $\sigma_v$, which are again used in the horizontal and vertical perturbation, respectively.

Finally we will use \emph{intermittent} flows. To this end we introduce temporal intermittency functions in section~\ref{subsec:intermittency}. These are time-dependent functions which contain the temporal concentration parameters $\kappa_h$, $\kappa_v$ and temporal oscillation parameters $\nu_h$, $\nu_v$.

\subsection{The parameters} 

As mentioned in section~\ref{subsec:outline} we have two sets of four parameters, namely $\{\mu_h,\sigma_h,\kappa_h,\nu_h\}$ and $\{\mu_v,\sigma_v,\kappa_v,\nu_v\}$, so eight parameters in total. In addition to that, we work with two ``master parameters'' $\lambda_h$ and $\lambda_v$, which the other parameters depend on via 
\begingroup
\allowdisplaybreaks
\begin{equation} \label{parameters}
    \begin{split}
        \nu_i &= \lambda^{a_i}_i, \quad \sigma_i = \lambda^{b_i}_i, \\
        \kappa_i &= \lambda^{c_i}_i, \quad \mu_i = \lambda_i
    \end{split}
\end{equation}
\endgroup
for $i=h,v$ and fixed exponents $a_i,b_i,c_i>0$. These exponents are determined in the following Lemma. We will later fix $\lambda_h,\lambda_v$. These parameters will be very large and such that $\sigma_h,\sigma_v\in \N$, as well as $\kappa_h, \kappa_v > 1$. 

\begin{lemma} \label{parameterlemma} 
    Let $1 \leq q_1, q_2, q_3 \leq \infty$ and $0<s_1, s_3$ satisfy the following conditions\footnote{Obviously conditions \eqref{constraints-lemma} are weaker than constraints \eqref{constraints-prop}, which again are weaker than \eqref{3dconstraints}.} 
    \begin{equation} \label{constraints-lemma}
        \frac{2}{q_1} > s_1 + 1, \quad \frac{2}{q_3} > s_3 + \frac{2}{q_2}.
    \end{equation} 
    Then we can choose $a_i,b_i,c_i>0$ for $i=h,v$ in \eqref{parameters} with the property that there exist $\gamma_h, \gamma_v > 0$ such that
    \begin{align}
        \kappa_h^{1/2-1/q_1} (\sigma_h \mu_h)^{s_1} &\leq \lambda_h^{-\gamma_h}, \label{parameterineq1} \\
        \sigma_h^{-1} \nu_h \kappa_h^{1/2} \mu_h^{-1} &\leq \lambda_h^{-\gamma_h}, \label{parameterineq2} \\ 
        \kappa_v^{1/q_2 - 1/q_3} (\sigma_v \mu_v)^{s_3} &= 1, \label{parameterineq3} \\ 
        \sigma_v^{-1} \nu_v \kappa_v^{1/2} \mu_v^{-1} &\leq \lambda_v^{-\gamma_v}, \label{parameterineq4} \\
        \kappa_v^{-\delta} &\leq \lambda_v^{-\gamma_v} ,\label{parameterineq5} 
    \end{align}
    and in addition $\mu_i,\sigma_i,\kappa_i,\nu_i\geq \lambda_i^{\gamma_i}$ for $i=h,v$.
\end{lemma}

\begin{proof} 
We choose $0< a_h, a_v < 1$ and set
\begin{align*} 
    b_h &\coloneqq \frac{2s_1}{\frac{2}{q_1} -s_1 - 1}, \\ 
    b_v &\coloneqq \frac{s_3}{\frac{2}{q_3}-s_3 - \frac{2}{q_2}}, \\
    c_h &\coloneqq 2b_h, \\
    c_v &\coloneqq 2b_v.
\end{align*}
Notice that \eqref{constraints-lemma} ensures that $b_h,b_v>0$. Consequently $c_h,c_v>0$.

By taking logarithms, inequalities \eqref{parameterineq1}-\eqref{parameterineq5} are equivalent to 
\begingroup
\allowdisplaybreaks
\begin{align}
    -\bigg( \frac{1}{2} - \frac{1}{q_1} \bigg) c_h - s_1 (b_h+1) &\geq \gamma_h, \label{eq:p1}\\
    b_h - a_h - \frac{1}{2} c_h + 1 &\geq \gamma_h, \label{eq:p2}\\
    -\bigg( \frac{1}{q_2} - \frac{1}{q_3} \bigg) c_v - s_3 (b_v+1) &=0 , \label{eq:p3}\\
    b_v - a_v - \frac{1}{2} c_v + 1 &\geq \gamma_v, \label{eq:p4}\\
    \delta c_v &\geq \gamma_v. \label{eq:p5}
\end{align}
\endgroup
Using the definition of $b_v$ and $c_v$ we immediately conclude that \eqref{eq:p3} is valid. 

The required additional estimates $\mu_i,\sigma_i,\kappa_i,\nu_i\geq \lambda_i^{\gamma_i}$ for $i=h,v$ translate into the bounds $1,a_i,b_i,c_i\geq \gamma_i$. Since these upper bounds for $\gamma_i$ are positive, it remains to show that the upper bounds given by the left-hand sides of \eqref{eq:p1}, \eqref{eq:p2}, \eqref{eq:p4} and \eqref{eq:p5} are positive as well.

It is obvious that $\delta c_v >0$. Furthermore from our choice of $a_i,b_i,c_i$ we obtain
\begin{equation*}
    b_i - a_i -\frac{1}{2}c_i + 1 = 1 -a_i >0,
\end{equation*}
and 
\begin{equation*}
    -\bigg( \frac{1}{2} - \frac{1}{q_1} \bigg) c_h - s_1 (b_h+1) =  b_h \bigg( \frac{2}{q_1} -s_1 - 1 \bigg) -s_1 = s_1 > 0.
\end{equation*}
\end{proof}

\begin{remark} 
    When proving Proposition~\ref{perturbativeproposition}, inequality \eqref{parameterineq1} ensures that $\lVert \overline{u}_p\rVert_{L^{q_1}(H^{s_1})}$ can be made small (see section~\ref{subsubsec:hpp}), while inequality \eqref{parameterineq3} guarantees that both $\lVert \widetilde{u}_p\rVert_{L^{q_3-}(H^{s_3})}$ and $\lVert w_p\rVert_{L^{q_3'-}(H^{-s_3})}$ can be made small (see section~\ref{subsubsec:vpp}). Moreover \eqref{parameterineq5} will be used at several points during the proof. Finally, inequalities \eqref{parameterineq2} and \eqref{parameterineq4} make sure that the temporal parts of the linear error are controlled, see section~\ref{subsubsec:lin-err-time}. 
\end{remark}

\subsection{Inverse divergence operators} \label{inversedivergenceappendix}

Like in most of the convex integration schemes in the context of fluid dynamics in the literature, we will need inverse divergence operators in order to define the new Reynolds stress tensors $R_{h,1}$ and $R_{v,1}$. In this context the first inverse divergence operator goes back to \cite{lelliscontinuous}. In this paper we will work with three inverse divergence operators. The horizontal inverse divergence $\mathcal{R}_h$ and its bilinear version $\mathcal{B}$ will be used to define the new horizontal Reynolds stress tensor $R_{h,1}$. Those operators are treated in sections \ref{subsubsec:hid} and \ref{subsubsec:bid}, respectively. In order to determine the new vertical Reynolds stress tensor $R_{v,1}$ we need a ``vertical inverse divergence'' which is just an integral in $z$. It is introduced in section~\ref{subsubsec:vid}.

\subsubsection{Horizontal inverse divergence} \label{subsubsec:hid}

Our horizontal inverse divergence coincides with the two-dimensional inverse divergence from \cite{cheskidovluo3}. It is based upon the inverse divergence introduced in \cite{lelliscontinuous} and is defined as follows.

\begin{definition}
    We define the map\footnote{We denote the set of all symmetric $2\times 2$ matrices with zero trace by $\mathcal{S}_0^{2 \times 2}$.} $\mathcal{R}_h : C^\infty (\mathbb{T}^2; \mathbb{R}^2) \rightarrow C^\infty (\mathbb{T}^2; \mathcal{S}_0^{2 \times 2} )$ by\footnote{We are using the Einstein summation convention here, in particular we sum over $k=1,2$. Moreover, we recall that $\delta_{ij}$ is the Kronecker delta and in the definition of the inverse horizontal Laplacian $\Delta_h^{-1}$ we assume the spatial horizontal average to be zero (in order to ensure uniqueness).} 
    \begin{equation}
        (\mathcal{R}_h v)_{ij} \coloneqq \mathcal{R}_{ijk,h} v_k,
    \end{equation}
    where 
    \begin{equation}
        \mathcal{R}_{ijk,h} \coloneqq - \Delta_h^{-1} \partial_k \delta_{ij} + \Delta_h^{-1} \partial_i \delta_{jk} + \Delta_h^{-1} \partial_j \delta_{ik}
    \end{equation}
    for $i,j,k\in \{1,2\}$.
\end{definition} 

The following Lemma, which can also be found in \cite[Appendix~B]{cheskidovluo3}, summarizes some properties of the map $\mathcal{R}_h$.

\begin{lemma} \label{lemma:hid}
    \begin{enumerate}
        \item The following identities hold
        \begin{align}
            \nabla_h \cdot (\mathcal{R}_h v) &= v  - \int_{\mathbb{T}^2} v \dx, \qquad \text{ for all } v\in C^\infty(\T^2;\R^2), \label{eq:lem-Rh-id1} \\
            \mathcal{R}_h \Delta_h v &= \nabla_h v + \nabla_h v^T, \qquad \text{ for all divergence-free } v\in C^\infty(\T^2;\R^2). \label{eq:lem-Rh-id2}
        \end{align}
        
        \item For $1 \leq p \leq \infty$, the operator $\mathcal{R}_h$ is bounded, i.e., for all $f \in C^\infty (\mathbb{T}^2 ; \mathbb{R}^2)$ we have that
        \begin{equation} \label{eq:lem-Rh-est1}
            \lVert \mathcal{R}_h f \rVert_{L^p} \lesssim \lVert f \rVert_{L^p}.
        \end{equation}
        If $f$ is mean-free, i.e. $\int_{\mathbb{T}^2} f \dx = 0$, then 
        \begin{equation} \label{eq:lem-Rh-est2}
            \lVert \mathcal{R}_h f (\sigma \cdot ) \rVert_{L^p } \lesssim \sigma^{-1} \lVert f \rVert_{L^p}, \qquad \text{for any } \sigma \in \mathbb{N}. 
        \end{equation}

        \item The operator $\mathcal{R}_h\nabla_h: C^\infty(\T^2;\R^{2\times 2}) \to C^\infty(\T^2;\mathcal{S}_0^{2\times 2})$ is a Calder{\'o}n-Zygmund operator, in particular for any $1<p<\infty$ and all $A\in C^\infty(\T^2;\R^{2\times 2})$ we have 
        \begin{equation} \label{eq:lem-Rh-est3}
            \|\mathcal{R}_h\nabla_h \cdot A \|_{L^p} \lesssim \|A\|_{L^p}.
        \end{equation} 
    \end{enumerate}
\end{lemma}

For the proof we refer to \cite[Appendix~B]{cheskidovluo3}.

\subsubsection{Horizontal bilinear inverse divergence} \label{subsubsec:bid}

Next we recall the bilinear inverse divergence operator from \cite{cheskidovluo3}. For our purposes we call it \emph{horizontal} bilinear inverse divergence. 

\begin{definition}
    We define $\mathcal{B}: C^\infty(\T^2;\R^2) \times C^\infty(\T^2;\R^{2\times 2}) \to C^\infty(\T^2;\mathcal{S}_0^{2\times 2})$ by
    \begin{equation}
        ( \mathcal{B} (b, A) )_{ij} = b_l \mathcal{R}_{ijk,h} A_{lk} - \mathcal{R}_h (\partial_i b_l \mathcal{R}_{ijk,h} A_{lk} ),
    \end{equation}
    or written without components (where we have abused notation)
    \begin{equation}
        \mathcal{B} (b,A) = b \mathcal{R}_h A - \mathcal{R}_h ( \nabla_h b \mathcal{R}_h A).
    \end{equation}
\end{definition}

We will also use the following Lemma from \cite{cheskidovluo3}. 

\begin{lemma} \label{lemma:bid}
    For $1 \leq p \leq \infty$, $b \in C^\infty (\mathbb{T}^2; \mathbb{R}^2)$ and $A \in C^\infty (\mathbb{T}^2 ; \mathbb{R}^{2 \times 2})$ with $\int_{\mathbb{T}^2} A \dx = 0$, it holds that
    \begin{equation}
        \nabla_h \cdot ( \mathcal{B} (b, A)) = b A - \int_{\mathbb{T}^2} b A \dx. 
    \end{equation}
    Moreover, we have the following estimate
    \begin{equation}
        \lVert \mathcal{B} (b, A) \rVert_{L^p} \lesssim \lVert b \rVert_{C^1} \lVert \mathcal{R}_h A \rVert_{L^p}. 
    \end{equation}
\end{lemma}

The proof of Lemma~\ref{lemma:bid} can be found in \cite[Appendix~B]{cheskidovluo3}.

\subsubsection{Vertical inverse divergence} \label{subsubsec:vid}

Finally we introduce the vertical inverse divergence as follows.

\begin{definition} \label{defn:vid}
    We define the map\footnote{We denote the space of all functions in $C^\infty(\T^3;\R^2)$ which have zero-mean with respect to $z$ by $C^\infty_{0,z}(\T^3;\R^2)$.} $\mathcal{R}_v: C^\infty_{0,z}(\T^3;\R^2) \to C^\infty(\T^3;\R^2)$ by 
    \begin{equation}
        (\mathcal{R}_v v) (x_1,x_2,z) \coloneqq \int_0^z v (x_1,x_2,z') \dz' - \int_0^1 \int_0^{z'} v (x_1,x_2,z'') \dz'' \dz'. 
    \end{equation} 
\end{definition}

The vertical inverse divergence operator has the following properties.

\begin{lemma} \label{Lpverticalbound}
    \begin{enumerate}
        \item The following identities hold for any $v\in C^\infty_{0,z}(\T^3;\R^2)$
        \begin{align}
            \int_{\T} \mathcal{R}_v v \dz &= 0, \label{eq:lem-Rv-id1} \\
            \partial_z \mathcal{R}_v v &= v , \label{eq:lem-Rv-id2} \\
            \mathcal{R}_v (\partial_{zz} v) &= \partial_z v. \label{eq:lem-Rv-id3}
        \end{align}
        
        \item For $1 \leq p,q \leq \infty$ and $s\in \mathbb{R}$, the operator $\mathcal{R}_v$ satisfies the following estimates
        \begingroup
        \allowdisplaybreaks
        \begin{align}
            \lVert \mathcal{R}_v f \rVert_{L^p} &\lesssim \lVert f \rVert_{L^p}, \label{eq:lem-Rv-est1}\\
            \lVert \mathcal{R}_v f \rVert_{B^{s}_{p,q}} &\lesssim \lVert f \rVert_{B^{s}_{p,q}}. \label{eq:lem-Rv-est2} 
        \end{align}
        \endgroup
        Moreover,
        \begin{equation} \label{eq:lem-Rv-est3}
            \lVert \mathcal{R}_v f (\sigma \cdot ) \rVert_{L^p } \lesssim \sigma^{-1} \lVert f \rVert_{L^p} \qquad \text{ for any }\sigma \in \mathbb{N}.
        \end{equation}

        \item For any $1\leq p\leq \infty$ and all $v\in C^\infty(\T^3;\R^2)$ we have 
        \begin{equation} \label{eq:lem-Rv-est4}
            \|\mathcal{R}_v\partial_z v \|_{L^p} \lesssim \|v\|_{L^p}.
        \end{equation} 
    \end{enumerate}
\end{lemma}

\begin{proof}
The identities \eqref{eq:lem-Rv-id1}, \eqref{eq:lem-Rv-id2} are just a simple consequence of the definition of $\mathcal{R}_v$. We also observe that
\begin{align*} 
    \mathcal{R}_v (\partial_{zz} v) (x_1,x_2,z) &= \int_0^z \partial_{zz} v (x_1,x_2,z') \dz' - \int_0^1 \int_0^{z'} \partial_{zz} v (x_1,x_2,z'') \dz'' \dz' \\
    &= \partial_z v (x_1,x_2,z) - \int_0^1 \partial_z v (x_1,x_2,z') \dz' = \partial_z v (x_1,x_2,z),
\end{align*}
i.e. \eqref{eq:lem-Rv-id3}.

Estimate \eqref{eq:lem-Rv-est1} is established simply by moving the $L^p$ norm inside the integral. In order to prove estimate \eqref{eq:lem-Rv-est2}, we first observe that 
\begin{equation} \label{eq:proof-1}
    \mathcal{R}_v \Delta_j f = \Delta_j \mathcal{R}_v f,
\end{equation}
which can be verified by a direct computation. Alternatively, thanks to equation \eqref{eq:lem-Rv-id1} we have $(\widehat{\mathcal{R}_v f})_k = \frac{\widehat{f}_k}{2\pi i k_3}$ (where $k_3 \neq 0$) and then equation \eqref{eq:proof-1} follows by using the definition of the Littlewood-Paley blocks. From \eqref{eq:proof-1} and \eqref{eq:lem-Rv-est1} we obtain
\begin{equation*}
    \lVert \Delta_j \mathcal{R}_v f \rVert_{L^p} = \lVert \mathcal{R}_v \Delta_j f \rVert_{L^p} \lesssim \lVert \Delta_j f \rVert_{L^p},
\end{equation*}
which implies \eqref{eq:lem-Rv-est2}.

To prove \eqref{eq:lem-Rv-est3} we set $f_{\sigma} (x_1,x_2,z) := f (x_1,x_2,\sigma z)$ and compute
\begingroup
\allowdisplaybreaks
\begin{align*}
    \mathcal{R}_v f_{\sigma}(x_1,x_2,z) &= \int_0^z f_{\sigma}(x_1,x_2,z') \dz' - \int_0^1 \int_0^{z'} f_{\sigma}(x_1,x_2,z'') \dz'' \dz' \\
    &= \sigma^{-1} \int_0^{\sigma z} f(x_1,x_2,z') \dz' - \sigma^{-1} \int_0^1 \int_0^{\sigma z'} f(x_1,x_2,z'') \dz'' \dz' \\
    &= \sigma^{-1} \int_0^{\sigma z} f(x_1,x_2,z') \dz' - \sigma^{-2} \int_0^{\sigma} \int_0^{z'} f(x_1,x_2,z'') \dz'' \dz' \\
    &= \sigma^{-1} \int_0^{\sigma z} f(x_1,x_2,z') \dz' - \sigma^{-1} \int_0^{1} \int_0^{z'} f(x_1,x_2,z'') \dz'' \dz'.
\end{align*}
\endgroup
Hence
\begingroup
\allowdisplaybreaks
\begin{align*}
    &\lVert \mathcal{R}_v f_{\sigma}(x_1,x_2,\cdot) \rVert_{L^p(\T)} \\
    &\leq \sigma^{-1} \bigg( \int_{\mathbb{T}} \bigg\lvert  \int_0^{\sigma z} f(x_1,x_2,z') \dz' \bigg\rvert^p \dz \bigg)^{1/p} + \sigma^{-1} \bigg( \int_{\mathbb{T}} \bigg\lvert \int_0^{1} \int_0^{z'} f(x_1,x_2,z'') \dz'' \dz' \bigg\rvert^p \dz \bigg)^{1/p} \\
    &\leq \sigma^{-1} \bigg( \sigma^{-1} \int_{\sigma \mathbb{T}} \bigg\lvert  \int_0^{z} f(x_1,x_2,z') \dz' \bigg\rvert^p \dz \bigg)^{1/p} + \sigma^{-1} \bigg( \int_{\mathbb{T}} \bigg\lvert \int_0^{1} \int_0^{z'} f(x_1,x_2,z'') \dz'' \dz' \bigg\rvert^p \dz \bigg)^{1/p} \\
    &\leq \sigma^{-1} \bigg( \int_{\mathbb{T}} \bigg\lvert  \int_0^{z} f(x_1,x_2,z') \dz' \bigg\rvert^p \dz \bigg)^{1/p} + \sigma^{-1} \bigg( \int_{\mathbb{T}} \bigg\lvert \int_0^{1} \int_0^{z'} f(x_1,x_2,z'') \dz'' \dz' \bigg\rvert^p \dz \bigg)^{1/p} \\
    &\lesssim \sigma^{-1} \lVert f(x_1,x_2,\cdot) \rVert_{L^p(\T)}. 
\end{align*}
\endgroup
This implies
\begingroup
\allowdisplaybreaks
\begin{align*}
    \lVert \mathcal{R}_v f(\sigma\cdot) \rVert_{L^p(\T^3)} &= \bigg(\int_{\T^2} \lVert \mathcal{R}_v f(\sigma x_1, \sigma x_2, \sigma\cdot) \rVert^p_{L^p(\T)} \dx_1 \dx_2 \bigg)^{1/p} \\
    &\lesssim \sigma^{-1} \bigg(\int_{\T^2} \lVert f(\sigma x_1,\sigma x_2,\cdot) \rVert^p_{L^p(\T)} \dx_1 \dx_2 \bigg)^{1/p} \\
    &= \sigma^{-1} \lVert f \rVert_{L^p(\T^3)} ,
\end{align*}
\endgroup
i.e. \eqref{eq:lem-Rv-est3}. The case $p = \infty$ follows in a similar fashion.

Finally observe that 
\begin{align*}
    \mathcal{R}_v \partial_z v &= \int_0^z \partial_z v \dz' - \int_0^1 \int_0^{z'} \partial_z v \dz'' \dz' \\
    &= v - \int_0^1 v \dz' ,
\end{align*}
which immediately yields \eqref{eq:lem-Rv-est4}.
\end{proof}

\subsection{Building blocks for the perturbation} \label{subsec:buildingblocks}

Next we recall the building blocks. We begin with the Mikado flows and Mikado densities which we use to handle $R_v$. We state their existence together with their most important properties in the following proposition. The construction of the Mikado flows and densities is nowadays standard and goes back to \cite{daneri}. For the proof of the following proposition we refer to \cite[Section~4.1]{cheskidovluo3}.

\begin{proposition} \label{prop:mikado-v}
    For each $k\in \{1,2\}$ there exist functions $W_k\in C^\infty(\T^2; \R^2)$ and $\phi_k\in C^\infty(\T^2;\R)$ (referred to as the Mikado flows and Mikado densities respectively) depending on a parameter $\mu_v$, with the following properties: 
    \begin{enumerate}
        \item The functions $W_k, \phi_k$ have zero mean for all $k\in \{1,2\}$. Moreover
        \begin{equation} \label{eq:mean-Wvphiv}
            \int_{\T^2} W_k \phi_k \dx = \mathbf{e}_k \quad \text{ for all }k\in \{1,2\},
        \end{equation}
        where $\mathbf{e}_k$ denotes the $k$-th standard basis vector in $\R^2$, and by construction $W_k=\phi_k \mathbf{e}_k$.
        
        \item For any $k\in \{1,2\}$ there exists\footnote{We denote the set of all skew-symmetric $2\times 2$ matrices by $\mathcal{A}^{2\times 2}$.} $\Omega_k\in C^\infty(\T^2;\mathcal{A}^{2\times 2})$ with zero mean such that $W_k=\nabla_h \cdot \Omega_k$. In particular, $\nabla_h\cdot W_k=0$. Moreover $\nabla_h\cdot (W_k \phi_k)= W_k \cdot \nabla_h \phi_k = 0$.
        
        \item For all $s\geq 0$, $1\leq p\leq \infty$ and $k,k'\in \{1,2\}$ with $k\neq k'$ the following estimates hold:
        \begin{align}
            \|\phi_k\|_{W^{s,p}(\T^2)} &\lesssim \mu_v^{\frac{1}{2}-\frac{1}{p} + s} ; \label{eq:est-phiv}\\
            \|W_k\|_{W^{s,p}(\T^2)} &\lesssim \mu_v^{\frac{1}{2}-\frac{1}{p} + s}; \label{eq:est-Wv}\\
            \|\Omega_k\|_{W^{s,p}(\T^2)} &\lesssim \mu_v^{-\frac{1}{2}-\frac{1}{p}+s}; \label{eq:est-Ov}\\
            \|W_k\otimes W_{k'}\|_{L^p(\T^2)} &\lesssim \mu_v^{1-\frac{2}{p}}. \label{eq:est-WvWv}
        \end{align}
        Here the implicit constant may depend on $s,p$, but it does not depend on $\mu_v$.
    \end{enumerate}
\end{proposition}

Let us now recall the Mikado flows which we will use to treat $R_h$. In the following proposition $B_{1/2}(\mathbb{I})$ denotes the closed ball in $\mathcal{S}^{2\times 2}$ around the identity matrix $\mathbb{I}$ with radius $1/2$. For the proof we refer to \cite[Lemma~4.2,~Theorem~4.3]{cheskidovluo3}.

\begin{proposition} \label{prop:mikado-h}
    There exists $N\in \N$, $N\geq 3$ and for each $k\in \Lambda:=\{3,...,N\}$ there exists a flow $W_k\in C^\infty(\T^2;\R^2)$ (called Mikado flows) depending on a parameter $\mu_h$ and a function $\Gamma_k\in C^\infty(B_{1/2}(\mathbb{I});\R)$, with the following properties: 
    \begin{enumerate}
        \item The flows $W_k$ have zero mean, i.e. $\int_{\mathbb{T}^2} W_k \dx = 0$, for all $k\in \Lambda$. Moreover 
        \begin{equation} \label{eq:mean-WhWh}
            \sum_{k\in \Lambda} \Gamma_k^2(R) \int_{\T^2} W_k\otimes W_k \dx = R \quad \text{ for all }R\in B_{1/2}(\mathbb{I}).
        \end{equation}
        
        \item For any $k\in \Lambda$ there exists $\Omega_k\in C^\infty(\T^2;\mathcal{A}^{2\times 2})$ with zero mean such that $W_k=\nabla_h \cdot \Omega_k$. In particular, $\nabla_h\cdot W_k=0$. Moreover, $\nabla_h\cdot (W_k\otimes W_k)= W_k \cdot \nabla_h W_k = 0$.
        
        \item For all $s\geq 0$, $1\leq p\leq \infty$ and $k,k'\in \Lambda$ with $k\neq k'$ the following estimates hold:
        \begin{align}
            \|W_k\|_{W^{s,p}(\T^2)} &\lesssim \mu_h^{\frac{1}{2}-\frac{1}{p}+s}; \label{eq:est-Wh} \\
            \|\Omega_k\|_{W^{s,p}(\T^2)} &\lesssim \mu_h^{-\frac{1}{2}-\frac{1}{p}+s}; \label{eq:est-Oh} \\
            \|W_k\otimes W_{k'}\|_{L^p(\T^2)} &\lesssim \mu_h^{1-\frac{2}{p}}. \label{eq:est-WhWh}
        \end{align}
        Here the implicit constant may depend on $s,p$, but it does not depend on $\mu_h$.
    \end{enumerate}
\end{proposition}

The following Lemma is a simple corollary of \eqref{eq:est-phiv}-\eqref{eq:est-Ov}, \eqref{eq:est-Wh} and \eqref{eq:est-Oh}.

\begin{lemma} \label{mikadobounds}
Let $\sigma \in \mathbb{N}$. Then we have the following bounds for all $s \geq 0$, $1 \leq p \leq \infty$ and $k\in \{1,2\}$, $k'\in \Lambda$:
\begingroup
\allowdisplaybreaks
\begin{align}
    \| \phi_k (\sigma \cdot ) \|_{W^{s,p}} &\lesssim (\sigma \mu_v)^s \mu_v^{\frac{1}{2} - \frac{1}{p}}, \label{mikadodensity} \\
    \| W_k (\sigma \cdot ) \|_{W^{s,p}} &\lesssim (\sigma \mu_v)^s \mu_v^{\frac{1}{2} - \frac{1}{p}}, \label{mikadoflow-v} \\
    \| W_{k'} (\sigma \cdot ) \|_{W^{s,p}} &\lesssim (\sigma \mu_h)^s \mu_h^{\frac{1}{2} - \frac{1}{p}}, \label{mikadoflow-h} \\
    \| \Omega_{k} (\sigma \cdot ) \|_{W^{s,p}} &\lesssim (\sigma \mu_v)^s \mu_v^{-\frac{1}{2} - \frac{1}{p}}, \label{Omega-v} \\
    \| \Omega_{k'} (\sigma \cdot ) \|_{W^{s,p}} &\lesssim (\sigma \mu_h)^s \mu_h^{-\frac{1}{2} - \frac{1}{p}}. \label{Omega-h}
\end{align} 
\endgroup
\end{lemma}

\begin{proof}
For any $s\in \N_0$, the estimates simply follow from taking derivatives and \eqref{eq:est-phiv}-\eqref{eq:est-Ov}, \eqref{eq:est-Wh} and \eqref{eq:est-Oh}. Then by interpolation we obtain the desired estimates for any $s\geq 0$.
\end{proof}

\subsection{Intermittency} \label{subsec:intermittency}

As was done in \cite{cheskidovluo1,cheskidovluo2,cheskidovluo3} we now introduce the temporal intermittency functions, which differ for the horizontal and vertical perturbations. We first fix a non-negative function $G \in C^\infty_c ((0,1/2))$ with
\begin{equation} \label{eq:intG}
    \int_{[0,1]} G^2 (t) \dt = 1.
\end{equation}

\subsubsection{Horizontal temporal intermittency functions} \label{subsubsec:htif}
We set
\begin{equation*}
    g_h (t) \coloneqq \kappa_h^{1/2} G(\kappa_h t),
\end{equation*} 
more precisely, $g_h$ is the 1-periodic extension of the right-hand side (where we require that $\kappa_h > 1$). Note that from \eqref{eq:intG} we obtain the normalisation identity
\begin{equation} \label{eq:intgh}
    \int_{[0,1]} g_h^2 \dt = 1,
\end{equation}
and furthermore it is straightforward to verify that 
\begin{equation} \label{eq:est-gh}
    \lVert g_h \rVert_{L^p([0,1])} \lesssim \kappa_h^{1/2 - 1/p},
\end{equation} 
for any $p\in [1,\infty]$. Subsequently, we introduce the temporal correction function
\begin{equation*}
    h_h(t) \coloneqq \int_0^t \big(g_h^2(\tau) - 1\big) \dtau.
\end{equation*}
Due to \eqref{eq:intgh}, $h_h$ is 1-periodic and we have 
\begin{equation} \label{eq:est-hh}
    \lVert h_h \rVert_{L^\infty([0,1])} \leq 1.
\end{equation}

\subsubsection{Vertical temporal intermittency functions} \label{subsubsec:vtif}
The vertical temporal oscillation functions are given by the 1-periodic extension (assuming that $\kappa_v > 1$) of
\begin{equation*}
    g_{v,1}^-(t) \coloneqq \kappa_v^{1/q_2} G(\kappa_v t), \quad g_{v,1}^+(t) \coloneqq \kappa_v^{1-1/q_2} G(\kappa_v t).
\end{equation*}
The corresponding temporal correction function is defined by 
\begin{equation*}
    h_{v,1}(t) \coloneqq \int_0^t \big(g_{v,1}^-(\tau)g_{v,1}^+(\tau) - 1\big) \dtau.
\end{equation*}

In addition to that we need temporal oscillation functions where the argument of $G$ is shifted. Those are defined as the 1-periodic extension of 
\begin{equation*}
    g_{v,2}^-(t) \coloneqq \kappa_v^{1/q_2} G(\kappa_v (t-1/2)), \quad g_{v,2}^+(t) \coloneqq \kappa_v^{1-1/q_2} G(\kappa_v (t-1/2)) 
\end{equation*}
with corresponding correction function
\begin{equation*}
    h_{v,2}(t) \coloneqq \int_0^t \big(g_{v,2}^-(\tau)g_{v,2}^+(\tau) - 1\big) \dtau .
\end{equation*}
Since $G$ has compact support in $(0,1/2)$ and $\kappa_v > 1$, the functions $g_{v,1}^\pm$ and $g_{v,2}^\pm$ have disjoint supports. 

Note that due to the fact that $q_2>2$, we have $1/q_2<1-1/q_2$ which justifies the notation $g_{v,k}^-$, $g_{v,k}^+$ for $k=1,2$.

Similarly to the horizontal temporal functions which we introduced in section~\ref{subsubsec:htif}, we have the following estimates for any $p\in[1,\infty]$ and $k\in\{1,2\}$
\begin{align}
    \lVert g_{v,k}^- \rVert_{L^p([0,1])} &\lesssim \kappa_v^{1/q_2 - 1/p}, \label{eq:est-gvm} \\
    \lVert g_{v,k}^+ \rVert_{L^p([0,1])} &\lesssim \kappa_v^{1-1/q_2 - 1/p}, \label{eq:est-gvp} \\
    \lVert h_{v,k} \rVert_{L^\infty([0,1])} &\leq 1. \label{eq:est-hv}
\end{align} 

Finally in a similar manner one can show that 
\begin{equation} \label{eq:est-gvm-sobolev}
    \lVert g_{v,k}^- \rVert_{W^{n,p}([0,1])} \lesssim \kappa_v^{1/q_2+n-1/p} 
\end{equation}
for any $n\in \mathbb{N}_0$ and $p\in[1,\infty]$.

\section{Velocity perturbation and new Reynolds stress tensor} \label{perturbationsection}

In sections~\ref{perturbationsection}, \ref{perturbationestimates} and \ref{reynoldsestimates} we prove Proposition~\ref{perturbativeproposition} hence we suppose that the assumptions of Proposition~\ref{perturbativeproposition} hold.

The perturbation will be written as
\begingroup
\allowdisplaybreaks
\begin{align}
    \overline{u}_p &= u_{p,h} + u_{c,h} + u_{t,h} , \label{eq:defn-up-bar} \\
    \widetilde{u}_p &= u_{p,v} + u_{c,v} + u_{t,v} , \label{eq:defn-up-tilde} \\
    w_p &= w_{p,v} + w_{t,v}, \label{eq:defn-wp}
\end{align}
\endgroup
where $u_{p,h}$ and $u_{p,v}$ are referred to as the horizontal and vertical principal parts of the perturbation, while $u_{c,h}$, $u_{c,v}$, $u_{t,h}$ and $u_{t,v}$ are referred to as the horizontal and vertical spatial and temporal correctors. 

\begin{remark} 
    We would like to remark that $u_{p,h}, u_{c,h}$ and $u_{t,h}$ do no depend on $z$, while $u_{p,v}, u_{c,v}$ and $u_{t,v}$ are mean-free with respect to $z$. Therefore, the first three are indeed a barotropic perturbation, while the latter three are a baroclinic perturbation. This is already hidden in \eqref{eq:defn-up-bar} and \eqref{eq:defn-up-tilde}.
\end{remark}

In sections~\ref{subsec:new-horizontal} and \ref{subsec:new-vertical} we define the horizontal perturbation $\overline{u}_p$ and the vertical perturbation $\widetilde{u}_p$, respectively. The pressure perturbation $P$ is determined in section~\ref{subsec:new-pressure}. Finally we define the new Reynolds stress tensors $R_{h,1}$ and $R_{v,1}$ in section~\ref{subsec:new-reynolds}.

\subsection{The horizontal perturbation} \label{subsec:new-horizontal}

We begin by constructing the horizontal perturbation which consists (see above) of a principal part $u_{p,h}$, a spatial corrector $u_{c,h}$ and a temporal corrector $u_{t,h}$.

In order to construct $u_{p,h}$, we introduce a cutoff function $\chi$. First we choose $\widetilde{\chi}\in C^\infty([0,\infty))$ to be increasing and satisfying
\begin{equation*} 
    \widetilde{\chi} (\sigma) = \left\{\begin{array}{ll}
    4 \lVert R_h \rVert_{L^1 (L^1)} & \text{ if } 0 \leq \sigma \leq \lVert R_h \rVert_{L^1 (L^1 )}, \\
    4 \sigma & \text{ if } \sigma \geq 2 \lVert R_h \rVert_{L^1 (L^1)}.
    \end{array} \right.
\end{equation*}
Next we define the function
\begin{equation*}
    \chi (x,t) \coloneqq \widetilde{\chi} \big(|R_h(x,t)| \big).
\end{equation*}
It is straightforward to check that $\mathbb{I} - \frac{R_h}{\chi} \in B_{1/2} (\mathbb{I})$ for all $(x,t) \in \mathbb{T}^2 \times [0,T]$. This means in particular that we can evaluate the functions $\Gamma_k$ (see Proposition~\ref{prop:mikado-h}) at $\mathbb{I} - \frac{R_h}{\chi}$.

We now introduce a temporal smooth cutoff function $\theta\in C^\infty([0,T];[0,1])$ which satisfies
\begin{equation} \label{eq:def-theta}
    \theta (t) = \begin{cases}
        1 \quad \text{if } \dist (t, I^c) \geq \tau, \\
        0 \quad \text{if } \dist (t, I^c) \leq \frac{1}{2}\tau,
    \end{cases} 
\end{equation}
in order to achieve the desired property of the supports of the perturbations $\overline{u}_p, \widetilde{u}_p, w_p$. The horizontal principal perturbation is then defined by
\begin{equation} \label{eq:def-uph}
    u_{p,h} (x,t) := \sum_{k \in \Lambda} a_k (x,t) W_k (\sigma_h x),
\end{equation}
where the $W_k$ are given by Proposition~\ref{prop:mikado-h} and the amplitude functions are
\begin{equation} \label{eq:def-ak}
    a_k (x,t) := \theta(t) g_h (\nu_h t) \chi^{1/2}(x,t) \Gamma_k \bigg( \mathbb{I} - \frac{R_h(x,t)}{\chi(x,t)} \bigg).
\end{equation}
Notice that $u_{p,h}$ does not need to be divergence free. To overcome this we define the corrector $u_{c,h}$ as 
\begin{equation} \label{eq:def-uch}
    u_{c,h} \coloneqq \sigma_h^{-1} \sum_{k \in \Lambda} \nabla_h a_k \cdot \Omega_k (\sigma_h x) . 
\end{equation}
Hence
\begin{equation} \label{eq:uph+uch}
    u_{p,h} + u_{c,h} = \sigma_h^{-1} \sum_{k \in \Lambda} \nabla_h \cdot (a_k (x,t) \Omega_k (\sigma_h x)),
\end{equation} 
which implies 
\begin{equation}
    \nabla_h \cdot (u_{p,h} + u_{c,h}) = \sigma_h^{-1} \sum_{k \in \Lambda} (\nabla_h \otimes \nabla_h) : (a_k (x,t) \Omega_k (\sigma_h x)) = 0,
\end{equation}
as $\Omega_k$ is skew-symmetric. Moreover, using the definition of $\theta$ in \eqref{eq:def-theta}, we have $u_{p,h}=u_{c,h}=0$ whenever $\dist (t, I^c) \leq \tau/2$.

Next, we define the horizontal temporal corrector to be 
\begin{equation} \label{eq:def-uth}
    u_{t,h} \coloneqq \nu_h^{-1} h_h (\nu_h t) (\nabla_h \cdot R_h - \nabla_h \Delta_h^{-1} [(\nabla_h \otimes \nabla_h) : R_h] ).
\end{equation}
It is straightforward to check that $(\nabla_h \otimes \nabla_h) : R_h$ is mean-free (so that the inverse Laplacian $\Delta_h^{-1}$ can be applied to this expression), $\nabla_h\cdot u_{t,h}=0$, and that $u_{t,h}=0$ whenever $\dist (t, I^c) \leq \tau/2$. 

Finally notice, that $u_{p,h}$, $u_{c,h}$ and $u_{t,h}$ are indeed independent of $z$.

\subsection{The pressure perturbation} \label{subsec:new-pressure}

The pressure perturbation is defined as follows
\begin{equation} \label{eq:def-P} 
    P \coloneqq - \theta^2 g_h^2 (\nu_h t) \chi + \nu_h^{-1} \Delta_h^{-1} (\nabla_h \otimes \nabla_h) : \partial_t (h_h (\nu_h t) R_h). 
\end{equation}
Note that $\partial_z P=0$, since $R_h$, and hence also $\chi$, are independent of $z$.

\subsection{The vertical perturbation} \label{subsec:new-vertical}

We denote the components of the vertical Reynolds stress tensor as $R_{v,k}$, $k=1,2$, i.e.,
\begin{equation*}
    R_v = \begin{pmatrix}
        R_{v,1} \\
        R_{v,2}
    \end{pmatrix}.
\end{equation*}
This allows us to define the vertical principal perturbation as
\begin{align} 
    u_{p,v}(x,t) &\coloneqq -\sum_{k=1}^2 \frac{g_{v,k}^- (\nu_v t) \theta (t) R_{v,k}(x,t) W_k (\sigma_v x)}{\lVert R_h \rVert_{L^1 (L^1 )}} , \label{eq:def-upv}\\ 
    w_{p,v}(x,t) &\coloneqq \sum_{k=1}^2 g_{v,k}^+ (\nu_v t) \theta (t) \phi_k (\sigma_v x) \lVert R_h \rVert_{L^1 (L^1 )}, \label{eq:def-wpv} 
\end{align}
where $W_k$ and $\phi_k$ are given by Proposition~\ref{prop:mikado-v}.

We now introduce the vertical spatial corrector in order to make the perturbation divergence free. We set
\begin{equation}
    u_{c,v} \coloneqq - \frac{\sigma_v^{-1}}{\lVert R_h \rVert_{L^1 (L^1 )}} \sum_{k=1}^2 \nabla_h \big(g_{v,k}^- (\nu_v t) \theta R_{v,k}\big) \Omega_k (\sigma_v x). 
\end{equation}
Observe that
\begin{equation} \label{eq:upv+ucv}
    u_{p,v} + u_{c,v} = -\frac{\sigma_v^{-1}}{\lVert R_h \rVert_{L^1 (L^1 )}} \sum_{k = 1}^2 \nabla_h \cdot \big(g_{v,k}^- (\nu_v t) \theta R_{v,k} \Omega_k (\sigma_v x)\big), 
\end{equation}
and hence
\begin{equation*}
    \nabla_h\cdot (u_{p,v} + u_{c,v}) = -\frac{\sigma_v^{-1}}{\lVert R_h \rVert_{L^1 (L^1 )}} \sum_{k = 1}^2 (\nabla_h\otimes \nabla_h) : \big(g_{v,k}^- (\nu_v t) \theta R_{v,k} \Omega_k (\sigma_v x)\big) = 0 ,
\end{equation*}
since $\Omega_k$ is skew-symmetric. Notice that $w_{p,v}$ is independent of $z$. Again according to the definition of $\theta$ in \eqref{eq:def-theta}, $u_{p,v}=u_{c,v}=w_{p,v}=0$ whenever $\dist (t, I^c) \leq \tau/2$. 

Moreover, we introduce the vertical temporal corrector to be
\begin{equation} \label{verttemporal1}
    u_{t,v} \coloneqq \nu_v^{-1} \sum_{k=1}^2 h_{v,k} (\nu_v t) \partial_z R_{v,k} \mathbf{e}_k .
\end{equation}
Since $u_{t,v}$ does not need to be divergence free, we introduce the corrector
\begin{equation} \label{verttemporal2}
    w_{t,v} \coloneqq - \nu_v^{-1} \sum_{k=1}^2 h_{v,k} (\nu_v t) \partial_{k} R_{v,k} .
\end{equation}
It is then simple to check that $\nabla_h\cdot u_{t,v} + \partial_z w_{t,v} = 0$. Similar to $u_{t,h}$, see above, we have $u_{t,v}=0$ and $w_{t,v}=0$ whenever $\dist(t,I^c)\leq \tau/2$.

Finally notice, that $u_{p,v}$, $u_{c,v}$ and $u_{t,v}$ are mean-free with respect to $z$ because $R_v$ is mean-free with respect to $z$. 

\subsection{New Reynolds stress tensors} \label{subsec:new-reynolds} 

The goal of this section is to define the new Reynolds stress tensors $R_h$ and $R_v$. These will consist of several pieces.

\subsubsection{Horizontal oscillation error}

Let us first define
\begin{align} 
    R_{\far} &\coloneqq \sum_{k,k' \in \Lambda, k \neq k'} a_{k} a_{k'} W_k (\sigma_h x) \otimes W_{k'} (\sigma_h x), \label{eq:defn-Rfar} \\
    R_{\osc,x,h} &\coloneqq \sum_{k \in \Lambda} \mathcal{B} \bigg( \nabla_h (a_k^2), W_k (\sigma_h x) \otimes W_k (\sigma_h x) - \int_{\mathbb{T}^2} W_k \otimes W_k \dx  \bigg), \label{eq:defn-Roscxh} \\ 
    R_{\osc,t,h} &\coloneqq \nu_h^{-1} h_h (\nu_h t) \partial_t R_h. \label{eq:defn-Roscth} 
\end{align}
where $\mathcal{B}$ is the bilinear inverse divergence operator from section~\ref{subsubsec:bid}. Moreover, we set 
\begin{equation} \label{eq:defn-Rosch}
    R_{\osc,h} = R_{\osc,x,h} + R_{\osc,t,h} + R_{\far}.
\end{equation}

\begin{lemma} \label{lemma:Rosch}
    We have
    \begin{equation} \label{eq:Rosch}
        \partial_t u_{t,h} + \nabla_h \cdot (u_{p,h} \otimes u_{p,h} + R_h) + \nabla_h P = \nabla_h \cdot R_{\osc,h} .
    \end{equation}
\end{lemma}

\begin{proof}
Let us first look at the term $\nabla_h \cdot ( u_{p,h} \otimes u_{p,h} + R_h )$. We may write 
\begin{align*}
    \nabla_h \cdot ( u_{p,h} \otimes u_{p,h} + R_h ) = \nabla_h \cdot \bigg( \sum_{k \in \Lambda} a_k^2 W_k (\sigma_h x) \otimes W_k (\sigma_h x) + R_h \bigg) + \nabla_h \cdot R_{\far}.
\end{align*}
Using the definition of the $a_k$, items 1 and 2 of Proposition~\ref{prop:mikado-h} and Lemma~\ref{lemma:bid} we find
\begingroup
\allowdisplaybreaks
\begin{align*}
    &\nabla_h \cdot \bigg( \sum_{k \in \Lambda} a_k^2 W_k (\sigma_h x) \otimes W_k (\sigma_h x) + R_h \bigg) \\
    &= \nabla_h \cdot \bigg[ \sum_{k \in \Lambda} a_k^2 \bigg( W_k (\sigma_h x) \otimes W_k (\sigma_h x) - \int_{\mathbb{T}^2} W_k \otimes W_k \dx + \int_{\T^2} W_k \otimes W_k \dx \bigg) + R_h \bigg] \\
    &= \sum_{k \in \Lambda} \nabla_h (a_k^2) \cdot \bigg( W_k (\sigma_h x) \otimes W_k (\sigma_h x) - \int_{\mathbb{T}^2} W_k \otimes W_k \dx  \bigg) \\
    &\qquad + \theta^2 g_h^2 (\nu_h t)  \nabla_h \chi + (1  - \theta^2 g_h^2(\nu_h t)) \nabla_h \cdot R_h \\
    &= \nabla_h \cdot R_{\osc,x,h} + \theta^2 g_h^2 (\nu_h t)  \nabla_h \chi + (1  - \theta^2 g_h^2(\nu_h t)) \nabla_h \cdot R_h .
\end{align*}
\endgroup

Next we compute
\begingroup
\allowdisplaybreaks
\begin{align*}
    \partial_t u_{t,h} &= \nu_h^{-1} \partial_t (h_h (\nu_h t)) \nabla_h \cdot R_h + \nu_h^{-1} h_h (\nu_h t) \nabla_h \cdot \partial_t R_h \\
    &\qquad - \nu_h^{-1} \nabla_h \Delta_h^{-1} (\nabla_h \otimes \nabla_h) : \partial_t (h_h (\nu_h t) R_h) ) \\
    &= (g_h^2 (\nu_h t) - 1) \nabla_h \cdot R_h + \nabla_h \cdot R_{\osc,t,h} - \nu_h^{-1} \nabla_h \Delta_h^{-1} (\nabla_h \otimes \nabla_h) : \partial_t (h_h (\nu_h t) R_h) .
\end{align*}
\endgroup

Hence we have shown
\begin{equation*}
    \partial_t u_{t,h} + \nabla_h \cdot (u_{p,h} \otimes u_{p,h} + R_h) + \nabla_h P = \nabla_h \cdot (R_{\osc,x,h} + R_{\osc,t,h} + R_{\far}) + g_h^2 (\nu_h t) (1 - \theta^2) \nabla_h \cdot R_h .
\end{equation*}
If $\dist(t, I^c)\leq \tau$, then $R_h=0$ by well-preparedness. If $\dist(t, I^c)\geq \tau$ then $\theta(t)=1$ and hence $1 - \theta^2=0$. This completes the proof of \eqref{eq:Rosch}.

\end{proof}

\subsubsection{Vertical oscillation error}

We define 
\begin{align}
    R_{\osc,x,v} &\coloneqq - \sum_{k=1}^2 g_{v,k}^- (\nu_v t) g_{v,k}^+ (\nu_v t) \theta^2  R_{v,k} \bigg( \phi_k (\sigma_v x) W_k (\sigma_v x) - \int_{\mathbb{T}^2} \phi_k W_k \dx \bigg), \label{eq:defn-Roscxv} \\
    R_{\osc,t,v} &\coloneqq \nu_v^{-1} \sum_{k=1}^2 h_{v,k} (\nu_v t) \partial_t R_{v,k} \mathbf{e}_k . \label{eq:defn-Rosctv} 
\end{align} 
and set
\begin{equation} \label{eq:defn-Roscv}
    R_{\osc,v} = R_{\osc,x,v} + R_{\osc,t,v}.
\end{equation}

\begin{lemma} \label{lemma:Roscv}
    We have
    \begin{equation} \label{eq:Roscv}
        \partial_t u_{t,v} + \partial_z ( w_{p,v} u_{p,v} + R_v) = \partial_z R_{\osc,v}.
    \end{equation}
\end{lemma}

\begin{proof}
First we observe that 
\begin{equation*}
    w_{p,v} u_{p,v} = - \sum_{k=1}^2 g_{v,k}^- (\nu_v t) g_{v,k}^+ (\nu_v t) \theta^2 R_{v,k} \phi_{k} (\sigma_v x) W_{k} (\sigma_v x),
\end{equation*}
since $g_{v,1}^\pm g_{v,2}^\pm = 0$, see section~\ref{subsubsec:vtif}. Hence we obtain using Proposition~\ref{prop:mikado-v}
\begingroup
\allowdisplaybreaks
\begin{align*}
    &\partial_z (w_{p,v} u_{p,v} + R_v)\\
    &= - \sum_{k=1}^2 g_{v,k}^- (\nu_v t) g_{v,k}^+ (\nu_v t) \theta^2 \partial_z R_{v,k} \bigg( \phi_k (\sigma_v x)  W_k (\sigma_v x) - \int_{\mathbb{T}^2} \phi_k W_k \dx \bigg) \\
    &\qquad - \sum_{k=1}^2 \big( g_{v,k}^- (\nu_v t) g_{v,k}^+ (\nu_v t) - 1 \big) \partial_z R_{v,k} \mathbf{e}_k + \sum_{k=1}^2 g_{v,k}^- (\nu_v t) g_{v,k}^+ (\nu_v t) (1 - \theta^2) \partial_z R_{v,k} \mathbf{e}_k \\ 
    &= \partial_z R_{\osc,x,v} - \sum_{k=1}^2 \big( g_{v,k}^- (\nu_v t) g_{v,k}^+ (\nu_v t) - 1 \big) \partial_z R_{v,k} \mathbf{e}_k .
\end{align*}
\endgroup
Here we used that $(1 - \theta^2) \partial_z R_{v,k}=0$, see the proof of Lemma~\ref{lemma:Rosch}. Moreover a straightforward computation shows 
\begin{align*}
    \partial_t u_{t,v} - \sum_{k=1}^2 \big( g_{v,k}^- (\nu_v t) g_{v,k}^+ (\nu_v t) - 1 \big) \partial_z R_{v,k} \mathbf{e}_k &= \nu_v^{-1} \sum_{k=1}^2 h_{v,k}(\nu_v t) \partial_t \partial_z R_{v,k} \mathbf{e}_k \\
    &= \partial_z R_{\osc,t,v},
\end{align*}
which finishes the proof of Lemma~\ref{lemma:Roscv}.
\end{proof}

\subsubsection{Linear errors} \label{subsubsec:lin-errors}

Next we define the horizontal and vertical linear errors by 
\begingroup
\allowdisplaybreaks
\begin{align*} 
    R_{\lin,h} \coloneqq \mathcal{R}_h\bigg[ &\partial_t \big(u_{p,h}+u_{c,h}\big) + \nabla_h \cdot \bigg( \overline{u}\otimes \big(u_{p,h}+u_{c,h}+u_{t,h}\big) + \overline{u\otimes\big(u_{p,v}+u_{c,v}+u_{t,v}\big)} \\
    &\qquad\qquad + \big(u_{p,h}+u_{c,h}+u_{t,h}\big)\otimes \overline{u} + \overline{\big(u_{p,v}+u_{c,v}+u_{t,v}\big)\otimes u} \bigg) \bigg],
\end{align*}
\endgroup
and 
\begin{align*} 
    R_{\lin,v} \coloneqq \mathcal{R}_v\bigg[ &\partial_t \big(u_{p,v}+u_{c,v}\big) + \nabla_h \cdot \bigg( \widetilde{u}\otimes \big(u_{p,h}+u_{c,h}+u_{t,h}\big) + \reallywidetilde{u\otimes\big(u_{p,v}+u_{c,v}+u_{t,v}\big)} \\
    &\qquad\qquad + \big(u_{p,h}+u_{c,h}+u_{t,h}\big)\otimes \widetilde{u} + \reallywidetilde{\big(u_{p,v}+u_{c,v}+u_{t,v}\big)\otimes u} \bigg) \\
    &\qquad + \partial_z \bigg( w\big(u_{p,h}+u_{c,h}+u_{t,h} + u_{p,v}+u_{c,v}+u_{t,v}\big) + \big(w_{p,v}+w_{t,v}\big) u \bigg) \bigg].
\end{align*}
Note that the arguments of the operators $\mathcal{R}_h$ and $\mathcal{R}_v$ satisfy the required properties, i.e. they are independent of $z$ and mean-free with respect to $z$, respectively.

For convenience let us write
\begin{align*} 
    R_{\lin,t,h} &\coloneqq \mathcal{R}_h \partial_t \big(u_{p,h}+u_{c,h}\big), \\
    R_{\lin,t,v} &\coloneqq \mathcal{R}_v \partial_t \big(u_{p,v}+u_{c,v}\big),
\end{align*}
and 
\begin{align*} 
    R_{\lin,x,h} &\coloneqq R_{\lin,h} - R_{\lin,t,h}, \\
    R_{\lin,x,v} &\coloneqq R_{\lin,v} - R_{\lin,t,v}.
\end{align*}

\subsubsection{Corrector errors} \label{subsubsec:cor-errors}

Finally, we define the horizontal and vertical corrector errors by 
\begin{align*} 
    R_{\cor,h} \coloneqq \mathcal{R}_h\bigg[ &\nabla_h \cdot \bigg( \big(u_{c,h}+u_{t,h}\big) \otimes\big(u_{c,h}+u_{t,h}\big) + u_{p,h} \otimes\big(u_{c,h}+u_{t,h}\big) + \big(u_{c,h}+u_{t,h}\big) \otimes u_{p,h} \\
    &\qquad\qquad + \overline{\big(u_{p,v}+u_{c,v}+u_{t,v}\big) \otimes\big(u_{p,v}+u_{c,v}+u_{t,v}\big)} \bigg)\bigg]
\end{align*}
and
\begin{align*} 
    R_{\cor,v} \coloneqq \mathcal{R}_v\bigg[ &\nabla_h \cdot \bigg( \reallywidetilde{\big(u_{p,v}+u_{c,v}+u_{t,v}\big) \otimes\big(u_{p,v}+u_{c,v}+u_{t,v}\big)} \\
    &\qquad\qquad + \big(u_{p,h}+u_{c,h}+u_{t,h}\big) \otimes\big(u_{p,v}+u_{c,v}+u_{t,v}\big) \\
    &\qquad\qquad + \big(u_{p,v}+u_{c,v}+u_{t,v}\big) \otimes\big(u_{p,h}+u_{c,h}+u_{t,h}\big) \bigg) \\
    &\qquad + \partial_z \bigg( w_{t,v} \big(u_{p,h}+u_{c,h}+u_{t,h} + u_{p,v}+u_{c,v}+u_{t,v}\big) \\
    &\qquad\qquad + w_{p,v} \big(u_{c,v}+u_{t,v}\big) \bigg) \bigg].
\end{align*} 
As in section~\ref{subsubsec:lin-errors} we remark that the arguments of the operators $\mathcal{R}_h$ and $\mathcal{R}_v$ satisfy the required properties, i.e. they are independent of $z$ and mean-free with respect to $z$, respectively.

\subsubsection{Conclusion} 

The new Reynolds stress tensors $R_{h,1}, R_{v,1}$ are then given by
\begin{align*}
    R_{h,1} \coloneqq R_{\osc,h} + R_{\lin,h} + R_{\cor,h},\qquad R_{v,1} \coloneqq R_{\osc,v} + R_{\lin,v} + R_{\cor,v}.
\end{align*}

First, we note that by definition $R_{\osc,h}$, $R_{\lin,h}$ and $R_{\cor,h}$ (and consequently also $R_{h,1}$) are independent of $z$. Moreover, by definition $R_{\osc,v}$ is mean-free with respect to $z$, and so are $R_{\lin,v}$ and $R_{\cor,v}$ according to Lemma~\ref{Lpverticalbound} (consequently $R_{v,1}$ has the same property).

Next we remark that $R_{h,1} (x,t) = R_{v,1} (x,t) = 0$ whenever $\dist (t, I^c) \leq \frac{\tau}{2}$. Indeed, for the oscillation errors $R_{\osc,h}$, $R_{\osc,v}$ this follows from the fact that $R_h=R_v=0$ whenever $\dist(t,I^c)\leq \tau$, and the definition of $\theta$, see \eqref{eq:def-theta}. The fact that $R_{\lin,h}=0$, $R_{\lin,v}=0$, $R_{\cor,h}=0$, and $R_{\cor,v}=0$ whenever $\dist (t, I^c) \leq \frac{\tau}{2}$ immediately follows from item 1 of Proposition~\ref{perturbativeproposition}, which we already proved in section~\ref{subsec:new-horizontal} and \ref{subsec:new-vertical}

Finally, with the help of 
\begin{itemize}
    \item the fact that $\nabla_h\cdot (u+\overline{u}_p+\widetilde{u}_p) + \partial_z (w+w_p)=0$, 
    \item Lemmas~\ref{lemma:Rosch} and \ref{lemma:Roscv}, 
    \item Lemmas~\ref{lemma:hid} and \ref{Lpverticalbound}, 
    \item the fact that $\int_{\T^2} \big(u_{p,h} + u_{c,h}\big) \dx =0$, see \eqref{eq:uph+uch}, 
    \item and $\partial_z u_{p,h} = \partial_z u_{c,h} = \partial_z u_{t,h} = \partial_z w_{p,v} = 0$, see section~\ref{subsec:new-horizontal} and section~\ref{subsec:new-vertical}, 
\end{itemize}    
a long but straightforward computation shows 
\begin{align*}
    &\partial_t (u+\overline{u}_p+\widetilde{u}_p) + (u+\overline{u}_p+\widetilde{u}_p)\cdot \nabla_h (u+\overline{u}_p+\widetilde{u}_p) + (w+w_p) \partial_z (u+\overline{u}_p+\widetilde{u}_p) + \nabla_h (p+P) \\
    &= \partial_t \big(\overline{u}_p+\widetilde{u}_p\big) + \nabla_h \cdot \Big( u\otimes \big(\overline{u}_p+\widetilde{u}_p\big) + \big(\overline{u}_p+\widetilde{u}_p\big)\otimes u + \big(\overline{u}_p+\widetilde{u}_p\big)\otimes\big(\overline{u}_p+\widetilde{u}_p\big) \Big) \\
    &\qquad + \partial_z \Big( w\big(\overline{u}_p+\widetilde{u}_p\big) + w_p u + w_p \big(\overline{u}_p+\widetilde{u}_p\big)\Big) + \nabla_h P + \nabla_h \cdot R_h + \partial_z R_v \\
    &= \nabla_h \cdot R_{h,1} + \partial_z R_{v,1}.
\end{align*}
Hence $(u + \overline{u}_p + \widetilde{u}_p, w + w_p, p + P, R_{h,1}, R_{v,1})$ solves \eqref{eq:hyd-er-u}.

\section{Estimates on the perturbation} \label{perturbationestimates}

In the remaining sections we will use the following convention: for quantities $Q_1$ and $Q_2$ we write $Q_1 \lesssim Q_2$ if there exists a constant $C$ such that $Q_1 \leq C Q_2$. In general we require that $C$ does not depend on $(u,w,p,R_h,R_v)$. However if the right-hand side $Q_2$ only contains powers of the parameters $\mu_i,\sigma_i,\kappa_i,\nu_i,\lambda_i$ ($i=h,v$), then the implicit constant $C$ may depend on $(u,w,p,R_h,R_v)$. 

\subsection{Principal perturbation}

\subsubsection{Horizontal principal perturbation} \label{subsubsec:hpp}
First, we estimate the horizontal part of the principal perturbation. We recall the following Lemma from \cite{cheskidovluo3}.
\begin{lemma} \label{coefficientlemma}
    We have the following estimates for all $n,m\in \N_0$, $p\in [1,\infty]$
    \begingroup
    \allowdisplaybreaks
    \begin{align}
        \lVert \partial^n_t \nabla^m a_k \rVert_{L^p (L^\infty)} &\lesssim (\nu_h \kappa_h)^n \kappa_h^{1/2-1/p}, \label{eq:est-ak-1} \\ 
        \lVert a_k (t) \rVert_{L^2} &\lesssim \theta(t) g_h (\nu_h t) \bigg( \int_{\mathbb{T}^2} \chi (x,t) \dx \bigg)^{1/2}, \quad \text{ for any }t\in [0,T]. \label{eq:est-ak-2} 
    \end{align}
    \endgroup
    The implicit constant in \eqref{eq:est-ak-1} might depend on $u$ or $R_h$, whereas the implicit constant in \eqref{eq:est-ak-2} neither depends on $t$ nor on $u$ or $R_h$.
\end{lemma}
For the proof we refer to \cite[Lemma~5.2]{cheskidovluo3}.

With Lemma~\ref{coefficientlemma} at hand we are ready to prove the estimates on the horizontal principal perturbation. 
\begin{lemma} \label{horizontalprinc} 
    If $\lambda_h$ is chosen sufficiently large (depending on $R_h$), then the horizontal principal perturbation satisfies the following estimates
    \begin{align}
        \lVert u_{p,h} \rVert_{L^2 (L^2)} &\lesssim \lVert R_h \rVert_{L^1(L^1)}^{1/2}, \label{eq:est-uph-1} \\ 
        \lVert u_{p,h} \rVert_{L^{q_1} (H^{s_1})} &\lesssim \lambda_h^{-\gamma_h}. \label{eq:est-uph-2}
    \end{align}
\end{lemma}

\begin{proof}
By applying the improved Hölder inequality in Lemma~\ref{holderlemma} and Lemmas \ref{mikadobounds} and \ref{coefficientlemma}, we find that
\begin{align*}
    \lVert u_{p,h} \rVert_{L^2 (L^2)} &\leq \sum_{k \in \Lambda} \lVert a_k  W_k(\sigma_h\cdot) \rVert_{L^2(L^2)} \lesssim \sum_{k \in \Lambda} \bigg( \lVert a_k \rVert_{L^2 (L^2)} \lVert W_k \rVert_{L^2} + \sigma_h^{-1/2} \lVert a_k \rVert_{L^2 (C^1)} \lVert W_k \rVert_{L^2} \bigg) \\
    &\lesssim \bigg\| g_h(\nu_h \cdot) \bigg( \int_{\mathbb{T}^2} \chi (\cdot,x) \dx \bigg)^{1/2} \bigg\|_{L^2} + C_{u,R_h} \sigma^{-1/2}_h 
\end{align*} 
with a constant $C_{u,R_h}$ depending on $u,R_h$. Since $t\mapsto \big( \int_{\mathbb{T}^2} \chi (\cdot,x) \dx \big)^{1/2}$ is smooth, we can apply Lemma~\ref{holderlemma} once again to obtain
\begin{align*}
    &\bigg\| g_h(\nu_h \cdot) \bigg( \int_{\mathbb{T}^2} \chi (\cdot,x) \dx \bigg)^{1/2} \bigg\|_{L^2} \\
    &\lesssim \bigg\| \bigg( \int_{\mathbb{T}^2} \chi (\cdot,x) \dx \bigg)^{1/2} \bigg\|_{L^2} \|g_h\|_{L^2} + \nu_h^{-1/2} \bigg\| \bigg( \int_{\mathbb{T}^2} \chi (\cdot,x) \dx \bigg)^{1/2} \bigg\|_{C^1} \|g_h\|_{L^2} \\
    &\lesssim \|\chi\|_{L^1(L^1)}^{1/2} + C_{u,R_h} \nu_h^{-1/2},
\end{align*}
where we made use of \eqref{eq:intgh}. Since $\chi(x,t)\leq 4|R_h(x,t)| + 4\|R_h\|_{L^1(L^1)}$, we have 
\begin{align*}
    \|\chi\|_{L^1(L^1)} \lesssim \|R_h\|_{L^1(L^1)} + \|R_h\|_{L^1(L^1)} \lesssim \|R_h\|_{L^1(L^1)}.
\end{align*}
So have shown that 
\begin{align*}
    \lVert u_{p,h} \rVert_{L^2 (L^2)} &\lesssim \|R_h\|_{L^1(L^1)}^{1/2} +  C_{u,R_h} \nu_h^{-1/2} + C_{u,R_h} \sigma^{-1/2}_h,
\end{align*}
which implies \eqref{eq:est-uph-1} by taking $\lambda_h$ sufficiently large (depending on $R_h$).

Furthermore by applying Lemmas \ref{parameterlemma}, \ref{mikadobounds} and \ref{coefficientlemma} we find that
\begin{align*}
    \lVert u_{p,h} \rVert_{L^{q_1} (H^{s_1})} &\leq \sum_{k \in \Lambda} \lVert a_k \lVert_{L^{q_1} (W^{1,\infty})} \lVert  W_k(\sigma_h\cdot) \rVert_{H^{s_1}} \lesssim  \kappa_h^{1/2 - 1/q_1} (\sigma_h \mu_h)^{s_1} \leq \lambda_h^{-\gamma_h}
\end{align*}
which proves \eqref{eq:est-uph-2}.
\end{proof}

\subsubsection{Vertical principal perturbation} \label{subsubsec:vpp}

Let us first show the following Lemma. 

\begin{lemma} \label{lemma:prod}
    For all $1\leq p\leq \infty$ and $k\in \{1,2\}$ we have
    \begin{equation}
        \bigg\lVert \bigg( \phi_k (\sigma_v \cdot) W_k (\sigma_v \cdot) - \int_{\mathbb{T}^2} \phi_k ( x) W_k (x) \dx \bigg) R_{v,k} \bigg\rVert_{L^p (B^{-1}_{1,\infty})} \lesssim \sigma_v^{-1} .
    \end{equation}
\end{lemma}

\begin{proof}
Using equation \eqref{eq:lem-Rh-id1} we find
\begingroup
\allowdisplaybreaks
\begin{align*}
    &\bigg( \phi_k (\sigma_v x) W_k (\sigma_v x) - \int_{\mathbb{T}^2} \phi_k ( x) W_k (x) \dx \bigg) R_{v,k} \\
    &= R_{v,k} \nabla_h \cdot \bigg[ \mathcal{R}_h \bigg( \phi_k (\sigma_v x) W_k (\sigma_v x) - \int_{\mathbb{T}^2} \phi_k W_k \dx \bigg) \bigg] \\
    &= \nabla_h \cdot \bigg[ R_{v,k} \mathcal{R}_h \bigg( \phi_k (\sigma_v x) W_k (\sigma_v x) - \int_{\mathbb{T}^2} \phi_k W_k \dx \bigg) \bigg] \\
    &\qquad - \nabla_h R_{v,k} \cdot \mathcal{R}_h \bigg( \phi_k (\sigma_v x) W_k (\sigma_v x) - \int_{\mathbb{T}^2} \phi_k W_k \dx \bigg) . 
\end{align*}
\endgroup
Next, we observe that according to Lemma~\ref{mikadobounds}
\begin{equation} \label{eq:prod1}
    \bigg\lVert \phi_k W_k - \int_{\mathbb{T}^2} \phi_k ( x) W_k (x) \dx \bigg\rVert_{L^1} \lesssim \lVert \phi_k W_k\rVert_{L^1} + 1 \leq \| \phi_k\|_{L^2} \|W_k\|_{L^2} + 1 \lesssim 1.
\end{equation}
Then by Lemmas~\ref{lemma:hid} and \ref{lemma:essential-besov}, and inequality \eqref{eq:prod1} we obtain 
\begingroup 
\allowdisplaybreaks
\begin{align*}
    &\bigg\lVert \bigg( \phi_k (\sigma_v \cdot) W_k (\sigma_v \cdot) - \int_{\mathbb{T}^2} \phi_k ( x) W_k (x) \dx \bigg) R_{v,k} \bigg\rVert_{L^p (B^{-1}_{1,\infty})} \\
    &\lesssim \bigg\lVert \nabla_h \cdot \bigg[ R_{v,k} \mathcal{R}_h \bigg( \phi_k (\sigma_v \cdot) W_k (\sigma_v \cdot) - \int_{\mathbb{T}^2} \phi_k W_k \dx \bigg) \bigg]  \bigg\rVert_{L^p (B^{-1}_{1,\infty})} \\
    &\qquad + \bigg\lVert \nabla_h R_{v,k} \cdot \mathcal{R}_h \bigg( \phi_k (\sigma_v \cdot) W_k (\sigma_v \cdot) - \int_{\mathbb{T}^2} \phi_k W_k \dx \bigg) \bigg\rVert_{L^p (B^{-1}_{1,\infty})} \\
    &\lesssim \bigg\lVert R_{v,k} \mathcal{R}_h \bigg( \phi_k (\sigma_v \cdot) W_k (\sigma_v \cdot) - \int_{\mathbb{T}^2} \phi_k W_k \dx \bigg) \bigg\rVert_{L^p (L^1)} \\
    &\qquad + \bigg\lVert \nabla_h R_{v,k} \cdot \mathcal{R}_h \bigg( \phi_k (\sigma_v \cdot) W_k (\sigma_v \cdot) - \int_{\mathbb{T}^2} \phi_k W_k \dx \bigg) \bigg\rVert_{L^p (L^1)} \\
    &\lesssim \Big(\lVert R_{v,k} \rVert_{L^p(L^\infty)} + \lVert \nabla_h R_{v,k} \rVert_{L^p(L^\infty)}\Big) \bigg\lVert \mathcal{R}_h \bigg( \phi_k (\sigma_v \cdot) W_k (\sigma_v \cdot) - \int_{\mathbb{T}^2} \phi_k W_k \dx \bigg) \bigg\rVert_{L^1} \\
    &\lesssim \sigma_v^{-1} C_{R_v} \bigg\lVert \phi_k  W_k - \int_{\mathbb{T}^2} \phi_k W_k \dx \bigg\rVert_{L^1} \lesssim \sigma_v^{-1}.
\end{align*}
\endgroup 
\end{proof}

\begin{remark}
    We present an alternative proof of Lemma~\ref{lemma:prod} in the appendix, see section~\ref{subsec:ap-littlewood-alternative}. 
\end{remark}

Next we estimate the vertical principal perturbation.
\begin{lemma} \label{verticalprinc}
If $\lambda_v$ is chosen sufficiently large (depending on $R_v$), then the vertical principal perturbation satisfies the following inequalities
\begingroup 
\allowdisplaybreaks
\begin{align}
    \lVert u_{p,v} \rVert_{L^{q_3-} (H^{s_3})} &\lesssim \lambda_v^{-\gamma_v}, \label{eq:est-upv-1} \\
    \lVert u_{p,v} \rVert_{L^{q_2-} (L^2)} &\lesssim \lambda_v^{-\gamma_v}, \label{eq:est-upv-2} \\
    \lVert w_{p,v} \rVert_{L^{q_3'-} (H^{-s_3})} &\lesssim \lambda_v^{-\gamma_v}, \label{eq:est-wpv-1} \\ 
    \lVert w_{p,v} \rVert_{L^{q_2'-} (L^2)} &\lesssim \lambda_v^{-\gamma_v}, \label{eq:est-wpv-2} \\
    \lVert w_{p,v} \rVert_{L^{q_3'} (H^{-s_3})} &\lesssim \lVert R_h \rVert_{L^1(L^1)}, \label{eq:est-wpv-1e} \\ 
    \lVert w_{p,v} \rVert_{L^{q_2'} (L^2)} &\lesssim \lVert R_h \rVert_{L^1(L^1)}, \label{eq:est-wpv-2e} \\
    \lVert w_{p,v} u_{p,v} \rVert_{L^1 (B^{-1}_{1,\infty} )} &\lesssim \lVert R_v \rVert_{L^1 (B^{-1}_{1,\infty})}. \label{eq:est-wpvupv} 
\end{align}
\endgroup
\end{lemma}

\begin{proof}
According to \eqref{eq:est-gvm} and Lemmas~\ref{parameterlemma} and \ref{mikadobounds}, we obtain 
\begin{align*}
    \lVert u_{p,v} \rVert_{L^{q_3-} (H^{s_3})} &\leq \frac{1}{\lVert R_h \rVert_{L^1 (L^1)}} \sum_{k=1}^2 \lVert g_{v,k}^- (\nu_v \cdot) \rVert_{L^{q_3-}} \lVert R_v \rVert_{L^\infty (W^{1,\infty})} \lVert W_k (\sigma_v \cdot ) \rVert_{H^{s_3}} \\
    &\lesssim \kappa_v^{1/q_2 - 1/q_3 - \delta} (\sigma_v \mu_v)^{s_3} = \kappa_v^{-\delta} \leq \lambda_v^{-\gamma_v}.
\end{align*}
Similarly, 
\begin{align}
    \lVert u_{p,v} \rVert_{L^{q_2-} (L^2)} &\leq \frac{1}{\lVert R_h \rVert_{L^1 (L^1)}} \sum_{k=1}^2 \lVert g_{v,k}^- (\nu_v \cdot) \rVert_{L^{q_2-}} \lVert R_v \rVert_{L^\infty (L^{\infty})} \lVert W_k (\sigma_v \cdot ) \rVert_{L^2} \notag\\
    &\lesssim \kappa_v^{1/q_2 - 1/q_2 - \delta} = \kappa_v^{-\delta} \leq \lambda_v^{-\gamma_v}. \label{eq:513}
\end{align}
So we have shown \eqref{eq:est-upv-1}, \eqref{eq:est-upv-2}.

Next, notice that in accordance with Proposition~\ref{prop:mikado-v} and Lemma~\ref{mikadobounds} (keeping in mind that $s_3\leq 1$ and hence $-s_3+1\geq 0$)
\begin{align*}
    \lVert \phi_k (\sigma_v \cdot ) \rVert_{H^{-s_3}} &= \lVert W_k (\sigma_v \cdot ) \rVert_{H^{-s_3}} = \sigma_v^{-1} \big\lVert \nabla_h\cdot \big[\Omega_k (\sigma_v \cdot ) \big]\big\rVert_{H^{-s_3}} \\
    &\lesssim \sigma_v^{-1} \lVert \Omega_k (\sigma_v \cdot ) \rVert_{H^{-s_3+1}} \lesssim \sigma_v^{-1} (\sigma_v \mu_v)^{-s_3+1} \mu_v^{-1} = (\sigma_v \mu_v)^{-s_3}.
\end{align*}
Together with \eqref{eq:est-gvp} and Lemma~\ref{parameterlemma} this yields
\begin{align}
    \lVert w_{p,v} \rVert_{L^{q_3'-} (H^{-s_3})} &\leq \sum_{k=1}^2 \lVert \phi_k (\sigma_v \cdot ) \rVert_{H^{-s_3}} \lVert R_h \rVert_{L^1(L^1)} \lVert g_{v,k}^+ (\nu_v \cdot) \rVert_{L^{q_3'-}} \notag\\
    &\lesssim \kappa_v^{1 - 1/q_2 - 1/q_3' -\delta} (\sigma_v \mu_v)^{-s_3} = \kappa_v^{-\delta} \leq \lambda_v^{-\gamma_v}. \label{eq:514}
\end{align} 
Similarly, using Lemma~\ref{mikadobounds}, we obtain
\begin{align}
    \lVert w_{p,v} \rVert_{L^{q_2'-} (L^2)} &\leq \sum_{k=1}^2 \lVert \phi_k (\sigma_v \cdot ) \rVert_{L^2} \lVert R_h \rVert_{L^1(L^1)} \lVert g_{v,k}^+ (\nu_v \cdot) \rVert_{L^{q_2'-}} \notag\\
    &\lesssim \kappa_v^{1 - 1/q_2 - 1/q_2' -\delta} = \kappa_v^{-\delta} \leq \lambda_v^{-\gamma_v}. \label{eq:515}
\end{align} 
Hence we have proven \eqref{eq:est-wpv-1}, \eqref{eq:est-wpv-2}. Additionally, from \eqref{eq:514} and \eqref{eq:515} we see that \eqref{eq:est-wpv-1e} and \eqref{eq:est-wpv-2e} hold.

Finally, we derive estimate \eqref{eq:est-wpvupv} for the product $u_{p,v} w_{p,v}$. Because $g_{v,1}^\pm g_{v,2}^\pm = 0$, see section~\ref{subsubsec:vtif}, and the improved H\"older inequality (Lemma \ref{holderlemma}) we have 
\begin{align*}
    \lVert w_{p,v} u_{p,v} \rVert_{L^1 (B^{-1}_{1,\infty} )} &\lesssim \sum_{k=1}^2 \Big\lVert g_{v,k}^- (\nu_v \cdot) g_{v,k}^+ (\nu_v \cdot) \phi_k (\sigma_v \cdot) W_k (\sigma_v \cdot) R_{v,k} \Big\rVert_{L^1 (B^{-1}_{1,\infty})} \\
    &\lesssim \sum_{k=1}^2 \lVert g_{v,k}^- g_{v,k}^+ \rVert_{L^1} \Big\lVert \phi_k (\sigma_v \cdot) W_k (\sigma_v \cdot) R_{v,k} \Big\rVert_{L^1 (B^{-1}_{1,\infty})} \\
    &\qquad + \sum_{k=1}^2 \nu_v^{-1} \lVert g_{v,k}^- g_{v,k}^+ \rVert_{L^1} \Big\lVert \phi_k (\sigma_v \cdot) W_k (\sigma_v \cdot) R_{v,k} \Big\rVert_{C^1 (B^{-1}_{1,\infty})}.
\end{align*}
First, observe that 
\begin{equation} \label{eq:l1-gvpgvm}
    \lVert g_{v,k}^- g_{v,k}^+ \rVert_{L^1} \leq \lVert g_{v,k}^- \rVert_{L^2} \lVert g_{v,k}^+ \rVert_{L^2} \lesssim \kappa_v^{1/q_2 - 1/2 +1 -1/q_2 -1/2} = 1,
\end{equation}
according to \eqref{eq:est-gvm}, \eqref{eq:est-gvp}. Next, we estimate 
\begin{align*}
    \Big\lVert \phi_k (\sigma_v \cdot) W_k (\sigma_v \cdot) R_{v,k} \Big\rVert_{C^1 (B^{-1}_{1,\infty})} &\lesssim \Big\lVert \phi_k (\sigma_v \cdot) W_k (\sigma_v \cdot) R_{v,k} \Big\rVert_{C^1 (L^1)} \\
    &\lesssim \lVert R_{v,k} \rVert_{C^1(L^\infty)} \lVert \phi_k (\sigma_v \cdot) W_k (\sigma_v \cdot) \rVert_{L^1} \\
    &\lesssim C_{R_v} \lVert \phi_k (\sigma_v \cdot) \rVert_{L^2} \lVert W_k (\sigma_v \cdot) \rVert_{L^2} \\
    &\lesssim C_{R_v}, 
\end{align*} 
where $C_{R_v}$ is a constant depending on $R_v$, and where we used Lemmas~\ref{mikadobounds} and \ref{lemma:essential-besov}. Moreover, we obtain by Lemma~\ref{lemma:prod}
\begin{align*}
    &\Big\lVert \phi_k (\sigma_v \cdot) W_k (\sigma_v \cdot) R_{v,k} \Big\rVert_{L^1 (B^{-1}_{1,\infty})} \\
    &\leq \bigg\lVert \bigg( \phi_k (\sigma_v \cdot) W_k (\sigma_v \cdot) - \int_{\mathbb{T}^2} \phi_k ( x) W_k (x) \dx \bigg) R_{v,k} \bigg\rVert_{L^1 (B^{-1}_{1,\infty})} \\
    &\qquad + \bigg\lvert \int_{\mathbb{T}^2 } \phi_k (x) W_k (x) \dx \bigg\rvert \lVert R_v \rVert_{L^1 (B^{-1}_{1,\infty})} \\
    & \lesssim \sigma_v^{-1} + \lVert R_v \rVert_{L^1 (B^{-1}_{1,\infty})}.
\end{align*}

Hence we have shown 
\begin{equation} \label{eq:prod2}
    \lVert w_{p,v} u_{p,v} \rVert_{L^1 (B^{-1}_{1,\infty} )} \lesssim \|R_v\|_{L^1(B^{-1}_{1,\infty})} + C_{R_v} \Big( \sigma_v^{-1} + \nu_v^{-1} \Big)
\end{equation}
which implies \eqref{eq:est-wpvupv} by choosing $\lambda_v$ sufficiently large (depending on $R_v$).
\end{proof}

\begin{remark} \label{rem:uptilde-endpoint-proof}
    As already mentioned in Remark~\ref{rem:uptilde-endpoint} we can establish \eqref{eq:main-utilde1e}, \eqref{eq:main-utilde2e} instead of \eqref{eq:main-w1e}, \eqref{eq:main-w2e}. To this end we have to replace \eqref{eq:def-upv} and \eqref{eq:def-wpv} by
    \begingroup
    \allowdisplaybreaks
    \begin{align*} 
        u_{p,v}(x,t) &\coloneqq -\sum_{k=1}^2 g_{v,k}^- (\nu_v t) \theta (t) R_{v,k}(x,t) W_k (\sigma_v x)\frac{\lVert R_h \rVert_{L^1 (L^1 )}}{\lVert R_v \rVert_{L^\infty (W^{1,\infty} )}} , \\ 
        w_{p,v}(x,t) &\coloneqq \sum_{k=1}^2 g_{v,k}^+ (\nu_v t) \theta (t) \phi_k (\sigma_v x) \frac{\lVert R_v \rVert_{L^\infty (W^{1,\infty} )}}{\lVert R_h \rVert_{L^1 (L^1 )}}.
    \end{align*}
    \endgroup
    Then \eqref{eq:est-wpv-1e}, \eqref{eq:est-wpv-2e} are no longer true. Instead we find
    \begin{align*}
        \lVert u_{p,v} \rVert_{L^{q_3} (H^{s_3})} &\leq \frac{\lVert R_h \rVert_{L^1 (L^1 )}}{\lVert R_v \rVert_{L^\infty (W^{1,\infty} )}} \sum_{k=1}^2 \lVert g_{v,k}^- (\nu_v \cdot) \rVert_{L^{q_3}} \lVert R_v \rVert_{L^\infty (W^{1,\infty})} \lVert W_k (\sigma_v \cdot ) \rVert_{H^{s_3}} \\
        &\lesssim \lVert R_h \rVert_{L^1 (L^1 )}\kappa_v^{1/q_2 - 1/q_3} (\sigma_v \mu_v)^{s_3} = \lVert R_h \rVert_{L^1 (L^1 )} , \\
        \lVert u_{p,v} \rVert_{L^{q_2} (L^2)} &\leq \frac{\lVert R_h \rVert_{L^1 (L^1 )}}{\lVert R_v \rVert_{L^\infty (W^{1,\infty} )}} \sum_{k=1}^2 \lVert g_{v,k}^- (\nu_v \cdot) \rVert_{L^{q_2}} \lVert R_v \rVert_{L^\infty (L^{\infty})} \lVert W_k (\sigma_v \cdot ) \rVert_{L^2} \\
        &\lesssim \lVert R_h \rVert_{L^1 (L^1 )}\kappa_v^{1/q_2 - 1/q_2} = \lVert R_h \rVert_{L^1 (L^1 )}. 
    \end{align*}
    Further modifications are straightforward.
\end{remark}

\subsection{Spatial correctors}
\begin{lemma} \label{correctorlemma}
    The spatial correctors satisfy the following estimates
    \begin{align*}
        \lVert u_{c,h} \rVert_{L^{q_1} (H^{s_1})} + \lVert u_{c,h} \rVert_{L^2(L^\infty)} &\lesssim \lambda_h^{-\gamma_h}, \\ 
        \lVert u_{c,v} \rVert_{L^{q_2} (L^\infty)} + \lVert u_{c,v} \rVert_{L^{q_3} (H^{s_3})} &\lesssim \lambda_v^{-\gamma_v}. 
    \end{align*}
\end{lemma}

\begin{proof}
By using Lemmas \ref{parameterlemma}, \ref{mikadobounds} and \ref{coefficientlemma} as well as estimate \eqref{eq:est-gvm}, one gets that
\begingroup
\allowdisplaybreaks
\begin{align*}
    \lVert u_{c,h} \rVert_{L^{q_1} (H^{s_1})} &\leq \sigma_h^{-1} \sum_{k \in \Lambda} \lVert \nabla a_k \rVert_{L^{q_1} (W^{1,\infty})} \lVert \Omega_k (\sigma_h \cdot ) \rVert_{H^{s_1}} \lesssim \sigma_h^{-1} \kappa_h^{1/2 - 1/q_1} (\sigma_h \mu_h)^{s_1} \mu_h^{-1} \lesssim \lambda_h^{-\gamma_h}, \\
    \lVert u_{c,h} \rVert_{L^2 (L^\infty)} &\leq \sigma_h^{-1} \sum_{k \in \Lambda} \lVert \nabla a_k \rVert_{L^{2} (L^\infty)} \lVert \Omega_k (\sigma_h \cdot ) \rVert_{L^\infty} \lesssim \sigma_h^{-1} \mu_h^{-1/2} \lesssim \lambda_h^{-\gamma_h}, \\ 
    \lVert u_{c,v} \rVert_{L^{q_2} (L^\infty)} &\leq \frac{\sigma_v^{-1}}{\|R_h\|_{L^1(L^1)}} \sum_{k=1}^2 \lVert g_{v,k}^- (\nu_v \cdot ) \rVert_{L^{q_2}} \lVert \nabla_h R_{v,k} \rVert_{L^\infty (L^\infty)} \lVert \Omega_k (\sigma_v \cdot) \rVert_{L^\infty} \\
    &\lesssim \sigma_v^{-1} \mu_v^{-1/2}  \lesssim \lambda_v^{-\gamma_v}, \\ 
    \lVert u_{c,v} \rVert_{L^{q_3} (H^{s_3})} &\leq \frac{\sigma_v^{-1}}{\|R_h\|_{L^1(L^1)}} \sum_{k=1}^2 \lVert g_{v,k}^- (\nu_v \cdot ) \rVert_{L^{q_3}} \lVert \nabla_h R_{v,k} \rVert_{L^\infty (W^{1,\infty})} \lVert \Omega_k (\sigma_v \cdot) \rVert_{H^{s_3}} \\
    &\lesssim \sigma_v^{-1} \kappa_v^{1/q_2 - 1/q_3} (\sigma_v \mu_v)^{s_3} \mu_v^{-1} \lesssim \lambda_v^{-\gamma_v}.
\end{align*}
\endgroup
\end{proof}

\subsection{Temporal correctors} 
\begin{lemma} \label{temporallemma}
    The temporal correctors satisfy the estimates
    \begingroup
    \allowdisplaybreaks
    \begin{align*}
        \lVert u_{t,h} \rVert_{L^\infty (W^{n,\infty})} &\lesssim \lambda_h^{-\gamma_h}, \\
        \lVert u_{t,v} \rVert_{L^\infty (W^{n,\infty})} &\lesssim \lambda_v^{-\gamma_v}, \\
        \lVert w_{t,v} \rVert_{L^\infty (W^{n,\infty})} &\lesssim \lambda_v^{-\gamma_v},
    \end{align*} 
    \endgroup
    where $n\in \mathbb{N}$ is arbitrary, and the implicit constant may depend on $n$.
\end{lemma}

\begin{proof}
Using \eqref{eq:est-hh} and \eqref{eq:est-hv} we obtain
\begingroup
\allowdisplaybreaks
\begin{align*}
    \lVert u_{t,h} \rVert_{L^\infty (W^{n,\infty})} &\leq \nu_h^{-1} \lVert h_h (\nu_h \cdot) \rVert_{L^\infty} C_{R_h} \lesssim \nu_h^{-1} \leq \lambda_h^{-\gamma_h}, \\
    \lVert u_{t,v} \rVert_{L^\infty (W^{n,\infty})} &\leq \nu_v^{-1} \lVert h_{v,k} (\nu_v \cdot) \rVert_{L^\infty} C_{R_v} \lesssim \nu_v^{-1} \lesssim \lambda_v^{-\gamma_v}, \\
    \lVert w_{t,v} \rVert_{L^\infty (W^{n,\infty})} &\leq \nu_v^{-1} \lVert h_{v,k} (\nu_v \cdot) \rVert_{L^\infty} C_{R_v} \lesssim \nu_v^{-1} \lesssim \lambda_v^{-\gamma_v}.
\end{align*}
\endgroup
\end{proof}

\subsection{Conclusion} 
We have already shown in section~\ref{perturbationsection} that $\overline{u}_p$, $\widetilde{u}_p$ and $w_p$ satisfy 
$$
    \nabla_h \cdot (\overline{u}_p + \widetilde{u}_p ) + \partial_z w_p =0,
$$ 
as well as item 1 of Proposition~\ref{perturbativeproposition}. Hence $(u+\overline{u}_p + \widetilde{u}_p,w+w_p)$ fulfills \eqref{eq:hyd-er-div}. Additionally, we have shown in section~\ref{perturbationsection} that $\partial_z P=0$ and hence $p+P$ satisfies \eqref{eq:hyd-er-p}. Moreover, we proved that \eqref{eq:hyd-er-u} holds. Consequently $(u + \overline{u}_p + \widetilde{u}_p, w + w_p, p + P, R_{h,1}, R_{v,1})$ is indeed a solution of the Euler-Reynolds system \eqref{eq:hyd-er-u}-\eqref{eq:hyd-er-div}. We also showed in section~\ref{perturbationsection} that $(u + \overline{u}_p + \widetilde{u}_p, w + w_p, p + P, R_{h,1}, R_{v,1})$ is well-prepared for the time interval $I$ and parameter $\tau/2$.

Furthermore, estimates \eqref{eq:main-ubar1}-\eqref{eq:main-ubar2} of Proposition~\ref{perturbativeproposition} are a simple consequence of Lemmas~\ref{horizontalprinc}, \ref{verticalprinc}, \ref{correctorlemma} and \ref{temporallemma}, where one has to choose $\lambda_h$, $\lambda_v$ sufficiently large, depending on $R_h$ and $R_v$, respectively. 

In addition, estimate \eqref{eq:main-product} can be derived from Lemmas~\ref{verticalprinc}, \ref{correctorlemma}, \ref{temporallemma} as well. Indeed, Lemma~\ref{verticalprinc} already proves $\lVert w_{p,v} u_{p,v} \rVert_{L^1 (B^{-1}_{1,\infty} )} \lesssim \lVert R_v \rVert_{L^1 (B^{-1}_{1,\infty})}$. Moreover, from Lemmas~\ref{lemma:essential-besov}, \ref{verticalprinc}, \ref{correctorlemma} and \ref{temporallemma} we obtain 
\begin{align*} 
    \lVert w_{p,v} (u_{c,v} + u_{t,v}) \rVert_{L^1 (B^{-1}_{1,\infty})} &\lesssim \lVert w_{p,v} (u_{c,v} + u_{t,v}) \rVert_{L^1 (L^1)} \\
    &\lesssim \lVert w_{p,v} \rVert_{L^{q_2'} (L^2)} ( \lVert u_{c,v} \rVert_{L^{q_2} (L^2)} + \lVert u_{t,v} \rVert_{L^{q_2} (L^2)} ) \lesssim \lambda_v^{-\gamma_v}.
\end{align*}
Similarly (from the proof of Lemma~\ref{verticalprinc} we obtain $\|u_{p,v}\|_{L^{q_2}(L^2)}\lesssim C_{R_h,R_v}$) 
\begin{equation*} 
    \lVert w_{t,v} \widetilde{u}_p \rVert_{L^1 (B^{-1}_{1,\infty})} \lesssim \lVert w_{t,v} \widetilde{u}_p \rVert_{L^1 (L^1)} \lesssim \lVert w_{t,v} \rVert_{L^{q_2'} (L^2)} \lVert \widetilde{u}_p \rVert_{L^{q_2} (L^2)} \lesssim \lambda_v^{-\gamma_v}.
\end{equation*}
Finally, Lemmas~\ref{verticalprinc}, \ref{correctorlemma} and \ref{temporallemma} yield 
\begin{align*}
    \lVert w \widetilde{u}_p + w_p u \rVert_{L^1 (B^{-1}_{1,\infty})} &\lesssim \lVert w \widetilde{u}_p + w_p u \rVert_{L^1 (L^1)} \\
    &\lesssim \lVert w \rVert_{L^\infty (L^\infty)} \lVert \widetilde{u}_p \rVert_{L^1 (L^1)} + \lVert w_p \rVert_{L^1 (L^1)} \lVert u \rVert_{L^\infty (L^\infty)} \lesssim \lambda_v^{-\gamma_v}.
\end{align*}
Hence, if $\lambda_v$ is chosen sufficiently large, depending on $R_v$, we obtain \eqref{eq:main-product}.

\section{Estimates on the stress tensor} \label{reynoldsestimates} 

In order to finish the proof of Proposition~\ref{perturbativeproposition} it remains to show estimates \eqref{eq:main-Rh}, \eqref{eq:main-Rv}. These two estimates simply follow from Lemmas~\ref{horizontaloscestimate}, \ref{verticaloscestimate}, \ref{correctorestimate} and \ref{linearestimate}, which we prove in this section, below.

\subsection{Oscillation error}
\subsubsection{Horizontal part}

\begin{lemma} \label{horizontaloscestimate}
    If $\lambda_h$ is chosen sufficiently large (depending on $R_h$), then the horizontal oscillation error satisfies
    \begin{equation} \label{eq:est-Rosch}
        \lVert R_{\osc,h} \rVert_{L^1 (L^1)} \leq \frac{\epsilon}{3}. 
    \end{equation}
\end{lemma}

\begin{proof}
Using Lemmas~\ref{parameterlemma}, \ref{lemma:hid}, \ref{lemma:bid} and \ref{coefficientlemma} we estimate $R_{\osc,x,h}$ as follows
\begingroup
\allowdisplaybreaks
\begin{align*}
    \lVert R_{\osc,x,h} \rVert_{L^1 (L^1)} &= \bigg\lVert \sum_{k \in \Lambda} \mathcal{B} \bigg( \nabla_h (a_k^2), W_k (\sigma_h \cdot) \otimes W_k (\sigma_h \cdot) - \int_{\mathbb{T}^2} W_k \otimes W_k \dx  \bigg) \bigg\rVert_{L^1(L^1)} \\
    &\leq \sum_{k \in \Lambda} \lVert \nabla_h (a_k^2) \rVert_{L^1 (C^1)} \bigg\lVert \mathcal{R}_h \bigg( W_k (\sigma_h \cdot) \otimes W_k (\sigma_h \cdot) - \int_{\mathbb{T}^2} W_k \otimes W_k \dx  \bigg) \bigg\rVert_{L^1} \\
    &\leq \sigma_h^{-1} \sum_{k \in \Lambda} \lVert \nabla_h (a_k^2) \rVert_{L^1 (C^1)} \bigg\lVert W_k \otimes W_k - \int_{\mathbb{T}^2} W_k \otimes W_k \dx  \bigg\rVert_{L^1} \\
    &\lesssim \sigma_h^{-1} \kappa_h^{-1/2} .
\end{align*}
\endgroup
Here we have used that (similar to \eqref{eq:prod1}) 
\begin{equation}
    \bigg\lVert W_k \otimes W_k - \int_{\mathbb{T}^2} W_k \otimes W_k \dx  \bigg\rVert_{L^1} \lesssim \|W_k\otimes W_k\|_{L^1} + 1 \leq \|W_k\|_{L^2}^2 + 1 \lesssim 1,
\end{equation}
according to Lemma~\ref{mikadobounds}.

Next, we obtain from \eqref{eq:est-hh}
\begin{align*}
    \lVert R_{\osc,t,h} \rVert_{L^1 (L^1)} &\leq \nu_h^{-1} \lVert h_h (\nu_h \cdot) \rVert_{L^\infty} \lVert \partial_t R_h \rVert_{L^1 (L^1)} \lesssim \nu_h^{-1}. 
\end{align*}
Finally, using Lemma~\ref{coefficientlemma} and Proposition~\ref{prop:mikado-h} we get
\begin{align*}
    \lVert R_{\far} \rVert_{L^1 (L^1)} &\leq \bigg\lVert \sum_{k,k' \in \Lambda, k \neq k'} a_{k} a_{k'} W_k (\sigma_h \cdot) \otimes W_{k'} (\sigma_h \cdot ) \bigg\rVert_{L^1(L^1)} \\
    &\lesssim  \sum_{k,k' \in \Lambda, k \neq k'} \|a_{k}\|_{L^2(L^\infty)} \|a_{k'}\|_{L^2(L^\infty)} \Big\lVert W_k (\sigma_h \cdot) \otimes W_{k'} (\sigma_h \cdot ) \Big\rVert_{L^1} \\
    &\lesssim  \sum_{k,k' \in \Lambda, k \neq k'} \lVert W_k \otimes W_{k'} \rVert_{L^1}\lesssim \mu_h^{-1}.
\end{align*}
By choosing $\lambda_h$ large enough (depending on $R_h$), we conclude with \eqref{eq:est-Rosch}.
\end{proof}

\subsubsection{Vertical part} 
\begin{lemma} \label{verticaloscestimate}
    If $\lambda_v$ is chosen sufficiently large (depending on $R_v$), then the vertical oscillation error satisfies 
    \begin{equation} \label{eq:est-Roscv}
        \lVert R_{\osc,v} \rVert_{L^1 (B^{-1}_{1,\infty})} \leq \frac{\epsilon}{3}. 
    \end{equation}
\end{lemma}

\begin{proof}

Using \eqref{eq:l1-gvpgvm} and Lemma~\ref{lemma:prod} we find
\begingroup
\allowdisplaybreaks
\begin{align*}
    &\lVert R_{\osc,x,v} \rVert_{L^1 (B^{-1}_{1,\infty})} \\
    &\lesssim \sum_{k=1}^2 \Big\lVert g_{v,k}^-(\nu_v\cdot)g_{v,k}^+(\nu_v\cdot)\Big\rVert_{L^1} \lVert\theta^2 \rVert_{L^\infty} \bigg\lVert  R_{v,k} \bigg( \phi_k (\sigma_v \cdot) W_k (\sigma_v \cdot) - \int_{\mathbb{T}^2} \phi_k W_k \dx \bigg) \bigg\rVert_{L^\infty(B^{-1}_{1,\infty})} \\ 
    &\lesssim \sigma_v^{-1}.
\end{align*} 
\endgroup
For the temporal part of the error we obtain by \eqref{eq:est-hv}
\begin{align*}
    \lVert R_{\osc,t,v} \rVert_{L^1 (B^{-1}_{1,\infty})} \lesssim \lVert R_{\osc,t,v} \rVert_{L^\infty (L^\infty)} = \nu_v^{-1}  \sum_{k=1}^2 \| h_{v,k} (\nu_v \cdot )\|_{L^\infty} \|\partial_t R_{v,k} \|_{L^\infty(L^\infty)} &\lesssim \nu_v^{-1}.
\end{align*}
Consequently \eqref{eq:est-Roscv} follows by choosing $\lambda_v$ sufficiently large, depending on $R_v$.
\end{proof}

\subsection{Corrector error}
\begin{lemma} \label{correctorestimate} 
    If $\lambda_h$ and $\lambda_v$ are sufficiently large (depending on $R_h$ and $R_v$, respectively), then the corrector errors satisfy the estimates
    \begingroup
    \allowdisplaybreaks
    \begin{align}
        \lVert R_{\cor,h} \rVert_{L^1 (L^1)} &\leq \frac{\epsilon}{3}, \label{eq:est-Rcorh} \\
        \lVert R_{\cor,v} \rVert_{L^1 (B^{-1}_{1,\infty})} &\leq \frac{\epsilon}{3}. \label{eq:est-Rcorv} 
    \end{align}
    \endgroup
\end{lemma}

\begin{proof}
First, we estimate $R_{\cor,h}$. Since estimate \eqref{eq:lem-Rh-est3} does not hold for $p=1$, we have to introduce a suitable $r>1$. Let us fix $1<r\leq 2$ such that $1-\frac{1}{2}\delta c_v\leq \frac{1}{r}$, where $c_v>0$ is given by Lemma~\ref{parameterlemma}. More precisely, if $1-\frac{1}{2}\delta c_v>0$, we choose $1<r\leq \min\left\{2,\frac{1}{1-\frac{1}{2}\delta c_v}\right\}$ which is possible due to $1-\frac{1}{2}\delta c_v<1$. On the other hand, if $1-\frac{1}{2}\delta c_v\leq 0$, we simply take $1<r\leq 2$. Moreover, we set $\frac{1}{\widetilde{r}}=\frac{1}{r}-\frac{1}{2}$. Then (similar to \eqref{eq:513}), we obtain by \eqref{parameters}, \eqref{eq:est-gvm} and Lemmas~\ref{parameterlemma} and \ref{mikadobounds}
\begin{align}
    \lVert u_{p,v} \rVert_{L^{q_2-} (L^{\widetilde{r}})} &\leq \frac{1}{\lVert R_h \rVert_{L^1 (L^1)}} \sum_{k=1}^2 \lVert g_{v,k}^- (\nu_v \cdot) \rVert_{L^{q_2-}} \lVert R_v \rVert_{L^\infty (L^{\infty})} \lVert W_k (\sigma_v \cdot ) \rVert_{L^{\widetilde{r}}} \notag\\
    &\lesssim \kappa_v^{1/q_2 - 1/q_2 - \delta} \mu_v^{1/2-1/\widetilde{r}} = \kappa_v^{-\delta} \mu_v^{1-1/r}\lesssim \kappa_v^{-\delta} \mu_v^{\frac{1}{2}\delta c_v} = \kappa_v^{-\delta + \frac{1}{2}\delta} \lesssim \lambda_v^{-\frac{1}{2}\gamma_v}. \label{eq:66}
\end{align}

Now we are ready to estimate $R_{\cor,h}$. Using Lemmas~\ref{lemma:hid}, \ref{horizontalprinc}, \ref{verticalprinc}, \ref{correctorlemma} and \ref{temporallemma}, and bound \eqref{eq:66} we get\footnote{To be precise in the following we will use $q_2->2$. Note that $q_2>2$ by assumption \eqref{constraints-prop} and we may assume without loss of generality that $\delta$ is small enough such that $q_2->2$. Indeed shrinking $\delta$ makes the result in Proposition~\ref{perturbativeproposition} stronger.}
\begingroup
\allowdisplaybreaks
\begin{align*}
    \lVert R_{\cor,h} \rVert_{L^1 (L^1)} &\lesssim \lVert R_{\cor,h} \rVert_{L^1 (L^r)} \\
    &\lesssim \bigg\lVert \big( u_{c,h} + u_{t,h} \big) \otimes \big( u_{c,h} + u_{t,h} \big) + u_{p,h} \otimes \big(u_{c,h} + u_{t,h} \big) + \big(u_{c,h} + u_{t,h} \big)\otimes u_{p,h} \\
    &\qquad + \overline{\big( u_{p,v} + u_{c,v} + u_{t,v} \big) \otimes \big( u_{p,v} + u_{c,v} + u_{t,v} \big)} \bigg\rVert_{L^1 (L^r)} \\
    &\lesssim \lVert u_{p,h} \rVert_{L^2 (L^2)} \Big( \lVert u_{c,h} \rVert_{L^2 (L^\infty)} + \lVert u_{t,h} \rVert_{L^2(L^\infty)}\Big) + \lVert u_{c,h} \rVert_{L^2 (L^\infty)}^2 + \lVert u_{t,h} \rVert_{L^2 (L^\infty)}^2 \\
    &\qquad + \lVert u_{p,v} \rVert_{L^2 (L^2)} \lVert u_{p,v} \rVert_{L^2 (L^{\widetilde{r}})} + \lVert u_{c,v} \rVert_{L^2 (L^\infty)}^2 + \lVert u_{t,v} \rVert_{L^2 (L^\infty)}^2 \\
    &\lesssim \lVert R_h \rVert_{L^1(L^1)}^{1/2} \lambda_h^{-\gamma_h} + \lambda_h^{-2\gamma_h} + \lambda_v^{-\frac{3}{2}\gamma_v} + \lambda_v^{-2\gamma_v}, 
\end{align*}
\endgroup
which implies \eqref{eq:est-Rcorh} as soon as $\lambda_h$ and $\lambda_v$ are suitably large (depending on $R_h$ and $R_v$, respectively). 

Finally, according to Lemmas~\ref{Lpverticalbound}, \ref{horizontalprinc}, \ref{verticalprinc}, \ref{correctorlemma}, \ref{temporallemma} and \ref{lemma:essential-besov}
\begingroup
\allowdisplaybreaks
\begin{align*}
    &\lVert R_{\cor,v} \rVert_{L^1 (B^{-1}_{1,\infty})} \\
    &\lesssim \bigg\lVert \nabla_h \cdot \bigg( \reallywidetilde{\big( u_{p,v} + u_{c,v} + u_{t,v} \big) \otimes \big( u_{p,v} + u_{c,v} + u_{t,v} \big)} \\
    &\qquad \qquad+ \big( u_{p,h} + u_{c,h} + u_{t,h} \big) \otimes \big( u_{p,v} + u_{c,v} + u_{t,v} \big) \\
    &\qquad \qquad+ \big( u_{p,v} + u_{c,v} + u_{t,v} \big) \otimes \big( u_{p,h} + u_{c,h} + u_{t,h} \big)  \bigg) \bigg\rVert_{L^1 (B^{-1}_{1,\infty})} \\
    &\qquad + \Big\lVert w_{t,v} \big(u_{p,h}+u_{c,h}+u_{t,h} + u_{p,v}+u_{c,v}+u_{t,v}\big) + w_{p,v} \big( u_{c,v} + u_{t,v} \big) \Big\rVert_{L^1 (L^1)} \\
    &\lesssim \bigg\lVert \reallywidetilde{\big( u_{p,v} + u_{c,v} + u_{t,v} \big) \otimes \big( u_{p,v} + u_{c,v} + u_{t,v} \big)} \\
    &\qquad \qquad+ \big( u_{p,h} + u_{c,h} + u_{t,h} \big) \otimes \big( u_{p,v} + u_{c,v} + u_{t,v} \big) \\
    &\qquad \qquad+ \big( u_{p,v} + u_{c,v} + u_{t,v} \big) \otimes \big( u_{p,h} + u_{c,h} + u_{t,h} \big) \bigg\rVert_{L^1 (L^1)} \\
    &\qquad + \Big\lVert w_{t,v} \big(u_{p,h}+u_{c,h}+u_{t,h} + u_{p,v}+u_{c,v}+u_{t,v}\big) + w_{p,v} \big( u_{c,v} + u_{t,v} \big) \Big\rVert_{L^1 (L^1)} \\
    &\lesssim \lVert u_{p,v} \rVert_{L^2 (L^2)}^2 + \lVert u_{c,v} \rVert_{L^2 (L^2)}^2 + \lVert u_{t,v} \rVert_{L^2 (L^2)}^2 \\
    &\qquad + \Big( \lVert u_{p,h} \rVert_{L^2 (L^2)} + \lVert u_{c,h} \rVert_{L^2 (L^2)} + \lVert u_{t,h} \rVert_{L^2 (L^2)} \Big) \Big( \lVert u_{p,v} \rVert_{L^2 (L^2)} + \lVert u_{c,v} \rVert_{L^2 (L^2)} + \lVert u_{t,v} \rVert_{L^2 (L^2)} \Big) \\
    &\qquad + \lVert w_{t,v} \rVert_{L^2 (L^2)} \Big( \lVert u_{p,h} \rVert_{L^2 (L^2)} + \lVert u_{c,h} \rVert_{L^2 (L^2)} + \lVert u_{t,h} \rVert_{L^2 (L^2)} \Big) \\
    &\qquad + \lVert w_{t,v} \rVert_{L^2 (L^2)} \Big( \lVert u_{p,v} \rVert_{L^2 (L^2)} + \lVert u_{c,v} \rVert_{L^2 (L^2)} + \lVert u_{t,v} \rVert_{L^2 (L^2)} \Big) \\
    &\qquad + \lVert w_{p,v} \rVert_{L^{q_2'} (L^2)} \Big( \lVert u_{c,v} \rVert_{L^{q_2} (L^2)} + \lVert u_{t,v} \rVert_{L^{q_2} (L^2)} \Big) \\
    &\lesssim \lambda_v^{-2\gamma_v} + \big(\lVert R_h \rVert_{L^1 (L^1)}^{1/2} + \lambda_h^{-\gamma_h}\big) \lambda_v^{-\gamma_v} + \lVert R_h \rVert_{L^1 (L^1)} \lambda_v^{-\gamma_v} .
\end{align*}
\endgroup
In these estimates we have used the fact that the time interval $[0,T]$ is finite. Then \eqref{eq:est-Rcorv} follows by choosing $\lambda_h$ and $\lambda_v$ large enough (again depending on $R_h$ and $R_v$, respectively). 
\end{proof}

\subsection{Linear error}
\begin{lemma} \label{linearestimate}
    If $\lambda_h$ and $\lambda_v$ are chosen sufficiently large (depending on $R_h$ and $R_v$, respectively), then the linear errors satisfy the estimates
    \begin{align}
        \lVert R_{\lin,h} \rVert_{L^1 (L^1)} &\leq \frac{\epsilon}{3}, \label{eq:est-Rlinh} \\
        \lVert R_{\lin,v} \rVert_{L^1 (B^{-1}_{1,\infty})} &\leq \frac{\epsilon}{3}. \label{eq:est-Rlinv}
    \end{align} 
\end{lemma}
In order to prove this Lemma, we consider the time derivative (see section~\ref{subsubsec:lin-err-time}) and advective terms (see section~\ref{subsubsec:lin-err-adv}) separately.

\begin{proof}[Proof of Lemma~\ref{linearestimate}]
We simply conclude using Lemmas~\ref{timederivativeestimate} and \ref{advectionestimate} below by choosing $\lambda_h$ and $\lambda_v$ large enough.
\end{proof}

\subsubsection{Time derivative} \label{subsubsec:lin-err-time} 
\begin{lemma} \label{timederivativeestimate}
    For the time derivative part of the linear error, the following bounds hold
    \begingroup
    \allowdisplaybreaks
    \begin{align}
        \lVert R_{\lin,t,h} \rVert_{L^1 (L^1)} &\lesssim \lambda_h^{-\gamma_h}, \\
        \lVert R_{\lin,t,v} \rVert_{L^1 (B^{-1}_{1,\infty})} &\lesssim \lambda_v^{-\gamma_v}.
    \end{align}
    \endgroup
\end{lemma}

\begin{proof}
According to \eqref{eq:uph+uch} we have
\begin{equation*}
    \partial_t (u_{p,h} + u_{c,h}) = \sigma_h^{-1} \sum_{k \in \Lambda} \nabla_h \cdot (\partial_t a_k (x,t) \Omega_k (\sigma_h x)) .
\end{equation*}
Using Lemmas~\ref{parameterlemma}, \ref{lemma:hid}, \ref{mikadobounds} and \ref{coefficientlemma} we thus find
\begin{align*}
    \lVert R_{\lin,t,h} \rVert_{L^1 (L^1)} &\lesssim \lVert \mathcal{R}_h \partial_t (u_{p,h} + u_{c,h}) \rVert_{L^1 (L^2)} \\
    &\lesssim \sigma_h^{-1} \sum_{k \in \Lambda} \lVert \partial_t a_k \rVert_{L^1 (L^\infty)} \lVert \Omega_k (\sigma_h \cdot ) \rVert_{L^2} \leq \sigma_h^{-1} \nu_h \kappa_h^{1/2} \mu_h^{-1} \lesssim \lambda_h^{-\gamma_h}. 
\end{align*}

Similarly \eqref{eq:upv+ucv} implies 
\begin{equation*}
    \partial_t (u_{p,v} + u_{c,v}) = - \frac{\sigma_v^{-1}}{\lVert R_h \rVert_{L^1 (L^1 )}} \sum_{k = 1}^2 \nabla_h \cdot \partial_t\big(g_{v,k}^- (\nu_v t) \theta R_{v,k} \Omega_k (\sigma_v x)\big).
\end{equation*}
Hence from Lemmas~\ref{parameterlemma}, \ref{Lpverticalbound}, \ref{mikadobounds} and \ref{lemma:essential-besov}, the assumption $q_2>2$, and estimate \eqref{eq:est-gvm-sobolev} we obtain 
\begingroup
\allowdisplaybreaks
\begin{align*}
    &\lVert R_{\lin,t,v} \rVert_{L^1 (B^{-1}_{1,\infty})} \\
    &\lesssim \frac{\sigma_v^{-1}}{\lVert R_h \rVert_{L^1 (L^1 )}} \sum_{k = 1}^2 \Big\lVert \mathcal{R}_v \nabla_h \cdot \partial_t\big(g_{v,k}^- (\nu_v \cdot) \theta R_{v,k} \Omega_k (\sigma_v \cdot)\big) \Big\rVert_{L^1 (B^{-1}_{1,\infty})} \\
    &\lesssim \frac{\sigma_v^{-1}}{\lVert R_h \rVert_{L^1 (L^1 )}} \sum_{k = 1}^2 \Big\lVert \nabla_h \cdot \partial_t\big(g_{v,k}^- (\nu_v \cdot) \theta R_{v,k} \Omega_k (\sigma_v \cdot)\big) \Big\rVert_{L^1 (B^{-1}_{1,\infty})} \\
    &\lesssim \frac{\sigma_v^{-1}}{\lVert R_h \rVert_{L^1 (L^1 )}} \sum_{k = 1}^2 \Big\lVert \partial_t\big(g_{v,k}^- (\nu_v \cdot) \theta R_{v,k} \Omega_k (\sigma_v \cdot)\big) \Big\rVert_{L^1 (L^1)} \\
    &\lesssim \frac{\sigma_v^{-1}}{\lVert R_h \rVert_{L^1 (L^1 )}} \sum_{k = 1}^2 \lVert g_{v,k}^- (\nu_v \cdot)\rVert_{W^{1,1}} \lVert\theta\rVert_{W^{1,\infty}} \lVert R_{v,k}\rVert_{W^{1,\infty}(L^\infty)} \lVert\Omega_k (\sigma_v \cdot)\rVert_{L^1} \\
    &\lesssim \sigma_v^{-1} \nu_v \kappa_v^{1/q_2} \mu_v^{-3/2} \lesssim \lambda_v^{-\gamma_v}. 
\end{align*}
\endgroup
\end{proof}

\subsubsection{Advective terms} \label{subsubsec:lin-err-adv}
\begin{lemma} \label{advectionestimate}
    For the advective part of the linear error, the following bounds hold
    \begingroup
    \allowdisplaybreaks
    \begin{align*}
        \lVert R_{\lin,x,h} \rVert_{L^1 (L^1)} &\lesssim \lambda_h^{-\gamma_h} + \lambda_v^{-\gamma_v}, \\
        \lVert R_{\lin,x,v} \rVert_{L^1 (B^{-1}_{1,\infty})} &\lesssim \lambda_h^{-\gamma_h} + \lambda_v^{-\gamma_v}.
    \end{align*} 
    \endgroup
\end{lemma}
        
\begin{proof}
Lemmas~\ref{lemma:hid}, \ref{horizontalprinc}, \ref{verticalprinc}, \ref{correctorlemma} and \ref{temporallemma} 
yield
\begingroup
\allowdisplaybreaks
\begin{align*}
    &\lVert R_{\lin,x,h} \rVert_{L^1 (L^1)} \\
    &=\bigg\lVert \mathcal{R}_h \bigg[ \nabla_h \cdot \bigg( \overline{u} \otimes \big( u_{p,h} + u_{c,h} + u_{t,h} \big) + \overline{u \otimes \big( u_{p,v} + u_{c,v} + u_{t,v} \big) } \\
    &\qquad + \big( u_{p,h} + u_{c,h} + u_{t,h} \big) \otimes \overline{u} + \overline{\big( u_{p,v} + u_{c,v} + u_{t,v} \big) \otimes u}\bigg) \bigg] \bigg\rVert_{L^1 (L^1)} \\
    &\lesssim \bigg\lVert \overline{u} \otimes \big( u_{p,h} + u_{c,h} + u_{t,h} \big) + \overline{u \otimes \big( u_{p,v} + u_{c,v} + u_{t,v} \big) } \\
    &\qquad + \big( u_{p,h} + u_{c,h} + u_{t,h} \big) \otimes \overline{u} + \overline{\big( u_{p,v} + u_{c,v} + u_{t,v} \big) \otimes u} \bigg\rVert_{L^1 (L^2)} \\ 
    &\lesssim \lVert u\rVert_{L^\infty(L^\infty)} \Big( \lVert u_{p,h} \rVert_{L^1(L^2)} + \lVert u_{c,h} \rVert_{L^1(L^2)} + \lVert u_{t,h} \rVert_{L^1(L^2)} \Big) \\
    &\qquad + \lVert u\rVert_{L^\infty(L^\infty)} \Big( \lVert u_{p,v} \rVert_{L^1(L^2)} + \lVert u_{c,v} \rVert_{L^1(L^2)} + \lVert u_{t,v} \rVert_{L^1(L^2)} \Big) \\
    & \lesssim \lambda_h^{-\gamma_h} + \lambda_v^{-\gamma_v}.
\end{align*} 
\endgroup

For the vertical advective terms we have according to Lemmas~\ref{Lpverticalbound}, \ref{horizontalprinc}, \ref{verticalprinc}, \ref{correctorlemma}, \ref{temporallemma} and \ref{lemma:essential-besov} 
\begingroup
\allowdisplaybreaks
\begin{align*}
    &\lVert R_{\lin,x,v} \rVert_{L^1 (B^{-1}_{1,\infty})} \\
    &=\bigg\lVert \mathcal{R}_v \bigg[ \nabla_h \cdot \bigg( \widetilde{u} \otimes \big( u_{p,h} + u_{c,h} + u_{t,h} \big) + \reallywidetilde{u\otimes\big(u_{p,v}+u_{c,v}+u_{t,v}\big)} \\
    &\qquad\qquad +  \big( u_{p,h} + u_{c,h} + u_{t,h} \big) \otimes \widetilde{u} + \reallywidetilde{\big( u_{p,v} + u_{c,v} + u_{t,v} \big) \otimes u} \bigg)  \\ 
    &\qquad + \partial_z \bigg( w \big( u_{p,h} + u_{c,h} + u_{t,h} + u_{p,v} + u_{c,v} + u_{t,v} \big) + \big( w_{p,v} + w_{t,v} \big) u \bigg) \bigg] \bigg\rVert_{L^1 (B^{-1}_{1,\infty})} \\
    &\lesssim \bigg\lVert \widetilde{u} \otimes \big( u_{p,h} + u_{c,h} + u_{t,h} \big) + \reallywidetilde{u\otimes\big(u_{p,v}+u_{c,v}+u_{t,v}\big)} \\
    &\qquad\qquad +  \big( u_{p,h} + u_{c,h} + u_{t,h} \big) \otimes \widetilde{u} + \reallywidetilde{\big( u_{p,v} + u_{c,v} + u_{t,v} \big) \otimes u} \bigg\rVert_{L^1 (L^1)}  \\ 
    &\qquad + \bigg\lVert  w \big( u_{p,h} + u_{c,h} + u_{t,h} + u_{p,v} + u_{c,v} + u_{t,v} \big) + \big( w_{p,v} + w_{t,v} \big) u \bigg\rVert_{L^1 (L^1)} \\
    &\lesssim \Big(\lVert u\rVert_{L^\infty(L^\infty)} + \lVert w\rVert_{L^\infty(L^\infty)}\Big) \Big( \lVert u_{p,h} \rVert_{L^1(L^2)} + \lVert u_{c,h} \rVert_{L^1(L^2)} + \lVert u_{t,h} \rVert_{L^1(L^2)} \Big) \\
    &\qquad + \Big(\lVert u\rVert_{L^\infty(L^\infty)} + \lVert w\rVert_{L^\infty(L^\infty)}\Big) \Big( \lVert u_{p,v} \rVert_{L^1(L^2)} + \lVert u_{c,v} \rVert_{L^1(L^2)} + \lVert u_{t,v} \rVert_{L^1(L^2)} \Big) \\
    &\qquad + \lVert u\rVert_{L^\infty(L^\infty)} \Big( \lVert w_{p,v}\rVert_{L^1(L^1)} + \lVert w_{t,v}\rVert_{L^1(L^1)} \Big) \\
    &\lesssim \lambda_h^{-\gamma_h} + \lambda_v^{-\gamma_v}.
\end{align*}
\endgroup
\end{proof}

\section{The viscous primitive equations} \label{viscoussection}

In this section we consider the viscous primitive equations \eqref{eq:visc-u}-\eqref{eq:visc-div}. We begin by stating the viscous primitive-Reynolds system
\begin{align}
    \partial_t u - \nu_h^* \Delta_h u - \nu_v^* \partial_{zz} u + u \cdot \nabla_h u + w \partial_z u + \nabla_h p &= \nabla_h \cdot R_h + \partial_z R_v, \label{eq:visc-er-u} \\
    \partial_z p &= 0, \label{eq:visc-er-p}\\
    \nabla_h \cdot u + \partial_z w &=0 . \label{eq:visc-er-div} 
\end{align}
We prove the following version of Proposition~\ref{perturbativeproposition}. Theorem~\ref{viscousmainthm} can be proven in exactly the same fashion as Theorem~\ref{mainresult}. 

\begin{proposition} \label{viscousperturbativeprop}
	Suppose $(u,w,p,R_h,R_v)$ is a smooth solution of the viscous primitive-Reynolds system \eqref{eq:visc-er-u}-\eqref{eq:visc-er-div}, which is well-prepared with associated time interval $I$ and parameter $\tau>0$. Moreover consider parameters $1\leq q_1,q_2,q_3\leq \infty$ and $0<s_1,s_3$ which satisfy the following constraints\footnote{Again the constraints \eqref{visc-constraints-prop} are weaker than \eqref{viscousconstraints}, cf. the footnote in Proposition~\ref{perturbativeproposition}.}
	\begin{equation} \label{visc-constraints-prop} 
        q_2 > 2, \quad \frac{2}{q_1} > s_1 + 1, \quad \frac{2}{q_3} > s_3 + \frac{2}{q_2}, \quad s_3> \frac{1}{2\left(1-\frac{1}{q_2}\right)} \left(\frac{1}{q_3} - \frac{1}{q_2}\right).
    \end{equation} 
	Finally let $\delta,\epsilon>0$ be arbitrary. Then there exists another smooth solution $(u+\overline{u}_p+\widetilde{u}_p,w+w_p,p+P,R_{h,1},R_{v,1})$ of the viscous primitive-Reynolds system \eqref{eq:visc-er-u}-\eqref{eq:visc-er-div} which is well-prepared with respect to the same time interval $I$ and parameter $\tau/2$, and has the following properties:
	\begin{enumerate}
    	\item $(\overline{u}_p,\widetilde{u}_p,w_p)(x,t)=(0,0,0)$ whenever $\dist (t, I^c) \leq \tau/2$.
    	
    	\item The perturbation and Reynolds stress tensors satisfy the following estimates 
    	\begingroup
    	\allowdisplaybreaks
    	\begin{align}
        	\lVert R_{h,1} \rVert_{L^1 (L^1)} &\leq \epsilon, \label{eq:visc-Rh} \\
        	\lVert R_{v,1} \rVert_{L^1 (B^{-1}_{1,\infty})} &\leq \epsilon, \label{eq:visc-Rv} \\
        	\lVert \overline{u}_p \rVert_{L^1 (W^{1,1})} &\leq \epsilon, \label{eq:visc-ubar1} \\
        	\lVert \overline{u}_p \rVert_{L^{q_1} (H^{s_1})} &\leq \epsilon, \label{eq:visc-ubar2} \\
            \lVert \widetilde{u}_p \rVert_{L^1 (W^{1,1})} &\leq \epsilon, \label{eq:visc-utilde1} \\ 
        	\lVert \widetilde{u}_p \rVert_{L^{q_2-} (L^2)} &\leq \epsilon, \label{eq:visc-utilde2} \\
        	\lVert \widetilde{u}_p \rVert_{L^{q_3-} (H^{s_3})} &\leq \epsilon, \label{eq:visc-utilde3} \\
        	\lVert w_p \rVert_{L^{q_2'-} (L^2)} &\leq \epsilon, \label{eq:visc-w1} \\
        	\lVert w_p \rVert_{L^{q_3'-} (H^{-s_3})} &\leq \epsilon. \label{eq:visc-w2}
    	\end{align}
    	\endgroup
    
    	\item Moreover, we have the following bounds
        \begin{align}
            \lVert \overline{u}_p \rVert_{L^{2} (L^2)} &\lesssim \lVert R_h \rVert_{L^1 (L^1)}^{1/2}, \label{eq:visc-ubar3} \\
            \lVert w_p \widetilde{u}_p + w \widetilde{u}_p + w_p u \rVert_{L^1 (B^{-1}_{1,\infty})} &\lesssim \lVert R_v \rVert_{L^1 (B^{-1}_{1,\infty})}. \label{eq:visc-product} 
        \end{align}
	\end{enumerate}
\end{proposition}

In order to prove Proposition \ref{viscousperturbativeprop}, we need the following version of Lemma \ref{parameterlemma}.

\begin{lemma} \label{viscousparameter}
    Let $1 \leq q_1, q_2, q_3 \leq \infty$ and $0<s_1, s_3$ satisfy the conditions \eqref{visc-constraints-prop}. Then we can choose $a_i,b_i,c_i>0$ for $i=h,v$ in \eqref{parameters} with the property that there exist $\gamma_h, \gamma_v > 0$ such that
    \begingroup
    \allowdisplaybreaks
    \begin{align}
        \kappa_h^{-1/2} \sigma_h \mu_h^{1/2} &\leq \lambda_h^{-\gamma_h}, \label{visc-parameterineq1} \\
        \kappa_v^{1/q_2 - 1} \sigma_v \mu_v^{1/2} &\leq \lambda_v^{-\gamma_v}, \label{visc-parameterineq2} 
    \end{align}
    \endgroup
    in addition to \eqref{parameterineq1}-\eqref{parameterineq5} and $\mu_i,\sigma_i,\kappa_i,\nu_i\geq \lambda_i^{\gamma_i}$ for $i=h,v$.
\end{lemma}

\begin{proof} 
Similar to the proof of Lemma~\ref{parameterlemma}, it suffices to show that there is a choice of $a_i,b_i,c_i>0$ for $i=h,v$ such that 
\begingroup
\allowdisplaybreaks
\begin{align}
    -\bigg( \frac{1}{2} - \frac{1}{q_1} \bigg) c_h - s_1 (b_h+1) &>0, \label{eq:vp1}\\
    b_h - a_h - \frac{1}{2} c_h + 1 &>0, \label{eq:vp2}\\
    -\bigg( \frac{1}{q_2} - \frac{1}{q_3} \bigg) c_v - s_3 (b_v+1) &=0 , \label{eq:vp3}\\
    b_v - a_v - \frac{1}{2} c_v + 1 &>0, \label{eq:vp4}\\
    \frac{1}{2} c_h - b_h - \frac{1}{2} &>0 , \label{eq:vp5}\\
    -\left(\frac{1}{q_2} - 1\right) c_v - b_v - \frac{1}{2} &>0 . \label{eq:vp6}
\end{align}
\endgroup
Let us first fix $0<a_h<1/2$, and then $b_h>0$ such that 
\begin{equation} \label{eq:751}
    b_h \left(s_1 + 1 - \frac{2}{q_1} \right) < -s_1 + \left(2-2a_h\right) \left(\frac{1}{q_1}- \frac{1}{2} \right).
\end{equation}
Note that such a choice is possible since $s_1 + 1 - \frac{2}{q_1} <0$ according to \eqref{visc-constraints-prop}. Because \eqref{visc-constraints-prop} implies $q_1<2$, \eqref{eq:751} is equivalent to 
\begin{equation} \label{eq:752}
    \frac{s_1 (b_h + 1)}{\frac{1}{q_1} - \frac{1}{2}} < 2 b_h + 2 - 2a_h .
\end{equation}
Moreover as $a_h<1/2$, we have 
\begin{equation} \label{eq:753}
    2b_h + 1 < 2b_h + 2 - 2a_h.
\end{equation}
From \eqref{eq:752} and \eqref{eq:753} we deduce that there exists $c_h>0$ with 
\begingroup
\allowdisplaybreaks
\begin{align*}
    \frac{s_1 (b_h + 1)}{\frac{1}{q_1} - \frac{1}{2}} &< c_h, \\
    2b_h + 2 - 2a_h &> c_h, \\
    2b_h +1 &< c_h,
\end{align*}
\endgroup
which are equivalent to \eqref{eq:vp1}, \eqref{eq:vp2} and \eqref{eq:vp5} respectively.

Next we choose $a_v,b_v,c_v>0$. We simply deduce from \eqref{visc-constraints-prop} that 
$$
    \frac{s_3 \left(1-\frac{1}{q_2} \right)}{\frac{1}{q_3} - \frac{1}{q_2}} > \frac{1}{2}.
$$ 
Thus we can choose $0<b_v \ll 1$ such that 
\begin{equation} \label{eq:755}
    b_v \left( - 1 + \frac{s_3 \left(1-\frac{1}{q_2} \right)}{\frac{1}{q_3} - \frac{1}{q_2}} \right) - \frac{1}{2} + \frac{s_3 \left(1-\frac{1}{q_2} \right)}{\frac{1}{q_3} - \frac{1}{q_2}} > 0. 
\end{equation}
Then we fix 
$$
    c_v := \frac{s_3 \left(b_v +1\right)}{\frac{1}{q_3} - \frac{1}{q_2}},
$$
which is positive as $\frac{1}{q_3} - \frac{1}{q_2}>0$ which in turn follows from \eqref{visc-constraints-prop}. The choice of $c_v$ immediately implies \eqref{eq:vp3}, while \eqref{eq:755} is equivalent to \eqref{eq:vp6}. Finally \eqref{visc-constraints-prop} ensures 
$$
    (b_v+1) \left( 1 - \frac{s_3}{2 \left(\frac{1}{q_3} - \frac{1}{q_2}\right)}\right) >0.
$$
This is equivalent to 
$$
    b_v - \frac{1}{2} c_v + 1 > 0,
$$
which in turn allows for the choice of a small $a_v>0$ such that \eqref{eq:vp4} holds.
\end{proof}

Now we can prove Proposition \ref{viscousperturbativeprop}.

\begin{proof}[Proof of Proposition \ref{viscousperturbativeprop}]
We make the same choice of perturbations $\overline{u}_p$, $\widetilde{u}_p$, $w_p$, $P$ as we did in the inviscid case, see section \ref{perturbationsection}. The only errors that change compared to the inviscid case are the linear errors, which now contain the additional terms 
\begin{equation} \label{eq:visc-additionalterms}
    \mathcal{R}_h (\nu_h^* \Delta_h \overline{u}_p) , \qquad \mathcal{R}_v \big(\nu_h^* \Delta_h \widetilde{u}_p + \nu_v^* \partial_{zz} (\overline{u}_p + \widetilde{u}_p) \big),
\end{equation}
respectively. Thus the validity of \eqref{eq:visc-ubar2}, \eqref{eq:visc-utilde2}-\eqref{eq:visc-product} follows immediately from sections \ref{perturbationsection}-\ref{reynoldsestimates}.

As in the proof of \eqref{eq:est-uph-2} one can deduce from \eqref{visc-parameterineq1} that 
\begin{align*}
    \lVert u_{p,h} \rVert_{L^1(W^{1,1})} &\leq \sum_{k \in \Lambda} \lVert a_k \lVert_{L^{1} (W^{1,\infty})} \lVert  W_k(\sigma_h\cdot) \rVert_{W^{1,1}} \lesssim  \kappa_h^{-1/2} \sigma_h \mu_h^{1/2} \leq \lambda_h^{-\gamma_h}.
\end{align*}
Analogously we find 
\begin{align*}
    \lVert u_{c,h} \rVert_{L^{1} (W^{1,1})} &\leq \sigma_h^{-1} \sum_{k \in \Lambda} \lVert \nabla a_k \rVert_{L^{1} (W^{1,\infty})} \lVert \Omega_k (\sigma_h \cdot ) \rVert_{W^{1,1}} \lesssim \sigma_h^{-1} \kappa_h^{-1/2} \sigma_h \mu_h^{-1/2} \lesssim \lambda_h^{-\gamma_h}.
\end{align*}
Similarly we obtain from \eqref{visc-parameterineq2}
\begin{align*}
    \lVert u_{p,v} \rVert_{L^{1} (W^{1,1})} &\leq \frac{1}{\lVert R_h \rVert_{L^1 (L^1)}} \sum_{k=1}^2 \lVert g_{v,k}^- (\nu_v \cdot) \rVert_{L^{1}} \lVert R_v \rVert_{L^\infty (W^{1,\infty})} \lVert W_k (\sigma_v \cdot ) \rVert_{W^{1,1}} \\
    &\lesssim \kappa_v^{1/q_2 - 1} \sigma_v \mu_v^{1/2} \leq \lambda_v^{-\gamma_v}
\end{align*}
and
\begin{align*}
    \lVert u_{c,v} \rVert_{L^{1} (W^{1,1})} &\leq \frac{\sigma_v^{-1}}{\|R_h\|_{L^1(L^1)}} \sum_{k=1}^2 \lVert g_{v,k}^- (\nu_v \cdot ) \rVert_{L^{1}} \lVert \nabla_h R_{v,k} \rVert_{L^\infty (W^{1,\infty})} \lVert \Omega_k (\sigma_v \cdot) \rVert_{W^{1,1}} \\
    &\lesssim \sigma_v^{-1} \kappa_v^{1/q_2 - 1} \sigma_v \mu_v^{-1/2} \leq \lambda_v^{-\gamma_v}.
\end{align*}
Since $\lVert u_{t,h} \rVert_{L^1(W^{1,1})}\lesssim \lambda_h^{-\gamma_h}$ and $ \lVert u_{t,v} \rVert_{L^1(W^{1,1})}\lesssim \lambda_v^{-\gamma_v}$, which follow immediately from Lemma~\ref{temporallemma}, we deduce \eqref{eq:visc-ubar1} and \eqref{eq:visc-utilde1} by choosing $\lambda_h,\lambda_v$ sufficiently large (depending on $R_h$ and $R_v$, respectively).

In order to show \eqref{eq:visc-Rh} and \eqref{eq:visc-Rv}, it remains to estimate the terms in \eqref{eq:visc-additionalterms}. Using Lemmas~\ref{lemma:hid}, \ref{Lpverticalbound} and \ref{lemma:essential-besov}, as well as $\lVert\overline{u}_p \rVert_{L^1 (W^{1,1})} \lesssim \lambda_h^{-\gamma_h}$ and $\lVert\widetilde{u}_p \rVert_{L^1 (W^{1,1})} \lesssim \lambda_v^{-\gamma_v}$ which we established above, we find
\begin{align*}
    \lVert \mathcal{R}_h (\nu_h^* \Delta_h \overline{u}_p ) \rVert_{L^1 (L^1) } &= \nu_h^* \lVert \nabla_h \overline{u}_p + \nabla_h \overline{u}_p^T \rVert_{L^1 (L^1)} \lesssim \lVert \overline{u}_p \rVert_{L^1 (W^{1,1})} \lesssim \lambda_h^{-\gamma_h}, \\
    \lVert \mathcal{R}_v (\nu_h^* \Delta_h \widetilde{u}_p + \nu_v^* \partial_{zz} (\overline{u}_p + \widetilde{u}_p) ) \rVert_{L^1 (B^{-1}_{1,\infty}) } &\lesssim \lVert \Delta_h \widetilde{u}_p \rVert_{L^1 (B^{-1}_{1,\infty}) } + \lVert \partial_{z} (\overline{u}_p + \widetilde{u}_p) \rVert_{L^1 (B^{-1}_{1,\infty}) } \\
    &\lesssim \lVert \nabla_h \widetilde{u}_p \rVert_{L^1 (L^1) } + \lVert \overline{u}_p + \widetilde{u}_p \rVert_{L^1 (L^1) } \\
    &\lesssim \lVert \overline{u}_p \rVert_{L^1 (W^{1,1} )} + \lVert \widetilde{u}_p \rVert_{L^1 (W^{1,1} )} \lesssim \lambda_h^{-\gamma_h} + \lambda_v^{-\gamma_v}.
\end{align*}
\end{proof}

\section{Two-dimensional hydrostatic Euler equations} \label{2Dsection} 

In this section we will develop a convex integration scheme for the two-dimensional hydrostatic Euler equations \eqref{2Dhydreuler1}-\eqref{2Dhydreuler3}, which is somewhat different in nature to the scheme in the three-dimensional case. A similar scheme will be established in section~\ref{prandtlsection} for the (two-dimensional) Prandtl equations \eqref{eq:prandtl-u}-\eqref{eq:prandtl-div}.

In two dimensions the hydrostatic Euler-Reynolds system \eqref{eq:hyd-er-u}-\eqref{eq:hyd-er-div} reduces to
\begingroup
\allowdisplaybreaks
\begin{align*}
    \partial_t u + u \partial_{x_1} u + w \partial_z u + \partial_{x_1} p &= \partial_{x_1} R_h + \partial_z R_v, \\
    \partial_z p &= 0,  \\
    \partial_{x_1} u + \partial_z w &=0 . 
\end{align*}
\endgroup
We observe that in contrast to the three-dimensional case $u$, $R_h$ and $R_v$ are now just scalar quantities, where $R_h$ does not depend on $z$ and $R_v$ is mean-free with respect to $z$. The former allows to include $R_h$ as part of the pressure. In other words by setting 
\begin{equation*}
    p' = p - R_h
\end{equation*}
we may assume without loss of generality that $R_h = 0$ (up to a redefinition of the pressure). So all in all the two-dimensional hydrostatic Euler-Reynolds system we will work with, reads 
\begingroup
\allowdisplaybreaks
\begin{align}
    \partial_t u + u \partial_{x_1} u + w \partial_z u + \partial_{x_1} p &= \partial_z R_v, \label{eq:2D-er-u}\\
    \partial_z p &= 0, \label{eq:2D-er-p} \\
    \partial_{x_1} u + \partial_z w &=0 . \label{eq:2D-er-div}
\end{align}
\endgroup
with unknowns $u,w,p$ and $R_v$. As $R_h$ is no longer there, the only task is to minimise $R_v$. For this reason we will only have a baroclinic perturbation $\widetilde{u}_p$ in Proposition~\ref{prop:2D-perturbative} below. 

In this section, we will prove the following version of the inductive proposition (cf. Proposition \ref{perturbativeproposition}). Theorem~\ref{2Dtheorem} then follows exactly as Theorem~\ref{mainresult}.

\begin{proposition} \label{prop:2D-perturbative}
Suppose $(u,w,p,R_v)$ is a smooth solution of the two-di\-men\-sio\-nal hydrostatic Euler-Reynolds system \eqref{eq:2D-er-u}-\eqref{eq:2D-er-div}, which is well-prepared with associated time interval $I$ and parameter $\tau>0$. Moreover, consider parameters $1 \leq q_2, q_3 \leq \infty$ and $0<s_3 $ which satisfy the constraints in \eqref{2dconstraints}. Finally, let $\delta,\epsilon > 0$ be arbitrary. Then there exists another smooth solution $(u + \widetilde{u}_p, w + w_p, p+P, R_{v,1} )$ of the two-dimensional hydrostatic Euler-Reynolds system \eqref{eq:2D-er-u}-\eqref{eq:2D-er-div} which is well-prepared with respect to the same time interval $I$ and parameter $\tau/2$, and has the following properties:
\begin{enumerate}
    \item $(\widetilde{u}_p,w_p)(x,t)=(0,0)$ whenever $\dist (t, I^c) \leq \tau/2$.
    \item It satisfies the following estimates
    \begingroup
        \allowdisplaybreaks
        \begin{align}
            \lVert R_{v,1} \rVert_{L^1 (B^{-1}_{1,\infty})} &\leq \epsilon, \label{eq:2d-main-Rv} \\ 
            \lVert \widetilde{u}_p \rVert_{L^{q_2-} (L^2)} &\leq \epsilon, \label{eq:2d-main-u1} \\
            \lVert \widetilde{u}_p \rVert_{L^{q_3-} (H^{s_3})} &\leq \epsilon, \label{eq:2d-main-u2} \\
            \lVert w_p \rVert_{L^{q_2'-} (L^2)} &\leq \epsilon, \label{eq:2d-main-w1} \\
            \lVert w_p \rVert_{L^{q_3'-} (H^{-s_3})} &\leq \epsilon. \label{eq:2d-main-w2}
        \end{align} 
    \endgroup
    \item Moreover, we have that
    \begin{equation} \label{eq:2d-main-wu}
        \lVert w_p \widetilde{u}_p + w \widetilde{u}_p + w_p u \rVert_{L^1 (B^{-1}_{1,\infty})} \lesssim \lVert R_v \rVert_{L^1 (B^{-1}_{1,\infty})}.
    \end{equation}
\end{enumerate}
\end{proposition}


\subsection{Preliminaries}
Due to the fact that there is no longer an error $R_h$ in the two-dimensional hydrostatic Euler-Reynolds system \eqref{eq:2D-er-u}-\eqref{eq:2D-er-div}, there is no need for a barotropic perturbation $\overline{u}_p$. Thus there is only a baroclinic perturbation $\widetilde{u}_p$ in Proposition~\ref{prop:2D-perturbative}. For this reason only the `vertical' parameters $\mu_v,\sigma_v,\kappa_v,\nu_v$ are used in sections \ref{2Dsection} and \ref{prandtlsection}. Regarding the two-dimensional hydrostatic Euler equations \eqref{2Dhydreuler1}-\eqref{2Dhydreuler3}, we need the following version of Lemma~\ref{parameterlemma}.

\begin{lemma} \label{2Dparameterlemma} 
Let $1 \leq q_2,  q_3 \leq \infty$ and $0<s_3$ satisfy the constraints \eqref{2dconstraints}. Then we can choose $a_v,b_v,c_v>0$ in \eqref{parameters} with the property that there exists $\gamma_v>0$ such that 
\begin{align} 
    \kappa_v^{2/q_2 - 1} \sigma_v \mu_v &\leq \lambda_v^{-\gamma_v}, \label{parameterineq-add-2d} \\
    \sigma_v^{-1} \nu_v \kappa_v^{1/q_2} \mu_v^{-3/2} &\leq \lambda_v^{-\gamma_v}, \label{parameterineq-improved-2d} 
\end{align}
in addition to \eqref{parameterineq3}, \eqref{parameterineq5} and $\mu_v,\sigma_v,\kappa_v,\nu_v\geq \lambda_v^{\gamma_v}$.
\end{lemma}

\begin{proof} 
Similar to the proof of Lemma~\ref{parameterlemma}, it suffices to show that there is a choice of $a_v,b_v,c_v>0$ such that 
\begingroup
\allowdisplaybreaks
\begin{align}
    -\bigg( \frac{1}{q_2} - \frac{1}{q_3} \bigg) c_v - s_3 (b_v+1) &=0 , \label{eq:2p1}\\
    b_v - a_v - \frac{1}{q_2} c_v + \frac{3}{2} &>0, \label{eq:2p2}\\
    -\left(\frac{2}{q_2} - 1\right) c_v - b_v - 1 &>0 . \label{eq:2p3}
\end{align}
\endgroup
We know from \eqref{2dconstraints} that 
$$
    \frac{3}{2} - \frac{s_3}{q_2\left(\frac{1}{q_3} - \frac{1}{q_2}\right)} > 0.
$$
Hence we can choose $0<b_v \ll 1$ such that 
$$
    b_v \left( 1 -\frac{s_3}{q_2\left(\frac{1}{q_3} - \frac{1}{q_2}\right)}\right) + \frac{3}{2} - \frac{s_3}{q_2\left(\frac{1}{q_3} - \frac{1}{q_2}\right)} > 0.
$$
By setting 
$$
    c_v := \frac{s_3 \left(b_v +1\right)}{\frac{1}{q_3} - \frac{1}{q_2}},
$$
the latter implies that 
$$
    b_v - \frac{1}{q_2} c_v + \frac{3}{2} >0,
$$
which in turn allows for the choice of a small $a_v>0$ such that \eqref{eq:2p2} holds. The definition of $c_v$ immediately implies \eqref{eq:2p1}. Finally \eqref{2dconstraints} guarantees that 
$$
    \left(1-\frac{2}{q_2}\right)\frac{s_3}{\left(\frac{1}{q_3} - \frac{1}{q_2}\right)} - 1 > 0,
$$
hence 
$$
    (b_v+1) \left(\left(1-\frac{2}{q_2}\right)\frac{s_3}{\left(\frac{1}{q_3} - \frac{1}{q_2}\right)} - 1 \right) > 0,
$$
which is equivalent to \eqref{eq:2p3}.
\end{proof}

\begin{remark} 
Notice that compared to Lemma~\ref{parameterlemma} we have replaced \eqref{parameterineq4} by \eqref{parameterineq-improved-2d}, where the latter is a weaker restriction than the former. Indeed it is simple to see that \eqref{parameterineq4} implies \eqref{parameterineq-improved-2d} provided $q_2>2$. Note furthermore that \eqref{parameterineq-improved-2d} suffices to prove Lemma~\ref{timederivativeestimate}. Moreover, the additional inequality \eqref{parameterineq-add-2d} is needed to deal with an additional spatial corrector for the vertical velocity.
\end{remark}

Moreover, we need a two-dimensional version of the vertical inverse divergence operator $\mathcal{R}_v$, cf. Definition~\ref{defn:vid}.
\begin{definition} \label{defn:vid-2D}
    We define the map\footnote{Again we denote the space of all functions in $C^\infty(\T^2;\R)$ which have zero-mean with respect to $z$ by $C^\infty_{0,z}(\T^2;\R)$.} $\mathcal{R}_v: C^\infty_{0,z}(\T^2;\R) \to C^\infty(\T^2;\R)$ by 
    \begin{equation}
        (\mathcal{R}_v v) (x_1,z) \coloneqq \int_0^z v (x_1,z') \dz' - \int_0^1 \int_0^{z'} v (x_1,z'') \dz'' \dz'. 
    \end{equation} 
\end{definition}

Note that the vertical inverse divergence defined in Definition~\ref{defn:vid-2D} has the same properties as stated in Lemma~\ref{Lpverticalbound}.

Regarding the building blocks we will use the following version of Proposition~\ref{prop:mikado-v}.
\begin{proposition} \label{prop:mikado-v-2D} 
    There exists a function $\phi\in C^\infty(\T;\R)$ (referred to as the Mikado density) depending on a parameter $\mu_v$, with the following properties. 
    \begin{enumerate}
        \item The function $\phi$ has zero mean. Moreover $\int_{\T} \phi^2 \dx = 1$.
        
        \item There exists $\Omega\in C^\infty(\T;\R)$ with zero mean such that $\phi=\partial_{x_1} \Omega$. 
        
        \item For all $s\geq 0$ and $1\leq p\leq \infty$ the following estimates hold:
        \begin{align*}
            \|\phi\|_{W^{s,p}(\T)} &\lesssim \mu_v^{\frac{1}{2}-\frac{1}{p} + s} , \\
            \|\Omega\|_{W^{s,p}(\T)} &\lesssim \mu_v^{-\frac{1}{2}-\frac{1}{p}+s}. 
        \end{align*} 
        Here the implicit constant may depend on $s,p$ but it does not depend on $\mu_v$.
    \end{enumerate}
\end{proposition}

Similar to Proposition~\ref{prop:mikado-v}, Proposition~\ref{prop:mikado-v-2D} can be proven as in \cite[Section~4.1]{cheskidovluo3}. In fact, the function $\phi$ in Proposition~\ref{prop:mikado-v-2D} coincides with the function $\phi_2$ from Proposition~\ref{prop:mikado-v}.

Analogously to Lemma~\ref{mikadobounds} the estimates 
\begin{align}
    \| \phi (\sigma \cdot ) \|_{W^{s,p}} &\lesssim (\sigma \mu_v)^s \mu_v^{\frac{1}{2} - \frac{1}{p}},  \label{eq:est-phi-2D} \\
    \| \Omega (\sigma \cdot ) \|_{W^{s,p}} &\lesssim (\sigma \mu_v)^s \mu_v^{-\frac{1}{2} - \frac{1}{p}}, \label{eq:est-omega-2D}
\end{align} 
hold for all $\sigma\in \mathbb{N}$, $s\geq 0$ and $1\leq p\leq \infty$.

In what follows we will always write $\phi(x)$, but let us clarify that actually $\phi$ only depends on $x_1$.

Finally, we remark that for the two-dimensional convex integration scheme developed in this section, we will work with the same temporal intermittency functions as defined in section~\ref{subsubsec:vtif}.

\subsection{Definition of the perturbation}
The velocity perturbation will be written as 
\begin{align*} 
    \widetilde{u}_p &= u_{p,v} + u_{t,v}, \\
    w_p &= w_{p,v} + w_{c,v} + w_{t,v},
\end{align*}
so in contrast to the three-dimensional case, a spatial corrector $u_{c,v}$ is not needed, however we will have spatial corrector $w_{c,v}$ for the vertical velocity.

We make the following choice for the principal part of the perturbation 
\begin{align}
    u_{p,v}(x,t) &\coloneqq - g_{v,2}^-  (\nu_v t) \theta(t) R_{v}(x,t) \phi (\sigma_v x), \label{eq:2D-def-upv} \\
    w_{p,v}(x,t) &\coloneqq g_{v,2}^+ (\nu_v t) \theta(t)  \phi (\sigma_v x) . \label{eq:2D-def-wpv}
\end{align}

\begin{remark} \label{rem:endpoint-2D-proof}
In order to achieve endpoint time integrability for $w$ (cf. Remark~\ref{rem:endpoint-2D}) one has to multiply the right-hand side of \eqref{eq:2D-def-wpv} by $\lVert R_v \rVert_{L^1(B^{-1}_{1,\infty})}$ and divide the right-hand side of \eqref{eq:2D-def-upv} by the same factor. Further modifications are straightforward. In order to get endpoint time integrability for $u$ one proceeds as described in Remarks~\ref{rem:endpoint}, \ref{rem:uptilde-endpoint} and \ref{rem:uptilde-endpoint-proof}.
\end{remark}

Then we introduce a corrector for the vertical velocity in order to get a divergence-free perturbation:
\begin{equation} \label{eq:2D-def-wcv}
    w_{c,v} \coloneqq \mathcal{R}_v \Big( g_{v,2}^-  (\nu_v t) \theta \partial_{x_1} (R_{v} \phi (\sigma_v x)) \Big),
\end{equation} 
where $\mathcal{R}_v$ is now given by Definition~\ref{defn:vid-2D}. Note that $R_v$ is mean-free with respect to $z$ and hence the operator $\mathcal{R}_v$ can be applied to the expression in \eqref{eq:2D-def-wcv}. With Lemma~\ref{Lpverticalbound} it is simple to see that $\partial_{x_1} u_{p,v} + \partial_z w_{p,v} = 0$.

Finally, we introduce a temporal corrector of the form
\begin{equation}
    u_{t,v} \coloneqq \nu_v^{-1} h_{v,2} (\nu_v t) \partial_z R_v.
\end{equation}
To keep the whole velocity field divergence-free, we now must introduce a temporal corrector for the vertical velocity
\begin{equation}
    w_{t,v} \coloneqq -\nu_v^{-1} h_{v,2} (\nu_v t) \partial_{x_1} R_v.
\end{equation}

We observe that $\partial_{x_1} \widetilde{u}_p + \partial_z w_p = 0$, item 1 of Proposition~\ref{prop:2D-perturbative} holds and $u_{p,v}$ and $u_{t,v}$ are indeed mean-free with respect to $z$.

\subsection{The new Reynolds stress tensor} 
As in the three-dimensional scheme, the new Reynolds stress tensor $R_{v,1}$ will be written as 
$$
    R_{v,1} = R_{\osc,v} + R_{\lin,v} + R_{\cor,v}. 
$$

First, we set $R_{\osc,v}=R_{\osc,x,v}+R_{\osc,t,v}$, where
\begin{align*}
    R_{\osc,x,v} &\coloneqq - g_{v,2}^- (\nu_v t) g_{v,2}^+ (\nu_v t) \theta^2  R_v \bigg( \phi^2 (\sigma_v x) - \int_{\mathbb{T}} \phi^2 (x) \dx \bigg), \\
    R_{\osc,t,v} &\coloneqq \nu_v^{-1} h_{v,2} (\nu_v t) \partial_t R_v. 
\end{align*}
Exactly as in Lemma~\ref{lemma:Roscv} one can show that 
\begin{equation} \label{eq:Roscv-2D}
    \partial_t u_{t,v} + \partial_z ( w_{p,v} u_{p,v} + R_v) = \partial_z R_{\osc,v}.
\end{equation}

Next, we define the linear and corrector errors by 
\begin{align*} 
    R_{\lin,v} \coloneqq \mathcal{R}_v\bigg[ &\partial_t u_{p,v} + 2\partial_{x_1} \bigg( \reallywidetilde{u \big(u_{p,v}+u_{t,v}\big)}\bigg) + \partial_z \bigg( w\big(u_{p,v}+u_{t,v}\big) + \big(w_{p,v}+w_{c,v}+w_{t,v}\big) u \bigg) \bigg]
\end{align*} 
and 
\begin{align*} 
   R_{\cor,v} \coloneqq \mathcal{R}_v\bigg[ &\partial_{x_1} \bigg( \reallywidetilde{\big(u_{p,v}+u_{t,v}\big)^2 } \bigg) + \partial_z \bigg( \big(w_{c,v}+w_{t,v}\big) \big( u_{p,v}+u_{t,v}\big) + w_{p,v} u_{t,v} \bigg) \bigg].
\end{align*} 
Note that the arguments of the operator $\mathcal{R}_v$ are indeed mean-free with respect to $z$.

In the three-dimensional convex integration scheme we introduced a pressure perturbation, see section~\ref{subsec:new-pressure}. An analogue of this perturbation is not needed in two dimensions as there is no barotropic perturbation. However, the pressure has to absorb some terms that were covered by the horizontal Reynolds stress tensor in the three-dimensional scheme. To this end we define
\begin{equation}
    P \coloneqq - 2 \overline{u \big(u_{p,v}+u_{t,v}\big)} - \overline{\big(u_{p,v}+u_{t,v}\big)^2}.
\end{equation}
We observe that $\partial_z P = 0$.

Let us finally remark that it is straightforward to see that $R_{v,1}$ is mean-free with respect to $z$, that $R_{v,1}(x,t)=0$ whenever $\dist(t,I^c)\leq \frac{\tau}{2}$, and that $(u + \widetilde{u}_p, w + w_p, p+P, R_{v,1} )$ solves \eqref{eq:2D-er-u}.

\subsection{Estimates on the perturbation} 
Now we claim the following estimates on the perturbation.
\begin{lemma} \label{2dprincperturblemma}
If $\lambda_v$ is chosen sufficiently large (depending on $R_v$), then we have that 
\begin{align}
    \lVert u_{p,v} \rVert_{L^{q_2-} (L^2)} &\lesssim \lambda_v^{-\gamma_v}, \label{eq:est-upv1-2D} \\
    \lVert u_{p,v} \rVert_{L^{q_3-} (H^{s_3})} &\lesssim \lambda_v^{-\gamma_v}, \label{eq:est-upv2-2D} \\
    \lVert w_{p,v} \rVert_{L^{q_2'-} (L^2)} &\lesssim \lambda_v^{-\gamma_v}, \label{eq:est-wpv1-2D} \\
    \lVert w_{p,v} \rVert_{L^{q_3'-} (H^{-s_3})} &\lesssim \lambda_v^{-\gamma_v}, \label{eq:est-wpv2-2D} \\
    \lVert w_{p,v} u_{p,v} \rVert_{L^1 (B^{-1}_{1,\infty} )} &\lesssim \lVert R_v \rVert_{L^1 (B^{-1}_{1,\infty})}. \label{eq:est-wpvupv-2D} 
\end{align} 
\end{lemma}

\begin{proof}
In fact the proof of \eqref{eq:est-upv1-2D}-\eqref{eq:est-wpv2-2D} can be taken verbatim from Lemma~\ref{verticalprinc}. Estimate \eqref{eq:est-wpvupv-2D} can also be proven in a similar way as in Lemma~\ref{verticalprinc}. To this end we need a version of Lemma~\ref{lemma:prod}, namely the estimate
\begin{equation} \label{eq:801}
    \bigg\lVert \bigg( \phi^2 (\sigma_v \cdot) - \int_{\mathbb{T}} \phi^2 \dx \bigg) R_v \bigg\rVert_{L^p (B^{-1}_{1,\infty})} \lesssim \sigma_v^{-1},
\end{equation}
which holds for any $1\leq p\leq \infty$. To show \eqref{eq:801}, we mimic the proof of Lemma~\ref{lemma:prod}. This requires to introduce a one-dimensional horizontal inverse divergence operator $\mathcal{R}_h$, which is defined analogously to the vertical inverse divergence, cf. Definition~\ref{defn:vid-2D}, specifically
\begin{equation*}
    \mathcal{R}_h v (x) \coloneqq \int_0^x v (x') \dx' - \int_0^1 \int_0^{x'} v (x'') \dx'' \dx'. 
\end{equation*} 
Consequently $\mathcal{R}_h$ has the properties stated in Lemma~\ref{Lpverticalbound}. Using property \eqref{eq:lem-Rv-est3} we are able to prove \eqref{eq:801} and thus \eqref{eq:est-wpvupv-2D} follows. 
\end{proof}

The spatial and temporal correctors can be estimated as follows.

\begin{lemma} \label{2dcorrectorlemma}
    The spatial and temporal correctors satisfy the following estimates
    \begin{align} 
        \lVert w_{c,v} \rVert_{L^{q_2'} (L^2)} + \lVert w_{c,v} \rVert_{L^{q_3'} (H^{-s_3} )} &\lesssim \lambda_v^{-\gamma_v}, \label{eq:est-wcv-2D} \\
        \lVert u_{t,v} \rVert_{L^\infty (W^{n,\infty} )} &\lesssim \lambda_v^{-\gamma_v}, \label{eq:est-utv-2D} \\
        \lVert w_{t,v} \rVert_{L^\infty (W^{n,\infty} )} &\lesssim \lambda_v^{-\gamma_v},  \label{eq:est-wtv-2D}
    \end{align}
    where $n\in \mathbb{N}$ is arbitrary, and the implicit constant may depend on $n$.
\end{lemma}

\begin{proof} 
We only prove \eqref{eq:est-wcv-2D}, as the proof of the estimates for the temporal correctors \eqref{eq:est-utv-2D}, \eqref{eq:est-wtv-2D} is similar to the proof of Lemma \ref{temporallemma}. We obtain using Lemmas~\ref{Lpverticalbound} and \ref{2Dparameterlemma}, and equations~\eqref{eq:est-gvm}, \eqref{eq:est-phi-2D}
\begin{align*}
    \lVert w_{c,v} \rVert_{L^{q_2'} (L^2)} &\lesssim \lVert g_{v,2}^- (\nu_v \cdot) \rVert_{L^{q_2'}} \lVert \theta\rVert_{L^\infty} \lVert R_v \rVert_{L^\infty (W^{1,\infty})} \lVert \phi (\sigma_v \cdot) \rVert_{H^{1}} \\
    &\lesssim \kappa_v^{1/q_2 - 1/q_2'} \sigma_v \mu_v = \kappa_v^{2/q_2 - 1} \sigma_v \mu_v \leq \lambda_v^{-\gamma_v}. 
\end{align*}
Analogously 
\begin{align*} 
    \lVert w_{c,v} \rVert_{L^{q_3'} (H^{-s_3})} &\lesssim \lVert g_{v,2}^- (\nu_v \cdot) \rVert_{L^{q_3'}} \lVert \theta\rVert_{L^\infty} \lVert R_v \rVert_{L^\infty (W^{1,\infty})} \lVert \phi (\sigma_v \cdot) \rVert_{H^{1-s_3}} \\
    &\lesssim \kappa_v^{1/q_2 - 1/q_3'} (\sigma_v \mu_v)^{1-s_3} = \big(\kappa_v^{2/q_2 - 1} \sigma_v \mu_v \big)\big(\kappa_v^{1/q_3 - 1/q_2} (\sigma_v\mu_v)^{-s_3} \big)\leq \lambda_v^{-\gamma_v}. 
\end{align*} 
Here we have used that $\mathcal{R}_v$ is bounded in $H^{-s_3}$, see \eqref{eq:lem-Rv-est2} and keeping in mind that $H^{-s_3}=B^{-s_3}_{2,2}$, and $s_3\leq 1$.
\end{proof}

Lemmas~\ref{2dprincperturblemma} and \ref{2dcorrectorlemma} show that estimates \eqref{eq:2d-main-u1}-\eqref{eq:2d-main-wu} hold. 

\subsection{Estimates on the Reynolds stress tensor} 
In order to finish the proof of Proposition~\ref{prop:2D-perturbative}, it remains to show \eqref{eq:2d-main-Rv}.

\begin{lemma} \label{lemma:2d-error}
    If $\lambda_v$ is chosen sufficiently large (depending on $R_v$), then the errors satisfy the following estimates 
    \begingroup
    \allowdisplaybreaks
    \begin{align*}
        \lVert R_{\osc,v} \rVert_{L^1 (B^{-1}_{1,\infty})} &\leq \frac{\epsilon}{3}, \\
        \lVert R_{\cor,v} \rVert_{L^1 (B^{-1}_{1,\infty})} &\leq \frac{\epsilon}{3}, \\
        \lVert R_{\lin,v} \rVert_{L^1 (B^{-1}_{1,\infty})} &\leq \frac{\epsilon}{3}.
    \end{align*}
    \endgroup
\end{lemma}

\begin{proof}
Lemma~\ref{lemma:2d-error} can be proven similarly to Lemmas~\ref{verticaloscestimate}-\ref{advectionestimate}, with only small modifications: To obtain the required estimate for $R_{\osc,x,v}$ one has to use \eqref{eq:801} rather than Lemma~\ref{lemma:prod}. Moreover, since in two dimensions there is no spatial corrector $u_{c,v}$, the time derivative part of the linear error $R_{\lin,t,v}$ must be estimated slightly differently compared to Lemma~\ref{timederivativeestimate}. To this end we write 
\begin{align*}
    u_{p,v} &= - g_{v,2}^-  (\nu_v t) \theta R_{v} \phi (\sigma_v x) \\
    & = - \sigma_v^{-1} g_{v,2}^-  (\nu_v t) \theta R_{v} \partial_{x_1} \big[\Omega (\sigma_v x)\big] \\
    & = \partial_{x_1} \big[ - \sigma_v^{-1} g_{v,2}^-  (\nu_v t) \theta R_{v} \Omega (\sigma_v x) \big] + \sigma_v^{-1} g_{v,2}^-  (\nu_v t) \theta (\partial_{x_1} R_{v}) \Omega (\sigma_v x) .
\end{align*}
Hence we find (by using inequality \eqref{parameterineq-improved-2d})
\begingroup
\allowdisplaybreaks
\begin{align*}
    &\Big\lVert \mathcal{R}_v \big[ \partial_t u_{p,v} \big] \Big\rVert_{L^1 (B^{-1}_{1,\infty})} \\
    &\lesssim \lVert \partial_t u_{p,v} \rVert_{L^1 (B^{-1}_{1,\infty})} \\
    &\lesssim \Big\lVert \partial_t \big[ \sigma_v^{-1} g_{v,2}^-  (\nu_v \cdot) \theta R_{v} \Omega (\sigma_v \cdot) \big] \Big\rVert_{L^1(L^1)} + \Big\lVert \partial_t \big[ \sigma_v^{-1} g_{v,2}^-  (\nu_v \cdot) \theta (\partial_{x_1} R_{v}) \Omega (\sigma_v \cdot) \big] \Big\rVert_{L^1(L^1)} \\
    &\lesssim \sigma_v^{-1} \lVert g_{v,2}^-  (\nu_v \cdot) \rVert_{W^{1,1}} \lVert \theta \rVert_{W^{1,\infty}} \lVert R_v\rVert_{W^{1,\infty}(W^{1,\infty})} \lVert \Omega (\sigma_v \cdot) \rVert_{L^1} \\
    &\lesssim \sigma_v^{-1} \nu_v \kappa_v^{1/q_2} \mu_v^{-3/2} \leq \lambda_v^{\gamma_v}.
\end{align*}
\endgroup
\end{proof}
This finishes the proof of Proposition~\ref{prop:2D-perturbative}.

\section{The two-dimensional Prandtl equations} \label{prandtlsection}

In this section we will study the following Prandtl-Reynolds system
\begin{align}
    \partial_t u - \nu_v^* \partial_{zz} u + u \partial_{x_1} u + w \partial_z u + \partial_{x_1} p &= \partial_z R_v, \label{eq:prandtl-er-u} \\ 
    \partial_z p &= 0, \label{eq:prandtl-er-p} \\
    \partial_{x_1} u + \partial_z w &=0 , \label{eq:prandtl-er-div}
\end{align}
with unknowns $(u,w,p,R_v)$. As in section \ref{2Dsection} there is no horizontal Reynolds stress tensor $R_h$. In this setting we have the following version of the inductive proposition (cf. Proposition~\ref{perturbativeproposition}). As before, the proof of Theorem \ref{prandtltheorem} then works in the same way as the proof of Theorem~\ref{mainresult}.

\begin{proposition} \label{prop:prandtl-perturbative}
Suppose $(u,w,p,R_v)$ is a smooth solution of the Prandtl-Reynolds system \eqref{eq:prandtl-er-u}-\eqref{eq:prandtl-er-div}, which is well-prepared with associated time interval $I$ and parameter $\tau>0$. Moreover, consider parameters $1 \leq q_2, q_3 \leq \infty$ and $0<s_3$ which satisfy the constraints in \eqref{prandtlconstraints}. Finally, let $\delta,\epsilon > 0$ be arbitrary. Then there exists another smooth solution $(u + \widetilde{u}_p, w + w_p, p+P , R_{v,1} )$ of the Prandtl-Reynolds system \eqref{eq:prandtl-er-u}-\eqref{eq:prandtl-er-div}, which is well-prepared with respect to the same time interval $I$ and parameter $\tau/2$, and has the following properties:
\begin{enumerate}
    \item $(\widetilde{u}_p,w_p)(x,t)=(0,0)$ whenever $\dist (t, I^c) \leq \tau/2$. 
    \item The perturbation and the new Reynolds stress tensor satisfy the following estimates
    \begingroup
        \allowdisplaybreaks
        \begin{align}
            \lVert R_{v,1} \rVert_{L^1 (B^{-1}_{1,\infty})} &\leq \epsilon, \label{eq:prandtl-Rv} \\
            \lVert \widetilde{u}_p \rVert_{L^{1} (W^{1,1})} &\leq \epsilon, \label{eq:prandtl-utilde1} \\ 
            \lVert \widetilde{u}_p \rVert_{L^{q_2-} (L^2)} &\leq \epsilon, \label{eq:prandtl-utilde2} \\
            \lVert \widetilde{u}_p \rVert_{L^{q_3-} (H^{s_3})} &\leq \epsilon, \label{eq:prandtl-utilde3} \\
            \lVert w_p \rVert_{L^{q_2'-} (L^2)} &\leq \epsilon, \label{eq:prandtl-w1}\\
            \lVert w_p \rVert_{L^{q_3'-} (H^{-s_3})} &\leq \epsilon, \label{eq:prandtl-w2}
        \end{align} 
    \endgroup
    \item Finally, the products of the vertical and horizontal perturbations satisfy that
    \begin{equation} \label{eq:prandtl-prod}
        \lVert w_p \widetilde{u}_p + w \widetilde{u}_p + w_p u \rVert_{L^1 (B^{-1}_{1,\infty})} \lesssim \lVert R_v \rVert_{L^1 (B^{-1}_{1,\infty})}.
    \end{equation}
\end{enumerate}
\end{proposition}


\begin{proof}
In order to prove Proposition~\ref{prop:prandtl-perturbative} we modify the proof of Proposition~\ref{prop:2D-perturbative} in the same way as we did in the three-dimensional case, cf. proof of Proposition~\ref{viscousperturbativeprop}. Specifically we choose $\widetilde{u}_p$, $w_p$ and $P$ as in the proof of Proposition~\ref{prop:2D-perturbative}, while the linear error now contains the additional term 
\begin{equation} \label{eq:902}
    \mathcal{R}_v (\nu_v^* \partial_{zz} \widetilde{u}_p ).
\end{equation}
Then it follows from section \ref{2Dsection} that estimates \eqref{eq:prandtl-utilde2}-\eqref{eq:prandtl-prod} hold. In order to show \eqref{eq:prandtl-utilde1} we proceed as in the proof of Proposition~\ref{viscousperturbativeprop}. To this end we need the additional parameter estimate
$$
    \kappa_v^{1/q_2 - 1} \sigma_v \mu_v^{1/2} \leq \lambda_v^{-\gamma_v} . 
$$
We can achieve this as soon as 
\begin{equation} \label{eq:901}
    \frac{s_3\left(1 - \frac{1}{q_2}\right)}{\frac{1}{q_3} - \frac{1}{q_2}} > \frac{1}{2}, 
\end{equation}
see the proof of Lemma~\ref{viscousparameter}. Using 
$$
    \frac{1}{1 - \frac{2}{q_2}} > \frac{1}{2\left(1 - \frac{1}{q_2}\right)}
$$
we see that \eqref{eq:901} holds according to \eqref{prandtlconstraints}. 

It remains to estimate the additional term in \eqref{eq:902}. This works exactly as in the proof of Proposition~\ref{viscousperturbativeprop}.
\end{proof}

\section*{Acknowledgements}
D.W.B. acknowledges support from the Cambridge Trust, the Cantab Capital Institute for Mathematics of Information and the Prince Bernhard Culture fund. S.M. acknowledges support from the Alexander von Humboldt foundation. The work of E.S.T. has benefited from the inspiring environment of the CRC 1114 "Scaling Cascades in Complex Systems", Project Number 235221301, Project C06, funded by Deutsche Forschungsgemeinschaft (DFG). The authors would like to thank the Isaac Newton Institute for Mathematical Sciences, Cambridge, for support and hospitality during the programme ``Mathematical aspects of turbulence: where do we stand?'' where part of this work was undertaken and supported by EPSRC grant no.~EP/K032208/1.

\section*{Declaration}
Data sharing is not applicable to this article as no datasets were generated or analysed during the current study.

\begin{appendices}

\section{Littlewood-Paley theory, Besov spaces and paradifferential calculus} \label{paradifferentialappendix}
In this appendix, we state some basic definitions from Littlewood-Paley theory and paradifferential calculus, which will be used throughout the paper. More details can be found in \cite{bahouri,mourrat1,leoni,sawano}. 

\subsection{Littlewood-Paley theory and Besov spaces} \label{subsec:ap-littlewood-besov}
We first introduce a dyadic partition of unity $\{ \rho_j \}_{j=-1}^\infty$ as follows
\begin{equation*}
    \rho_0 (\xi) = \rho (\xi), \quad \rho_j (\xi) = \rho (2^{-j} \xi) \text{ for }j=1,2,\ldots, \quad \rho_{-1} (\xi) = 1 - \sum_{j=0}^\infty \rho_j (\xi).
\end{equation*}
Then for $f \in \mathcal{S}' (\mathbb{T}^3)$ the Littlewood-Paley blocks are given by
\begin{equation*}
    \widehat{\Delta_j f} (\xi) = \rho_j (\xi) \widehat{f} (\xi), \quad j=-1,0, \ldots.
\end{equation*}
The Besov space $B^{s}_{p,q} (\mathbb{T}^3)$ is then defined in terms of the Littlewood-Paley based norm, which reads for $q < \infty$
\begin{equation*}
    \lVert f \rVert_{B^{s}_{p,q}} \coloneqq \lVert \Delta_{-1} f \rVert_{L^p} + \bigg( \sum_{j=0}^\infty 2^{s j q} \lVert \Delta_j f \rVert_{L^p}^q \bigg)^{1/q}.
\end{equation*}
If $q = \infty$, the norm is defined as follows
\begin{equation*}
    \lVert f \rVert_{B^s_{p,\infty}} \coloneqq \lVert \Delta_{-1} f \rVert_{L^p} + \sup_{j \geq 0} \big( 2^{s j } \lVert \Delta_j f \rVert_{L^p} \big).
\end{equation*}
It is also possible to define Besov spaces using difference quotients. We first define the forward finite difference operator
\begin{equation*} 
    \Delta^1_h f (x) = f (x+h) - f(x).
\end{equation*}
The higher order finite differences are defined inductively. For $m \geq 2$, we define that
\begin{equation*}
    \Delta^m_h f(x) = \Delta^1_h ( \Delta_h^{m-1} f (x)).
\end{equation*}
Let $s > 0$ and $1 \leq p , q \leq \infty$ and let $\lfloor s \rfloor$ denote the integer part of $s$, then Besov norm may be defined as follows (if $q < \infty$)
\begin{equation*}
    \lVert f \rVert_{B^s_{p,q}} = \lVert f \rVert_{L^p} + \bigg( \int_{\mathbb{T}^3} \lVert \Delta_h^{\lfloor s \rfloor + 1} f \rVert_{L^p}^q \frac{{\rm d}h}{\lvert h \rvert^{3 + s q}} \bigg)^{1/q}.
\end{equation*}
If $q = \infty$, the Besov norm is given by
\begin{equation*}
    \lVert f \rVert_{B^s_{p,\infty}} \coloneqq \lVert f \rVert_{L^p} + \sup_{h \in \mathbb{R}^3 \backslash \{ 0 \}} \frac{\lVert \Delta_h^{\lfloor s \rfloor + 1} f \rVert_{L^p}}{\lvert h \rvert^s}.
\end{equation*}
These two different definitions of the Besov norm are equivalent.

\begin{remark}
    Note that the index $h$ in the notation for the finite differences $\Delta_h^m$ represents the shift. This is in contrast to the main body of this paper where the index $h$ always means ``horizontal''.
\end{remark}

\begin{remark} \label{rem:besov-sobolev}
    The reader should note that $B^s_{p,p} (\mathbb{T}^3) = W^{s,p} (\mathbb{T}^3)$ when $s \in \mathbb{R} \backslash \mathbb{Z}$ and $1 \leq p \leq \infty$. This is stated in \cite[~Equation 3.5]{amann} for example. The equivalence of function spaces in the case $0 < s < 1$ can also be found in \cite[~Definition 32.2]{tartar} (when the interpolation space definition of Besov spaces is used) and in \cite[~Proposition 2]{simon} and \cite[~Page 1686]{brasseur} (where the Sobolev spaces are defined using the Sobolev-Slobodeckij seminorm). 
    We note that the case $0 < s < 1$ can easily be extended to all $s > 0$ with $s \notin \mathbb{N}$ by using Theorem 2.3 in \cite{sawano}. Finally, we recall that $B^s_{2,2} (\mathbb{T}^3) = H^s (\mathbb{T}^3)$ for all $s \in \mathbb{R}$ (even when $s$ is an integer), see \cite[~Page 99]{bahouri}. The latter is used several times in this paper.
\end{remark}

Next we recall some essential estimates regarding the Besov norm. 
\begin{lemma} \label{lemma:essential-besov}
    For any $1\leq p,q,q_1,q_2\leq \infty$, $\alpha\in \mathbb{R}$, $\delta>0$ the following estimates hold
    \begin{align}
        \lVert f \rVert_{B^0_{p,\infty}} &\lesssim\lVert f \rVert_{L^p} \lesssim \lVert f \rVert_{B^0_{p,1}} , \label{eq:besov-1}\\
        \lVert f \rVert_{B^\alpha_{p,q_1}} &\lesssim \lVert f \rVert_{B^{\alpha+\delta}_{p,q_2}}, \label{eq:besov-2}\\
        \lVert f \rVert_{B^\alpha_{p,q_1}} &\lesssim \lVert f \rVert_{B^{\alpha}_{p,q_2}}, \quad \text{ if }q_1\geq q_2 \label{eq:besov-2a}\\
        \lVert \partial_i f \rVert_{B^{\alpha-1}_{p,q}} &\lesssim \lVert f \rVert_{B^\alpha_{p,q}}. \label{eq:besov-3}
    \end{align}
\end{lemma}

\begin{proof}
Estimates \eqref{eq:besov-1}, \eqref{eq:besov-2}, \eqref{eq:besov-2a} and \eqref{eq:besov-3} can be found in \cite{sawano} Propositions 2.1, 2.2, 2.3 and Theorem 2.2, respectively.
\end{proof}

Note that \eqref{eq:besov-2} implies that
\begin{equation*}
    \lVert f \rVert_{B^\alpha_{p,q}} \lesssim \lVert f \rVert_{B^{\beta}_{p,q}} 
\end{equation*}
for any $\alpha\leq \beta$.

\subsection{Paradifferential calculus} \label{subsec:ap-littlewood-paraproducts}
We recall Bony's product decomposition
\begin{equation}
f g = T_f g + T_g f + R(f,g).
\end{equation}
The terms $T_f g$ and $T_g f$ are called paraproducts and are given by
\begin{align}
T_f g = \sum_{j=-1}^\infty \sum_{i=-1}^{j-2} \Delta_i f \Delta_j g, \quad T_g f = \sum_{j=-1}^\infty \sum_{i=-1}^{j-2} \Delta_i g \Delta_j f.
\end{align}
The term $R(f,g)$ is referred to as the resonance term and is defined as
\begin{equation}
R(f,g) = \sum_{\lvert k - j \rvert \leq 1} \Delta_k f \Delta_j g. 
\end{equation}
We will also use the notation $T(f,g)\coloneqq T_f (g)$ and $T(g,f)\coloneqq T_g(f)$.

One can estimate the three terms of the product decomposition separately. For the paraproducts we have the following estimates.
\begin{lemma}[Lemma 2.1 in \cite{promel}] \label{paraproduct} 
Let $\alpha,\beta\in \mathbb{R}$ and $1\leq p,p_1,p_2,q,q_1,q_2\leq \infty$ with 
$$
    \frac{1}{p} = \frac{1}{p_1} + \frac{1}{p_2}, \qquad \frac{1}{q} = \frac{1}{q_1} + \frac{1}{q_2}.
$$
\begin{itemize}
    \item For any $f \in L^{p_1} (\mathbb{T}^3) $ and $g \in B^{\beta}_{p_2,q} (\mathbb{T}^3)$ we have that
    \begin{equation*}
        \lVert T_f (g) \rVert_{B^\beta_{p,q}} \lesssim \lVert f \rVert_{L^{p_1}} \lVert g \rVert_{B^\beta_{p_2,q}}.
    \end{equation*}

    \item If $\alpha  < 0$ then for any $f \in B^{\alpha}_{p_1, q_1} (\mathbb{T}^3)$ and $g \in B^\beta_{p_2,q_2} (\mathbb{T}^3)$ we have 
    \begin{equation*}
        \lVert T_f (g) \rVert_{B^{\alpha + \beta}_{p,q}} \lesssim \lVert f \rVert_{B^\alpha_{p_1,q_1}} \lVert g \rVert_{B^\beta_{p_2,q_2}}.
    \end{equation*}
\end{itemize}
\end{lemma}
We recall the following estimate on the resonance term.
\begin{lemma}[Theorem 2.85 in \cite{bahouri}] \label{resonance}
Let $\alpha,\beta\in \mathbb{R}$ and $1\leq p,p_1,p_2,q,q_1,q_2\leq \infty$ with 
$$
    \frac{1}{p} = \frac{1}{p_1} + \frac{1}{p_2}, \qquad \frac{1}{q} = \frac{1}{q_1} + \frac{1}{q_2}.
$$
\begin{itemize}
    \item If $\alpha+\beta>0$ then we have for any $f \in B^{\alpha}_{p_1,q_1} (\mathbb{T}^3)$ and $g \in B^{\beta}_{p_2,q_2} (\mathbb{T}^3)$
    $$
        \lVert R(f,g) \rVert_{B^{\alpha+\beta}_{p,q}} \lesssim \lVert f \rVert_{B^{\alpha}_{p_1,q_1}} \lVert g \rVert_{B^{\beta}_{p_2,q_2}}.
    $$
    
    \item If $\alpha+\beta=0$ and $q=1$ then we have for any $f \in B^{\alpha}_{p_1,q_1} (\mathbb{T}^3)$ and $g \in B^{\beta}_{p_2,q_2} (\mathbb{T}^3)$
    $$
        \lVert R(f,g) \rVert_{B^{0}_{p,\infty}} \lesssim \lVert f \rVert_{B^{\alpha}_{p_1,q_1}} \lVert g \rVert_{B^{\beta}_{p_2,q_2}}.
    $$
\end{itemize}
\end{lemma}

Combining these estimates leads to the following result. 
\begin{lemma} \label{lemma:paradiff-summary}
    Let $\alpha < 0 < \beta$, $\beta + \alpha > 0$, and $1 \leq p_1, p_2, p, q_1,q_2 \leq \infty$ with 
    \begin{equation*}
        \frac{1}{p_1} + \frac{1}{p_2} = \frac{1}{p}.
    \end{equation*}
    \begin{itemize}
    \item We have for any $f \in B^\alpha_{p_1,q_1} (\mathbb{T}^3)$ and $g \in B^\beta_{p_2,q_2} (\mathbb{T}^3)$
    \begin{equation} \label{eq:paradiff-2}
        \lVert fg \rVert_{B^\alpha_{p,q_1}} \lesssim \lVert f \rVert_{B^\alpha_{p_1,q_1}} \lVert g \rVert_{B^\beta_{p_2,q_2}}.
    \end{equation}

    \item Let $\{f_n\},\{g_n\}$ be sequences such that $f_n \to f$ in $B^\alpha_{p_1,q_1} (\mathbb{T}^3)$ and $g_n \to g$ in $B^\beta_{p_2,q_2} (\mathbb{T}^3)$. Then $f_n g_n\to fg$ in $B^\alpha_{p,q_1}$. 
    \end{itemize}
\end{lemma}

\begin{proof}
Estimate \eqref{eq:paradiff-2} is a simple consequence of Lemmas~\ref{paraproduct} and \ref{resonance}, see also \cite[Prop.~A.7]{mourrat1} for a similar proof. 

The convergence claimed in the second bullet point can be easily deduced from \eqref{eq:paradiff-2}. Indeed we have 
\begin{align*}
    \lVert f_n g_n - fg \rVert_{B^\alpha_{p,q_1}} &\lesssim \lVert f_n (g_n - g) \rVert_{B^\alpha_{p,q_1}} + \lVert (f_n - f) g \rVert_{B^\alpha_{p,q_1}} \\
    &\lesssim \lVert f_n \rVert_{B^\alpha_{p_1,q_1}} \lVert g_n-g \rVert_{B^\beta_{p_2,q_2}} + \lVert f_n-f \rVert_{B^\alpha_{p_1,q_1}} \lVert g \rVert_{B^\beta_{p_2,q_2}} \to 0.
\end{align*}
\end{proof}

\subsection{Alternative proof of Lemma~\ref{lemma:prod}} \label{subsec:ap-littlewood-alternative}
In this section we present an alternative proof of Lemma~\ref{lemma:prod} using a paraproduct decompostion and Lemmas~\ref{paraproduct} and \ref{resonance}.

\begin{proof}[Alternative Proof of Lemma~\ref{lemma:prod}]

We start by decomposing the term under consideration in terms of Bony's decomposition, i.e.
\begingroup
\allowdisplaybreaks
\begin{align*}
    &\bigg\lVert \bigg( \phi_k (\sigma_v \cdot) W_k (\sigma_v \cdot) - \int_{\mathbb{T}^2} \phi_k ( x) W_k (x) \dx \bigg) R_{v,k} \bigg\rVert_{L^p (B^{-1}_{1,\infty})} \\
    &\leq \bigg\lVert T \bigg( \bigg( \phi_k (\sigma_v \cdot) W_k (\sigma_v \cdot) - \int_{\mathbb{T}^2} \phi_k ( x) W_k (x) \dx \bigg),  R_{v,k} \bigg) \bigg\rVert_{L^p (B^{-1}_{1,\infty})} \\
    &\qquad + \bigg\lVert T \bigg( R_{v,k}, \bigg( \phi_k (\sigma_v \cdot) W_k (\sigma_v \cdot) - \int_{\mathbb{T}^2} \phi_k ( x) W_k (x) \dx \bigg) \bigg) \bigg\rVert_{L^p (B^{-1}_{1,\infty})} \\
    &\qquad + \bigg\lVert R \bigg( \bigg( \phi_k (\sigma_v \cdot) W_k (\sigma_v \cdot) - \int_{\mathbb{T}^2} \phi_k ( x) W_k (x) \dx \bigg), R_{v,k} \bigg) \bigg\rVert_{L^p (B^{-1}_{1,\infty})}.
\end{align*}
\endgroup 
Using Lemma~\ref{paraproduct} we find 
\begingroup
\allowdisplaybreaks
\begin{align}
    &\bigg\lVert T \bigg( \bigg( \phi_k (\sigma_v \cdot) W_k (\sigma_v \cdot) - \int_{\mathbb{T}^2} \phi_k ( x) W_k (x) \dx \bigg),  R_{v,k} \bigg) \bigg\rVert_{L^p (B^{-1}_{1,\infty})} \notag\\
    &\lesssim \bigg\lVert \phi_k (\sigma_v \cdot) W_k (\sigma_v \cdot) - \int_{\mathbb{T}^2} \phi_k ( x) W_k (x) \dx \bigg\rVert_{B^{-1}_{1,\infty}} \lVert R_{v,k} \rVert_{L^p(B^0_{\infty,\infty})} . \label{eq:paraproduct-alternative} 
\end{align}
\endgroup 
Another application of Lemma~\ref{paraproduct} yields
\begingroup
\allowdisplaybreaks
\begin{align*}
    &\bigg\lVert T \bigg( R_{v,k}, \bigg( \phi_k (\sigma_v \cdot) W_k (\sigma_v \cdot) - \int_{\mathbb{T}^2} \phi_k ( x) W_k (x) \dx \bigg) \bigg) \bigg\rVert_{L^p (B^{-1}_{1,\infty})} \\
    &\lesssim \bigg\lVert \phi_k (\sigma_v \cdot) W_k (\sigma_v \cdot) - \int_{\mathbb{T}^2} \phi_k ( x) W_k (x) \dx \bigg\rVert_{B^{-1}_{1,\infty}} \lVert R_{v,k} \rVert_{L^p (L^\infty)} .
\end{align*}
\endgroup 
Finally Lemmas~\ref{lemma:essential-besov} and \ref{resonance} lead to 
\begingroup
\allowdisplaybreaks
\begin{align*}
    &\bigg\lVert R \bigg( \bigg( \phi_k (\sigma_v \cdot) W_k (\sigma_v \cdot) - \int_{\mathbb{T}^2} \phi_k ( x) W_k (x) \dx \bigg), R_{v,k} \bigg) \bigg\rVert_{L^p (B^{-1}_{1,\infty})} \\
    &\lesssim \bigg\lVert R \bigg( \bigg( \phi_k (\sigma_v \cdot) W_k (\sigma_v \cdot) - \int_{\mathbb{T}^2} \phi_k ( x) W_k (x) \dx \bigg), R_{v,k} \bigg) \bigg\rVert_{L^p (B^{s}_{1,\infty})} \\
    &\lesssim \bigg\lVert \phi_k (\sigma_v \cdot) W_k (\sigma_v \cdot) - \int_{\mathbb{T}^2} \phi_k ( x) W_k (x) \dx \bigg\rVert_{B^{-1}_{1,\infty}} \lVert R_{v,k} \rVert_{L^p (B^{1+s}_{\infty,\infty})}
\end{align*}
\endgroup 
where $0<s\ll 1$ can be chosen arbitrary.

To conclude we proceed similar to the original proof. Due Lemmas~\ref{lemma:hid} and \ref{lemma:essential-besov}, and estimate \eqref{eq:prod1} 
\begingroup
\allowdisplaybreaks
\begin{align*}
    &\bigg\lVert \phi_k (\sigma_v \cdot) W_k (\sigma_v \cdot) - \int_{\mathbb{T}^2} \phi_k ( x) W_k (x) \dx \bigg\rVert_{B^{-1}_{1,\infty}} \\
    &= \bigg\lVert \nabla_h \cdot \bigg[ \mathcal{R}_h\bigg( \phi_k (\sigma_v \cdot) W_k (\sigma_v \cdot) - \int_{\mathbb{T}^2} \phi_k ( x) W_k (x) \dx \bigg)\bigg]\bigg\rVert_{B^{-1}_{1,\infty}} \\
    &\lesssim \bigg\lVert \mathcal{R}_h\bigg( \phi_k (\sigma_v \cdot) W_k (\sigma_v \cdot) - \int_{\mathbb{T}^2} \phi_k ( x) W_k (x) \dx \bigg)\bigg\rVert_{B^{0}_{1,\infty}} \\
    &\lesssim \bigg\lVert \mathcal{R}_h\bigg( \phi_k (\sigma_v \cdot) W_k (\sigma_v \cdot) - \int_{\mathbb{T}^2} \phi_k ( x) W_k (x) \dx \bigg)\bigg\rVert_{L^1} \\ 
    &\lesssim \sigma_v^{-1} \bigg\lVert \phi_k W_k  - \int_{\mathbb{T}^2} \phi_k ( x) W_k (x) \dx \bigg\rVert_{L^1} \lesssim \sigma_v^{-1} . 
\end{align*} 
\endgroup
This finishes the proof of Lemma~\ref{lemma:prod}.
\end{proof}

\section{Other estimates} \label{holdersection}
\subsection{Improved H\"older inequality} 

Let us recall the following estimate from \cite{modena}.

\begin{lemma}[Improved H\"older inequality] \label{holderlemma}
For any $\sigma \in \mathbb{N}$, $1 \leq p \leq \infty$ and all functions $f\in C^1(\mathbb{T}^d)$, $g\in L^p(\mathbb{T}^d)$ it holds that
\begin{equation}
    \Big\lvert \lVert f (\cdot ) g (\sigma \cdot ) \rVert_{L^p} - \lVert f \rVert_{L^p} \lVert g \rVert_{L^p} \Big\rvert \lesssim \sigma^{-1/p} \lVert f \rVert_{C^1} \lVert g \rVert_{L^p}.
\end{equation}
\end{lemma}
The proof can be found in \cite[Lemma 2.1]{modena}.

\subsection{Oscillatory paraproduct estimate} \label{subsec:ap-other-osc} 

We are going to prove a version of Lemma~\ref{holderlemma} in the case of Besov spaces. 

\begin{lemma} \label{oscillparaproduct} 
For any $\sigma\in \mathbb{N}$, $1\leq p,q\leq \infty$, $s\in\mathbb{R}$, $0<\epsilon\leq 1$ and all functions $f \in L^p (\mathbb{T}^3)$, $g \in B^{s+\epsilon}_{p,q} (\mathbb{T}^3) \cap B^{s+1+\epsilon}_{\infty,q} (\mathbb{T}^3)$ it holds that
\begin{equation}
    \lVert T (f_\sigma, g) \rVert_{B^{s}_{p,q}} \lesssim \lVert f \rVert_{L^p} \lVert g \rVert_{B^{s+\epsilon}_{p,q}} + \sigma^{-1/p} \lVert f \rVert_{L^p} \lVert g \rVert_{B^{s+1+\epsilon}_{\infty,q}},
\end{equation}
where we define $f_\sigma (x) \coloneqq f(\sigma x)$.
\end{lemma}

\begin{proof} 
We first recall the low-frequency cut-off operator $S_j$ from \cite{bahouri}:
\begin{equation*}
    S_j f \coloneqq \sum_{i=-1}^{j-1} \Delta_i f.
\end{equation*}
Hence we may write $T (f_\sigma, g) = \sum_{j=-1}^\infty S_{j-1} f_\sigma \Delta_j g = \sum_{j=1}^\infty S_{j-1} f_\sigma \Delta_j g$, where the latter equation follows from the fact that $S_{-2}f_\sigma=S_{-1}f_\sigma=0$. In order to estimate the Besov norm of $T (f_\sigma, g)$, we will use \cite[Lemma~2.69]{bahouri}. To be able to use this lemma we need to show that $\supp \big(\mathcal{F} (S_{j-1} f_\sigma \Delta_j g) \big)$ lies in $2^j \mathcal{C}$ for any $j\in \mathbb{N}_0$, where $\mathcal{C}$ is a fixed annulus. 

In order to show this, let $j\in\mathbb{N}$ and $i\in \{-1,...,j-2\}$. By construction of the dyadic partition of unity, there exist radii $0<r<r_0<R$ such that the support of $\rho_{-1}$ is contained in the ball with radius $r_0$, the support of $\rho_0$ is contained in the annulus with inner radius $r$ and outer radius $R$, and $r_0<2r< R< 4r$. Next we observe
\begin{equation*}
    \supp \big( \mathcal{F} (\Delta_i f_\sigma \Delta_j g) \big) = \supp \big( (\rho_i \widehat{f}_\sigma ) * (\rho_j \widehat{g}) \big).
\end{equation*}
For $i=-1$ this yields that $\supp \big( \mathcal{F} (\Delta_i f_\sigma \Delta_j g) \big)$ is contained in the annulus with inner radius $2^j(r-2^{-j}r_0)$ and outer radius $2^j(R+2^{-j}r_0)$, which is in turn a subset of the annulus with inner and outer radii $2^j(r-\frac{1}{2}r_0)$ and $2^j(R+\frac{1}{2}r_0)$. Note that the inner radius is positive due to $2r>r_0$. Similarly for $i\geq 0$ we obtain that $\supp \big( \mathcal{F} (\Delta_i f_\sigma \Delta_j g) \big)$ is contained in the annulus with inner radius $2^j(r-2^{i-j}R)$ and outer radius $2^j(R+2^{i-j}R)$, which is in turn subset of the annulus with inner and outer radii $2^j(r-\frac{1}{4}R)$ and $2^j(R+\frac{1}{4}R)$. Again note that $4r>R$ implies that the inner radius is positive. Hence there exists an annulus $\mathcal{C}$ such that 
\begin{equation*}
    \supp \big(\mathcal{F} (S_{j-1} f_\sigma \Delta_j g) \big) \subset 2^j \mathcal{C}.
\end{equation*}
Hence we may apply Lemma 2.69 from \cite{bahouri} to conclude that 
\begin{equation} \label{eq:proof-5}
    \lVert T (f_\sigma, g) \rVert_{B^{s}_{p,q}} \lesssim \Big\lVert \big(2^{js} \lVert S_{j-1} f_\sigma \Delta_j g \rVert_{L^p}\big)_{j\in \mathbb{N}} \Big\rVert_{l^q(\mathbb{N})}.
\end{equation}

In order to estimate $\lVert S_{j-1} f_\sigma \Delta_j g \rVert_{L^p}$ we define $\phi_i\coloneqq \mathcal{F}^{-1} \rho_i$. Hence we have $\Delta_i f = \phi_i * f$. A direct computation shows $\lVert \phi_i\rVert_{L^1} = \lVert \phi_0 \rVert_{L^1}$ for any $i\in \mathbb{N}$. Hence we obtain for $1\leq p <\infty$ and for any $i,j$ by Minkowski's integral inequality and Lemma~\ref{holderlemma}
\begingroup
\allowdisplaybreaks
\begin{align*}
    \lVert \Delta_{i} f_\sigma \Delta_j g \rVert_{L^p} &= \bigg( \int_{\mathbb{R}^n} \bigg\lvert \Delta_j g (x)  \bigg( \int_{\mathbb{R}^n } \phi_i (y) f_\sigma (x - y) \dy \bigg) \bigg\rvert^p \dx \bigg)^{1/p} \\
    &\leq \int_{\mathbb{R}^n} \bigg( \int_{\mathbb{R}^n} \lvert f_\sigma (x-y) \phi_i (y) \Delta_j g (x) \rvert^p \dx \bigg)^{1/p} \dy \\
    &= \int_{\mathbb{R}^n} \lvert \phi_i (y) \rvert \lVert f(\cdot \sigma - \sigma y) \Delta_j g ( \cdot ) \rVert_{L^p} \dy \\
    &\lesssim \int_{\mathbb{R}^n} \lvert \phi_i (y) \rvert \bigg(  \lVert f (\cdot - \sigma y) \rVert_{L^p} \lVert \Delta_j g  \rVert_{L^p} + \sigma^{-1/p} \lVert \Delta_j g \rVert_{C^1} \lVert f (\cdot - \sigma y ) \rVert_{L^p} \bigg) \dy \\
    &\leq \lVert \phi_0 \rVert_{L^1} \Big(\lVert f \rVert_{L^p} \lVert \Delta_j g  \rVert_{L^p} + \sigma^{-1/p} \lVert \Delta_j g \rVert_{C^1} \lVert f \rVert_{L^p}\Big) \\
    &\lesssim \lVert f \rVert_{L^p} \lVert \Delta_j g  \rVert_{L^p} + \sigma^{-1/p} \lVert \Delta_j g \rVert_{C^1} \lVert f \rVert_{L^p}.
\end{align*}
\endgroup 
For the case $p=\infty$ one obtains the same result, the details are left to the reader. Hence 
\begingroup
\allowdisplaybreaks
\begin{align}
    \lVert S_{j-1} f_\sigma \Delta_j g \rVert_{L^p} &\leq \sum_{i=-1}^{j-2} \lVert \Delta_{i} f_\sigma \Delta_j g \rVert_{L^p} \notag\\
    &\lesssim j \Big(\lVert f \rVert_{L^p} \lVert \Delta_j g  \rVert_{L^p} + \sigma^{-1/p} \lVert \Delta_j g \rVert_{C^1} \lVert f \rVert_{L^p} \Big) \label{eq:proof-6}
\end{align}
\endgroup 
for any $j\in \mathbb{N}$. Combining \eqref{eq:proof-5} and \eqref{eq:proof-6} we get for any $1\leq q<\infty$
\begin{align*} 
    &\lVert T (f_\sigma, g) \rVert_{B^{s}_{p,q}} \\
    &\lesssim \bigg(\sum_{j=1}^\infty 2^{jsq} \lVert S_{j-1}f_\sigma \Delta_j g\rVert_{L^p}^q \bigg)^{1/q} \\
    &\lesssim \bigg(\sum_{j=1}^\infty 2^{jsq} j^q \Big(\lVert f \rVert_{L^p} \lVert \Delta_j g  \rVert_{L^p} + \sigma^{-1/p} \lVert \Delta_j g \rVert_{C^1} \lVert f \rVert_{L^p} \Big)^q \bigg)^{1/q} \\
    &\lesssim \lVert f \rVert_{L^p} \bigg(\sum_{j=1}^\infty 2^{j(s+\epsilon)q} j^q 2^{-j\epsilon q} \lVert \Delta_j g \rVert_{L^p}^q \bigg)^{1/q} + \sigma^{-1/p} \lVert f \rVert_{L^p} \bigg(\sum_{j=1}^\infty 2^{j(s+\epsilon)q} j^q 2^{-j\epsilon q} \lVert \Delta_j g \rVert_{C^1}^q \bigg)^{1/q} \\
    &\lesssim \lVert f \rVert_{L^p} \lVert g \rVert_{B_{p,q}^{s+\epsilon}} + \sigma^{-1/p} \lVert f \rVert_{L^p} \lVert g \rVert_{B_{\infty,q}^{s+1+\epsilon}},
\end{align*}
where we have used that $\lVert \Delta_j g \rVert_{C^1} \leq \lVert \nabla \Delta_j g \rVert_{L^\infty} + \lVert \Delta_j g \rVert_{L^\infty}$, $\nabla \Delta_j g= \Delta_j \nabla g$ and Lemma~\ref{lemma:essential-besov}. For the case $q=\infty$ we proceed analogously. 
\end{proof}

Lemma~\ref{oscillparaproduct} can be used to prove the following slightly weaker version of Lemma~\ref{lemma:prod}.

\begin{lemma} \label{lemma:prod-alternative}
    For all $1\leq p\leq \infty$, $\delta>0$ and $k\in \{1,2\}$ we have
    \begin{equation}
        \bigg\lVert \bigg( \phi_k (\sigma_v \cdot) W_k (\sigma_v \cdot) - \int_{\mathbb{T}^2} \phi_k ( x) W_k (x) \dx \bigg) R_{v,k} \bigg\rVert_{L^p (B^{-1-\delta}_{1,\infty})} \lesssim \sigma_v^{-1} + \lVert R_{v} \rVert_{L^p (B^{-1}_{1,\infty})}.
    \end{equation}
\end{lemma}

\begin{proof}
Compared to the proof presented in section~\ref{subsec:ap-littlewood-alternative} the only difference is how we handle the paraproduct in \eqref{eq:paraproduct-alternative}. We use Lemma~\ref{oscillparaproduct} and estimate \eqref{eq:prod1} to find
\begingroup
\allowdisplaybreaks
\begin{align*}
    &\bigg\lVert T \bigg( \bigg( \phi_k (\sigma_v \cdot) W_k (\sigma_v \cdot) - \int_{\mathbb{T}^2} \phi_k ( x) W_k (x) \dx \bigg),  R_{v,k} \bigg) \bigg\rVert_{L^p (B^{-1-\delta}_{1,\infty})} \\
    &\lesssim \bigg\lVert \phi_k W_k - \int_{\mathbb{T}^2} \phi_k ( x) W_k (x) \dx \bigg\rVert_{L^1} \lVert R_{v} \rVert_{L^p (B^{-1}_{1,\infty})} \\
    &\qquad + \sigma_v^{-1} \bigg\lVert \phi_k W_k - \int_{\mathbb{T}^2} \phi_k (x) W_k (x) \dx \bigg\rVert_{L^1} \lVert R_{v} \rVert_{L^p (B^{0}_{\infty,\infty})}\\ 
    &\lesssim \lVert R_{v} \rVert_{L^p (B^{-1}_{1,\infty})} + C_{R_v} \sigma_v^{-1}  .
\end{align*} 
\endgroup 
\end{proof}

With Lemma~\ref{lemma:prod-alternative} one can show 
\begin{equation*}
    \lVert w_{p,v} u_{p,v} \rVert_{L^1(B^{-1-\delta}_{1,\infty})} \lesssim \lVert R_v \rVert_{L^1(B^{-1}_{1,\infty})}
\end{equation*}
which is slightly weaker than \eqref{eq:est-wpvupv}. This finally allows to prove a version of Theorem~\ref{mainresult} where the regularity parameter is $s=1+$. 

\end{appendices}
\bibliographystyle{abbrv}
{\footnotesize \bibliography{Nonuniqueness_primitive_equations}}

\end{document}